\documentclass[a4paper]{article}
\usepackage{euscript,amsfonts,amssymb,amsmath,amscd,amsthm}

\usepackage{color}

\usepackage[T2A]{fontenc}

\usepackage{datetime}

\usepackage{graphicx}

\usepackage{psfrag}

\newtheorem{theorem}{Theorem}[section]
\newtheorem*{theorem*}{Theorem}
\newtheorem{lemma}[theorem]{Lemma}
\newtheorem{definition}[theorem]{Definition}
\newtheorem{proposition}[theorem]{Proposition}

\newcommand{\plim}{\mathop{\varprojlim}\limits}
\newcommand{\la}{\lambda}
\newcommand\gt{\mathbb{GT}}

\numberwithin{equation}{section}

\newcommand{\Co}{{\mathcal C}}

\begin{document}

\title{Markov processes of infinitely many nonintersecting random walks}

\author{Alexei Borodin\thanks{Department of Mathematics, Massachusetts Institute of Technology,
Cambridge, MA, USA;  Department of Mathematics, California Institute of Technology,  Pasadena, CA,
USA; Dobrushin Mathematics Laboratory, Institute for Information Transmission Problems of Russian
Academy of Sciences, Moscow, Russia. e-mail: borodin@math.mit.edu} \and Vadim
Gorin\thanks{Mathematical Sciences Research Institute, Berkeley, CA, USA; Dobrushin Mathematics
Laboratory, Institute for Information Transmission Problems of Russian Academy of Sciences,
Moscow, Russia. e-mail: {vadicgor@gmail.com}}}

\date{January 23, 2012}

\maketitle

\begin{abstract}
Consider an $N$-dimensional Markov chain obtained from $N$ one-dimensional random walks by Doob
$h$-transform with the $q$-Vandermonde determinant. We prove that as $N$ becomes large, these
Markov chains converge to an infinite-dimensional Feller Markov process. The dynamical correlation
functions of the limit process are determinantal with an explicit correlation kernel.

The key idea is to identify random point processes on $\mathbb Z$ with $q$-Gibbs measures on
Gelfand-Tsetlin schemes and construct Markov processes on the latter space.

Independently, we analyze the large time behavior of PushASEP with finitely many particles and
particle-dependent jump rates (it arises as a marginal of our dynamics on Gelfand-Tsetlin
schemes). The asymptotics is given by a product of a marginal of the GUE-minor process and
geometric distributions.
\end{abstract}

\begin{flushleft}

\smallskip

\end{flushleft}

\tableofcontents

\section{Introduction} Let $X(t)=(X_1(t),\dots,X_N(t))\in \mathbb{Z}^N$
be $N\ge 1$ independent rate one Poisson processes started at $X(0)=(0,1,\dots,N-1)$ and
conditioned to finish at $X(T)=Y$ while taking mutually distinct values for all $0\le t\le T$.
Thus, $\{X_i(t)\}_{i\ge 1}$ form $N$ \emph{nonintersecting paths} in $\mathbb{Z}\times[0,T]$.

A result of \cite{KOR}, based on a classical theorem of Karlin and McGregor \cite{KM}, says that
as $T\to\infty$ with $Y$ being asymptotically linear, $Y\sim T\xi$ with a collection of asymptotic
speeds $\xi=(\xi_1\le\dots\le \xi_N)\in\mathbb{R}_{>0}^N$, the process $(X(t), 0\le t\le T)$ has a
limit $(\mathcal{X}(t),t\ge 0)$, which is a homogeneous Markov process on $\mathbb{Z}^N$ with
initial condition $X(0)$ and transition probabilities over time interval $t$ given by
$$
P_t\bigl((x_1,\dots,x_N)\to (y_1,\dots,y_N)\bigr)=const\cdot
\frac{\Xi_N(y)}{\Xi_N(x)}\,\det_{i,j=1,\dots,N}\left[\frac{t^{y_i-x_j}}{(y_i-x_j)!}\right]
$$
with $\Xi_N(u)=\det\bigl[\xi_i^{u_j}\bigr]_{i,j=1}^N$.
\begin{footnote}{If some of $\xi_j$'s coincide, for the formula to make sense one needs to
perform a limit transition from the case of different speeds.}
\end{footnote}The process $\mathcal{X}(t)$ is the Doob $h$-transform of $N$ independent Poisson processes with respect to the
harmonic function $h=\Xi_N$.

The case of equal speeds $\xi_j\equiv const$ is special. In that case, the distribution of $X(t)$
for a fixed $t>0$ is known as the \emph{Charlier orthogonal polynomial ensemble}, which is the
basic discrete probabilistic model of random matrix type. If one considers its limiting behavior
as $N\to\infty$, in different parts of the state space via suitable scaling limits one uncovers
discrete sine, sine, and Airy determinantal point processes which play a fundamental role in
Random Matrix Theory, cf. \cite{J_annals}, \cite{BO-RSK}.

The procedure of passing to a limit by increasing the number of particles and scaling the space
appropriately in models of random matrix type is very well developed. In many cases such a limit
can also be successfully performed for joint distributions at finitely many time moments, if the
original model undergoes a Markov dynamics. The most common approach is to control the limiting
behavior of local \emph{correlation functions} for the model at hand.

It is much less clear what happens to the Markovian structure of the dynamics under such limit
transitions. While there are no \emph{a priori} reasons for the Markov property to be preserved,
it is a common belief that in many cases it survives the limit. However, providing a proof for
such a statement seems to be difficult.

One goal of the present paper is to prove that if $\xi_j$'s form a geometric progression with
ratio $q^{-1}>1$, the limit $\mathfrak{X}(t)$ of $\mathcal{X}(t)$ (and other similar processes,
see below) as $N\to\infty$ can be viewed as a Feller Markov process  on all point configurations in
$\mathbb{Z}_{\ge 0}$ started from the densely packed initial condition. Note that no scaling is
needed in the limiting procedure. As all Feller processes, our limiting process can be started
from any initial condition, and it has a modification with c\`adl\`ag sample paths.

We also show that the dynamical correlation functions of $\mathfrak{X}(t)$ started from the packed
initial condition are determinantal with an explicit correlation kernel, and they are the limits
of the corresponding correlation functions for $\mathcal{X}(t)$.

The process $\mathfrak{X}(t)$ can be interpreted as a restriction of a `local' Markov process on
the infinite Gelfand-Tsetlin schemes that falls into a class of such processes constructed in
\cite{BF}. This interpretation implies, in particular, that the distribution of the leftmost
particle of $\mathfrak{X}(t)$ coincides with the asymptotic displacement of the $m$th particle, as
$m\to\infty$, in the totally asymmetric simple exclusion process (TASEP) with jump rates depending
on particles, particle $j$ has rate $\xi_j$, and step initial condition.
\begin{footnote}{The step initial condition means that we place the particle $j$ at location $-j$
at $t=0$.}\end{footnote} Inspired by this connection, we derive large time asymptotics for TASEP,
and more general PushASEP of \cite{BF-Push}, with finitely many particles that have arbitrary
jumps rates.

In the situation when the correlation functions have a determinantal structure, the question of
Markovianity for systems with infinite many particles has been previously addressed in \cite{S},
\cite{KT09}, \cite{Os} for the sine process, in \cite{KT10} for the Airy process, in
\cite{Liza_Jones} for nonintersecting Bessel processes, in \cite{Olsh2} for the Whittaker process,
and in \cite{BO} for a process arising in harmonic analysis on the infinite-dimensional unitary
group.

Our work seems to be the first one that deals with the discrete limiting state space.
Structurally, we adopt the approach of \cite{BO}: We use the fact (proved earlier by one of us,
see \cite[Theorem 1.1]{Gor}) that infinite point configurations in $\Bbb Z$ with finitely many
particles to the left of the origin can be identified with ergodic $q$-Gibbs measures on infinite
Gelfand-Tsetlin schemes. We further show that the Markov processes $\mathcal{X}(t)$ for different
$N$'s are consistent with respect to natural projections from the $N$-particle space to the
$(N-1)$-particle one; the projections are uniquely determined by the $q$-Gibbs property. Together
with certain (nontrivial) estimates, this leads to the existence of the limiting Feller Markov
process. One interesting feature of our construction is that we need to add Gelfand-Tsetlin
schemes with infinite entries in order to make the space of the ergodic $q$-Gibbs measures locally
compact.

It is worth noting what happens in the limit $q\to 1$. The space of ergodic 1-Gibbs measures has
countably many continuous parameters (as opposed to discrete ones for $q<1$), and it is naturally
isomorphic to the space of indecomposable characters of the infinite-dimensional unitary group,
and to the space of totally positive doubly infinite Toeplitz matrices {(see e.g.
\cite{Vk_unitary}, \cite{OkOlsh}, \cite{Olsh} and references therein.)} The $q\to 1$ limit of our
Feller Markov process ends up being a \emph{deterministic} (in fact, linear) flow on this space.

It is plausible that our results on the existence of limit Markov processes can be extended to the
case of arbitrary $\{\xi_j\}_{j\ge1}$ which grow sufficiently fast, but at the moment our proofs
of a number of estimates rely on the fact that $\{\xi_j\}_{j\ge 1}$ form
a geometric progression.

We now proceed to a more detailed description of our work.

\subsection{Extended Gelfand-Tsetlin schemes and $q$-Gibbs measures}
Following \cite{Wey}, for $N\ge 1$ we define a {\it signature\/} of length $N$ as an $N$-tuple of
nonincreasing integers $\la=(\la_1\ge\dots\ge \la_N)$, and we denote by $\gt_N$ the set of all
such signatures. For $\la\in\gt_{N}$ and $\nu\in \gt_{N+1}$, we say that $\la\prec\nu$ if
$\nu_{j+1}\le \la_j\le\nu_j$ for all meaningful values of indices. We agree that $\gt_0$ is a
singleton $\varnothing$, and $\varnothing\prec\nu$ for any $\nu\in \gt_1$.

A \emph{Gelfand-Tsetlin scheme} of order $M\in\{0,1,2,\dots\}\cup\{\infty\}$ is a length $M$
sequence
$$
\la^{(0)}\prec\la^{(1)}\prec\la^{(2)}\prec\dots,\qquad \la^{(j)}\in \gt_j.
$$
Equivalently, such a sequence can be viewed as an array of numbers $\bigl\{\la^{(j)}_i\bigr\}$
satisfying the inequalities $\la^{(j+1)}_{i+1}\le \la_i^{(j)}\le \la^{(j+1)}_i$. An interpretation
of Gelfand-Tsetlin schemes in terms of lozenge tilings or stepped surfaces can be found in
\cite[Section 1.1]{BF}.

Define an \emph{extended} signature of length $N$ as an $N$-tuple $\la=(\la_1\ge\dots\ge \la_N)$
with $\la_j\in \overline{\mathbb{Z}}=\mathbb{Z}\cup \{+\infty\}$\footnote{We assume that $+\infty
> k$ for every integer $k$.}, and define an \emph{extended} Gelfand-Tsetlin scheme as a sequence of extended signatures
$\la^{(0)}\prec\la^{(1)}\prec\la^{(2)}\prec\dots$, $\la^{(j)}\in\overline{\gt}_j$, with the
condition that the number of infinite coordinates $m_j$ of $\la^{(j)}$ has the property that if
$m_j>0$ then $m_{j-1}=m_j-1$. In other words, the number of finite coordinates grows from level to
level until the first infinite coordinate appears, on all further levels the number of finite
coordinates  remains the same.

Fix $q\in(0,1)$. We say that a probability measure on (extended) Gelfand-Tsetlin schemes is
\emph{$q$-Gibbs} if for any $N\ge 2$, the conditional distribution of $\la^{(0)}\prec
\la^{(1)}\prec\dots\prec\la^{(N-1)}$ given $\la^{(N)}$ has weights that are proportional to
$q^{|\la^{(1)}|+|\la^{(2)}|+\dots+|\la^{(N-1)}|}$, where $|\la|$ denotes the sum of all finite
coordinates of $\la$.

There are at least two motivations for our interest in $q$-Gibbs measures. On one hand, if one
views Gelfand-Tsetlin schemes as stepped surfaces, the conditional distribution above has weights
proportional to $q^{volume}$, where the exponent is the volume underneath the surface, and measures
of such type have various interesting features, cf.\ {\cite{V_Vienna}, \cite{CK},
\cite{OR}, \cite{BGR}}. On the other hand, the notion of $1$-Gibbsianness naturally arises in the
representation theory of the infinite-dimensional unitary group, and we hope that $q$-Gibbs
measures will arise in suitably defined representations of inductive limits of quantum groups.

In \cite[Theorem 1.1]{Gor}, one of the authors obtained a description of the convex set of
$q$-Gibbs measures on (non-extended infinite) Gelfand-Tsetlin schemes. The result says that the
space of such $q$-Gibbs measures is isomorphic to the space of Borel probability measures on the
space
$$
\mathcal{N}=\{\nu_1\le \nu_2\le \dots,\  \nu_i\in\mathbb{Z}\}
$$
equipped with the topology of coordinate-wise convergence. Moreover, for any \emph{extreme}
$q$-Gibbs measure $\mu$ (that corresponds to the delta measure at a point $\nu\in\mathcal{N}$),
the value of $\nu_k$, $k\ge 1$, is the almost sure limit of $\la^{(N)}_{N-k+1}$ as $N\to\infty$, with
$\{\la^{(j)}\}_{j\ge 0}$ distributed according to $\mu$.

In this paper we generalize this statement and prove that the space of $q$-Gibbs measures on
extended infinite Gelfand-Tsetlin schemes is isomorphic to the space of Borel probability measures
on the space
$$
\overline{\mathcal{N}}=\{\nu_1\le \nu_2\le \dots,\  \nu_i\in\overline{\mathbb{Z}}\}
$$
equipped with the topology of coordinate-wise convergence. Observe that $\overline{\mathcal{N}}$
is a locally compact topological space, and it should be viewed as a completion of $\mathcal{N}$.

Under the one-to-one correspondence
\begin{equation}\label{eq:onetoone}
\{\nu_1\le \nu_2\le \nu_3\le  \dots\} \longleftrightarrow \{\nu_1<\nu_2+1<\nu_3+2<\dots\},
\end{equation}
$\mathcal{N}$ turns into the set of all infinite subsets of $\mathbb{Z}$ that are bounded from
below, and $\overline{\mathcal{N}}$ can be viewed as the set of all (finite and infinite) subsets
of $\mathbb{Z}$ bounded from below.

\subsection{Consistent $N$-dimensional random walks}
\label{Subsection_consistent_in_intro}
 Let $g(x)$ be a finite product of elementary factors
of the form
$$
(1-\alpha x^{-1})^{-1},\ 0<\alpha<1,\qquad (1+\beta x^{\pm 1}),\ \beta>0; \qquad \exp(\gamma
x^{\pm 1}),\ \gamma>0,
$$
and $g(x)=\sum_{k\in\mathbb{Z}} g_k x^k$ be its Laurent series in
$\{x\in\mathbb{C}:1<|x|<\infty\}$. This means that ${\{g_k\}}_{k\in\mathbb{Z}}$ is, up to a
constant, a convolution of geometric, Bernoulli, and Poisson distributions.

For any such $g(x)$, define a transition probability on $\gt_N$ by
\begin{equation}
\label{eq_transition_prob_introduction} P_N(\la\to\mu;g(x))=
const\cdot\det_{i,j=1,\dots,N}[g_{\mu_i-i-\la_j+j}] \,\prod_{1\le i<j\le
N}\frac{q^{i-\la_i}-q^{j-\la_j}}{q^{i-\mu_i}-q^{j-\mu_j}}\,,
\end{equation}
and extend it to a transition probability $\overline P_N(\la\to\mu;g(x))$ on $\overline{\gt}_N$ by
applying $P_k(\,\cdot\,;g(x))$ with appropriate $k<N$ to all $k$ finite coordinates of extended
signatures.

We prove that if one starts with a $q$-Gibbs measure on extended Gelfand-Tsetlin schemes and
applies $\overline P_N(\,\cdot\,,g(x))$ to its projections on $\overline{\gt_N}$, $N\ge 1$, then
the resulting distributions are also projections of a new $q$-Gibbs measure. This gives rise to a
Markov transition kernel $\overline P_\infty(g(x))$ on $\overline{\mathcal{N}}$ or, via
\eqref{eq:onetoone}, on point configurations in $\mathbb{Z}$ that do not have $-\infty$ as an
accumulation point.

\subsection{Correlation functions}
\label{Subsection_corr_func_in_intro}
 Let $g_0(x),g_1(x),g_2(x),\dots$ be a sequence of functions as above, and let
$\mathcal{Z}(t)$, $t=0,1,2,\dots$, be the discrete time Markov process on $\overline{\mathcal{N}}$
with initial condition given by the delta measure at $\mathbf{0}=(0,0,\dots)$ and transition
probabilities $\overline P_\infty(g_0(x)), \overline P_\infty(g_1(x)),\dots$. One easily sees that
the process is in fact supported by $\mathcal{N}$.

Define the $n$th correlation function $\rho_n$ of $\mathcal{Z}(t)$ at $(x_1,t_1),\dots,(x_n,t_n)$
as the probability that $Z({t_j})$, viewed via \eqref{eq:onetoone} as a subset of $\mathbb{Z}$,
contains $x_j$ for every $j=1,\dots,n$. We prove that these correlation functions are
determinantal: For any $n\ge 1$, $t_j\in\mathbb{Z}_{\ge 0}$, $x_j\in\mathbb{Z}$, we have
$$
\rho_n(x_1,t_1;x_2,t_2;\dots;x_n,t_n)=\det_{i,j=1,\dots,n}[K(x_i,t_i;x_j,t_j)],
$$
where
\begin{multline*}
 K(x_1,t_1;x_2,t_2)=-\frac{1}{2\pi i} \oint_{\Co} \frac{dw}{w^{x_1-x_2+1}}
 {\prod_{t=t_2}^{t_1-1} {g_t(w)}} {\mathbf
 1}_{t_1>t_2}\\
 +\frac{1}{(2\pi i)^2} \oint_{\Co} dw \oint_{\Co'} dz \frac{\prod_{t=0}^{t_1-1}
 g_t(w)}{\prod_{t=0}^{t_2-1} {g_t(z)}} \prod_{j\ge 0}\frac{1-wq^j}{1-zq^j} \frac{z^{x_2}}{w^{x_1+1}}
 \frac{1}{w-z},
\end{multline*}
$\Co$ is a positively oriented contour that includes only the pole $0$ of the integrand, and
$\Co'$ goes from $+i\infty$ to $-i\infty$ between $\Co$ and point $1$.

\subsection{The Feller property}
\label{subsection_feller_intro}
 A Markov transition kernel on a locally compact state space is said
to have the \emph{Feller property} if its natural action on functions on the state space preserves the
space $C_0$ of continuous functions that converge to zero at infinity, cf.\ \cite{EK}.

We prove that for any function $g(x)$ as above, the kernel $\overline P_\infty(g(x))$ on
$\overline{\mathcal{N}}$ enjoys the Feller property.

We also prove that for any $\gamma^+,\gamma^-\ge 0$ and any probability measure on
$\overline{\mathcal{N}}$, the continuous time Markov process on $\overline{\mathcal{N}}$ with
initial condition $\mu$ and transition probabilities
$$
\overline P_\infty\bigl(\exp(t(\gamma_+x+\gamma^-/x))\bigr),\qquad t\ge 0,
$$
is Feller. In addition to the Feller property of the transition kernels, this also requires their
strong continuity in $t$ when they are viewed as an operator semigroup in
$C_0(\overline{\mathcal{N}})$.

Note that such a statement would have made no sense for non-extended Gelfand-Tsetlin schemes
because $\mathcal{N}$ is not locally compact. This is the place where the extended theory becomes
a necessity.

\subsection{Asymptotics of PushASEP}
\label{Subsection_pushASEP_intro}

Fix parameters $\zeta_1,\dots,\zeta_N>0$, $a, b\ge 0$, and assume that at least one of the numbers
$a$ and $b$ does not vanish.  Consider $N$ particles in $\mathbb Z$ located at different sites and
enumerated from left to right. The particle number $n$, $1\le n\le N$, has two exponential clocks
- the ``right clock'' of rate $a\zeta_n$ and the ``left clock'' of rate $b/\zeta_n$. When the
right clock of particle number $n$ rings, one checks whether the position to the right of the
particle is empty. If yes, then the particle jumps to the right by 1, otherwise it stays put. When
the left clock of particle number $n$ rings, it jumps to the left by 1 and pushes the (maybe
empty) block of particles sitting next to it. All $2N$ clocks are independent.

This interacting particle system was introduced in \cite{BF-Push} under the name of
\emph{PushASEP}. It interpolates between well-known TASEP and long-range TASEP, cf.
\cite{Spitzer}.

 Results of \cite[Section 2.7]{BF} imply that if
$\zeta_1=q^{1-N},\zeta_2=q^{2-N},\dots,\zeta_N=1$, and the PushASEP is started from the initial
configuration $(1-N,2-N,\dots,-1,0)$, then at time $t$ it describes the behavior of the
coordinates $(\la^{(N)}_N+1-N,\la^{(N-1)}_{N-1}+2-N,\dots,\la^{(1)}_1)$ of the random infinite
Gelfand-Tsetlin scheme distributed according to the $q$-Gibbs measure corresponding to the
distribution at the same time $t$ of the Markov process on $\overline{\mathcal N}$ with
transitional probabilities $\{\overline P_\infty\bigl(\exp(t(ax+b/x))\bigr)\}_{t\ge 0}$ and the
delta-measure at $\mathbf{0}$ as the initial condition.

We prove, for the $N$-particle PushASEP with any jump rates and any deterministic initial
condition, and independently of the rest of the paper, that at large times the PushASEP particles
demonstrate the following asymptotic behavior: In each asymptotic cluster, particles with the
lowest values of $\zeta_k$ fluctuate on $\sqrt{t}$ scale, and the fluctuations are given by the
distribution of the smallest eigenvalues in an appropriate GUE-minor process, while faster
particles remain at $O(1)$ distances from the blocking slower ones, with geometrically distributed
distances between neighbors. The GUE-governed fluctuations and the geometric distributions are
asymptotically independent.

Note that for TASEP, which arises when $b=0$, $\zeta_1=\zeta_2=\dots=\zeta_N$, the relation of
large time asymptotics with marginals of the GUE-minor process is well known, see\ \cite{Bar}.

\subsection{Organization of the paper}
In Section \ref{sc:walks} we discuss $N$-dimensional random walks with transition probabilities
$P_N(\la\to\mu;g(x))$. In Section \ref{Section_Commutation_relations} we prove that these random
walks are consistent for different values of $N$. In Section \ref{Section_Infinite-dimensional} we
explain how the family of consistent transition probabilities define the transition kernel
$\overline {P}_{\infty}(\la\to\mu;g(x))$ on $\overline{\mathcal N}$ and describe basic properties
of this kernel. The properties of the corresponding Markov processes are studied in Sections
\ref{Section_boundary_process_distributions} and \ref{Section_feller}.

The main results of Section
\ref{Section_boundary_process_distributions} are Theorem
\ref{Theorem_convergence_of_joint_distributions}, where we prove that finite-dimensional
distributions of the constructed processes are $N\to\infty$ limits of those for the $N$-dimensinal
processes; and Theorem \ref{th:corrfunc}, where we prove the corresponding
result for the correlation functions. In Section \ref{Subsection_first_coordinate} we describe
the connection between processes with transition probabilities $\overline {P}_{\infty}(\la\to\mu;g(x))$ and
stochastic dynamics of \cite{BF}.

Section \ref{Section_feller} is devoted to the study of extended Gelfand--Tsetlin schemes. In
Theorem \ref{Theorem_extended_boundary_1} we prove that the space of $q$-Gibbs measures on
extended infinite Gelfand-Tsetlin schemes is isomorphic to the space of Borel probability measures
on the space $\overline{\mathcal{N}}$. Theorems \ref{theorem_Feller_kernel} and
\ref{theorem_Feller_process} contain the proofs of the results on Feller properties mentioned above.
This is the most technical part of the paper.

Finally, in Section \ref{Section_pushASEP} we prove Theorem \ref{th:pushasep} describing the
asymptotic behavior of PushASEP with finitely many particles.

\section{$N$--dimensional random walks}
\label{sc:walks}

In this section we introduce a family of Markov chains on the set of $N$-point configurations in
$\mathbb Z$. In a variety of situations they can viewed as collections of $N$ independent
identically distributed random walks conditioned not to intersect and to have prescribed
asymptotic speeds at large times.

The state space of our processes is the set $\mathbb M_N$ of $N$-point configurations in $\mathbb
Z$:
$$
 \mathbb M_N=\{a(1)< a(2)<\dots<a(N)\mid a(i)\in\mathbb Z\}.
$$
We identify elements of $\mathbb M_N$ with weakly decreasing sequences of integers, which we call
\emph{signatures}. The set of all signatures of size $N$ is denoted by $\mathbb{GT}_N$,
$$
 \mathbb{GT}_N=\{\lambda_1\ge\dots\ge\lambda_N\mid \lambda_i\in\mathbb{Z}\}.
$$
We use the following correspondence between the elements of $\mathbb M_N$ and $\mathbb{GT}_N$:
$$
 a(1)<\dots<a(N)\longleftrightarrow a(N)-N+1\ge a(N-1)-N+2\ge\dots\ge a(2)-1\ge a(1).
$$
{We also agree that $\mathbb M_0$ ($\mathbb {GT}_0$) consists of a singleton --- the empty
configuration (the empty signature $\varnothing$).}

`Signature' is the standard term for the label of an irreducible representation of the unitary
group $U(N)$ over $\mathbb{C}$, and the letters $\mathbb{GT}$ stand for `Gelfand-Tsetlin' as in
Gelfand-Tsetlin bases of the same representations. Although the material of this paper is not
directly related to representation theory, we prefer not to change the notation of related
previous works, cf. \cite{OkOlsh, Gor}.

In studying probability measures and Markov processes on $\mathbb{GT}_N$ we extensively use
\emph{rational Schur functions}. These are Laurent polynomials
 $s_\lambda\in\mathbb{Z}[x_1^{\pm1},\dots,x_N^{\pm 1}]$ indexed by
$\lambda\in\mathbb{GT}_N$ and defined by
$$
 s_\lambda(x_1,\dots,x_N)=\frac{\det_{i,j=1,\dots,N}\left[
 x_i^{\lambda_j+N-j}\right]}{\prod_{i<j}(x_i-x_j)}.
$$

Let us introduce certain generating functions of probability measures on $\mathbb{GT}_N$
that may be viewed as analogues of characteristic functions but that are more suitable for
our purposes.

Fix a non-decreasing sequence of positive reals $\{\xi_i\}_{i=1,2,\dots}$. Let $T_N$ be
the $N$-dimensional torus
$$
T_N=\{(x_1,\dots,x_N)\in\mathbb{C}^N\mid |x_i|=\xi_i\}.
$$

Denote by $\mathcal F_N$ a class of functions on $T_N$ which can be decomposed as
\begin{equation}
\label{eq_definition_of_class}
 f(x_1,\dots,x_N)=\sum_{\lambda\in\mathbb{GT}_N} c_\lambda(f)
 \frac{s_\lambda(x_1,\dots,x_N)}{s_\lambda(\xi_1,\dots,\xi_N)}
\end{equation}
with $c_\lambda(f)\ge 0$ and $\sum_{\lambda} c_\lambda(f)=1$. Note that the latter condition is
equivalent to $f(\xi_1,\dots,\xi_N)=1$.

\begin{lemma}
 The series \eqref{eq_definition_of_class} converges uniformly on $T_N$ and its sum
 $f(x_1,\dots,x_N)$ is a real analytic function on $T_N$.
\end{lemma}
\begin{proof}
 This fact immediately follows from a simple observation
 $$
  \sup_{(x_1,\dots,x_N)\in T_N} \left|\frac{s_\lambda(x_1,\dots,x_N)}
{s_\lambda(\xi_1,\dots,\xi_N)}\right| =1,
 $$
 which, in turn, follows from the combinatorial formula for Schur functions (see e.g.\
Section I.5 of \cite{Mac}).
\end{proof}
The authors know no intristic definition of the set $\mathcal F_N$ and it would be interesting to
find one.

We note without proof that $\mathcal F_N$ is a closed subset of the Banach space of continuous
function on $T_N$, see \cite[Proposition 4.11]{Gor} for a similar fact and \cite[Section 6.2]{Gor}
for its proof which translates to our case almost literally.

Let $P$ be a probability measure on $\mathbb{GT}_N$. Its \emph{Schur generating function} is a
function ${\mathcal S}(x_1,\dots,x_N;P)\in\mathcal{F}_N$ with coefficients $c_\lambda(\mathcal S)$
defined through
$$
 c_\lambda(\mathcal S)=P(\lambda),
$$
where $P(\lambda)$ stays for the measure of the singleton $\{\lambda\}$. Let $\mathcal L_N$ be the
map sending probability measures on $\mathbb{GT}_N$ to the corresponding functions in $\mathcal
F_N$. Clearly this is an isomorphism of convex sets. {We agree that $\mathcal F_0$
contains a single function (constant $1$) which corresponds to a unique probability measure on the
singleton $\mathbb{GT}_0$.}

Our next goal is to construct a family of stochastic matrices with rows and columns
enumerated by elements of $\mathbb{GT}_N$. Let $Q$ be one such stochastic matrix. Then $Q$ acts on
probability measures on $\mathbb{GT}_N$
$$
 P\mapsto QP.
$$
We will always identify stochastic matrices with the corresponding operators and use the same
notations for them.

Let $\widetilde Q$ be a bounded linear operator in the Banach space of continuous functions on
$T_N$ such that $\widetilde Q(\mathcal F_N) \subset \mathcal F_N$. Then
 $\mathcal L_N^{-1}\widetilde Q \mathcal L_N$ is a Markovian linear operator or a
 stochastic matrix.

For a function $g(x)$ on $\bigcup_{i=1}^N\{x\in\mathbb{C}:|x|=\xi_i\}$, define an operator
$\widetilde Q^g_N$ via
$$
 \widetilde Q^g_N: f(x_1,\dots,x_N)\mapsto f(x_1,\dots,x_N)\prod_{i=1}^N \frac{g(x_i)}{g(\xi_i)},
$$
{(here we agree that $Q^g_0$ is the identity operator).} Clearly, if the
function $g(x)$ is continuous then $\widetilde Q^g_N$ is a bounded linear operator in the space of
continuous functions on $T_N$.

\begin{lemma}
 \label{Lemma_decomposition_of_multiplied_s} Suppose that $g(x)$ can be decomposed into a converging
 power series in annulus $\mathbf K=\{x\in\mathbb{C}: r<|x|< R\}$:
$$
  g(x)=\sum_{k\in\mathbb{Z}} c_k x^k,\quad x\in\mathbf K.
$$
Then for $(x_1,\dots,x_N)\in \mathbf K^N$ we have
 \begin{equation}
  \label{eq_decomposition_of_multiplied_s}
  s_{\lambda}(x_1,\dots,x_N)g(x_1)\cdots g(x_N)=\sum_{\mu\in\mathbb{GT}_N} \det_{i,j=1,\dots,N}
  \biggl[
  c_{\mu_i-i-\lambda_j+j}\biggr] s_{\mu}(x_1,\dots,x_N).
 \end{equation}
\end{lemma}
\begin{proof} Straightforward computation. One multiplies both sides of
\eqref{eq_decomposition_of_multiplied_s} by $\prod_{i<j} (x_i-x_j)$ and compares the coefficients
of the monomials, cf. \cite[Lemma 6.5]{Olsh}.
\end{proof}
{  Denote by $\mathfrak g(N,\xi)$ the set of functions
consisting of
$$
\begin{array}{rrrr}
   x;&&
   x^{-1};&\\
   (1+\beta x),&\beta>0; &
   (1+\beta x^{-1}),& \beta>0; \\
   e^{\gamma x},&\gamma>0; &
   e^{\gamma x^{-1}},& \gamma>0;\\
   (1-\alpha x)^{-1}&\quad 0<\alpha<\xi_N^{-1};& (1-\alpha x^{-1})^{-1},&  0<\alpha<\xi_1.
\end{array}
$$
Also denote $\mathfrak g(\infty,\xi)=\bigcap_N \mathfrak g(N,\xi)$. In what follows we call
$g(x)\in\mathfrak g(\infty,\xi)$ an \emph{elementary admissible function}. Let $\mathcal G(N,\xi)$
denote the set of all finite products of the functions of $\mathfrak g(N,\xi)$ and denote
$\mathcal G(\infty,\xi)=\bigcap_N \mathcal G(N,\xi)$. We call $g(x)\in\mathcal G(\infty,\xi)$ an
\emph{admissible function}. }

\begin{proposition}
 \label{Proposition_admissible_functions} If $g\in \mathcal G(N,\xi)$, then $$\widetilde
 Q_N^g(\mathcal F_N) \subset\mathcal F_N.$$
\end{proposition}
\begin{proof}
 For $g(x)\in \mathcal G(N,\xi)$ we can use Lemma \ref{Lemma_decomposition_of_multiplied_s}.
 Furthermore, it is known that in these cases all determinants in the decomposition
 \eqref{eq_decomposition_of_multiplied_s} are non-negative (see \cite{Ed}, { \cite{Vo}, \cite{Bo}, \cite{Vk_unitary}, \cite{OkOlsh}}).

 Then we proceed as follows: Take a function $f\in\mathcal F_N$, decompose it into a sum of
 Schur polynomials. Then by Lemma \ref{Lemma_decomposition_of_multiplied_s} we obtain a double sum
 for $\widetilde Q_N^g(f)$. Changing the order of summation we see that $\widetilde
 Q_N^g(f)\in\mathcal F_N$.
\end{proof}

Set
$$
P_N(g(x))= \mathcal L_N^{-1}\circ \widetilde Q^g_N\circ \mathcal L_N$$ and let $
 P_N(\lambda\to\mu ;g(x))$ be the matrix element of the corresponding stochastic matrix.

\begin{proposition}
\label{Proposition_transition_prob_explicit}
 Let
 $$
  g(x)=\sum_{k\in\mathbb{Z}} c_k x^k
 $$
 be a decomposition of $g(x)$ into a power series converging for all $x$ such that $\min_{i}
 \xi_i\le |x|\le\ \max_{i} \xi_i$, $i=1,2\dots,N$. Then
 \begin{equation}
 \label{eq_trans_prob_section2}
  P_N(\lambda\to\mu ;g(x))=\left(\prod_{i=1}^N \frac{1}{g(\xi_i)}\right)
  \det_{i,j=1,\dots,N}\biggl[c_{\mu_i-i-\lambda_j+j}\biggr]
  \frac{s_\mu(\xi_1,\dots,\xi_N)}{s_\lambda(\xi_1,\dots,\xi_N)}.
 \end{equation}
\end{proposition}
\begin{proof}
 This is an immediate consequence of Lemma \ref{Lemma_decomposition_of_multiplied_s} and
 definitions.
\end{proof}

{\bf Remark 1.}  For $g(x)=(1-\alpha x)^{-1}$ the determinants in the proposition above can be
explicitly evaluated:

\begin{equation}\det_{i,j=1,\dots,N}\biggl[c_{\mu_i-i-\lambda_j+j}\biggr]=
\begin{cases} \alpha^{\sum_{i=1}^N(\mu_i-\lambda_i)},& \text{ if }\
\mu_{i-1}\leq \lambda_i\leq \mu_i,\   1\leq i\leq N,\\
0,& \text{otherwise}.
\end{cases}
\end{equation}
(The condition $\mu_0\leq \lambda_1$ above is empty.)

For $g(x)=1+\beta x$, the evaluation takes the form
$$
 \det_{i,j=1,\dots,N}\biggl[c_{\mu_i-i-\lambda_j+j}\biggr]=\begin{cases}
 \beta^{\sum_{i=1}^N(\mu_i-\lambda_i)}, \text { if } \mu_i-\lambda_i\in \{0,1\},\   1\leq i\leq N,\\
 0,\text{ otherwise.}
 \end{cases}
$$

Similar formulas exist for $g(x)=(1-\alpha^{-1} x)^{-1}$ and $g(x)=1+\beta x^{-1}$ with $\lambda$
and $\mu$ interchanged.

\smallskip

{{\bf Remark 2.} When $\xi$ is a geometric progression, the Schur function
$s_\lambda(\xi_1,\dots,\xi_N)$ can be evaluated as follows (see e.g.~\cite[Example 3.1]{Mac})
\begin{equation}\label{eq:princ_spec}
 s_\lambda(1,\dots,q^{1-N})=q^{-((N-1)\lambda_1+(N-2)\la_2+\dots+\lambda_{N-1})}
\prod_{i<j} \frac{1-q^{\lambda_i-i-\lambda_j+j}}{1-q^{j-i}}\,.
\end{equation}
 Then the formula for the transition probability \eqref{eq_trans_prob_section2} turns into
 \eqref{eq_transition_prob_introduction}.

\smallskip

{\bf Remark 3.} If $\xi_i=q^{1-i}$ and $g(x)=(1+\beta x^{\pm 1})$, then one can formally send
$N\to\infty$ in formulas \eqref{eq_transition_prob_introduction}, \eqref{eq_trans_prob_section2}
and obtain well-defined transition probabilities; while for $g(x)=\exp(\gamma x)$ such formal
limit transition does not lead to anything meaningful.

}
\medskip

Denote by $\mathcal X_{N,g}(t)$ a discrete time homogeneous Markov chain on $\mathbb{GT}_N$ with
transition probabilities $P_N(\lambda\to\mu; g(x)) $, started from the delta-measure at the zero
signature $\mathbf{0}=(0\ge 0\ge\dots\ge 0)$.

Also let $\mathcal Y_{N,\gamma_+,\gamma_-}(t)$ be the continuous time homogeneous  Markov chain on
$\mathbb{GT}_N$ with transition probabilities $P_N\left(\lambda\to\mu; \exp(t(\gamma_+ x+\gamma_-
x^{-1}))\right) $, started from delta-measure on zero signature $\mathbf{0}$. (Clearly the
corresponding stochastic matrices form a semigroup.)

A number of such Markov chains with various $g(x)$, $\gamma_+$, $\gamma_-$, have independent
probabilistic interpretations. Let us list some of them.

\begin{itemize}

\item For $N=1$ and $\xi_1=1$, $\mathcal X_{1,1+\beta x}(t)$ and $\mathcal X_{N,1+\beta x^{-1}}(t)$
are simple Bernoulli random walks with jump probability $\beta(1+\beta)^{-1}$ and particle jumping
to the right and to the left, respectively; $\mathcal X_{1,x^{\pm 1}}(t)$ is the deterministic
shift of the particle to the right (left) by 1; $\mathcal X_{1,(1-\alpha x)^{-1}}(t)$ and
$\mathcal X_{N,(1-\alpha x^{-1})^{-1}}(t)$ are random walks with geometrical jumps; $\mathcal
Y_{1,\gamma_+,0}(t)$ is the Poisson process of intensity $\gamma_+$ and  $\mathcal
Y_{1,\gamma_+,\gamma_-}(t)$ is the homogeneous birth and death process on $\mathbb{Z}$.

\item For any $N\ge 1$ and an arbitrary sequence $\xi$, it is proved in \cite{KOR} that $\mathcal
Y_{N,\gamma_+,0}(t)$ can be viewed as $N$ independent rate $1$ Poisson processes conditioned never
to collide and to have asymptotic speeds of particles given by $\xi_i$. Similar interpretations
could be given for $\mathcal X_{N,g}(t)$.

\item If $\xi_1=\xi_2=\dots=\xi_N=1$, then $\mathcal X_{N,1+\beta x}(t)$ can be obtained as a limit
of uniformly distributed 3d Young diagrams in $a\times b\times c$ box (see
\cite{J_nonintersecting}, \cite{J_Hahn}, \cite{JN}, \cite{Gor_Hahn})  with $a=N$ and $b,c\to
\infty$ in such a way that $c/b\to\beta$.

\item The connection to exclusion processes is explained in Section \ref{Subsection_first_coordinate} below.

\end{itemize}

Proposition \ref{Proposition_transition_prob_explicit} implies that the one-dimensional
distribution of $\mathcal X_{N,g}(t)$ at a given time $t_0$ is a (possibly, two-sided) \emph{Schur
measure}, cf. \cite{B-Schur},
$$
 {\rm Prob}(\mathcal X_{N,g}(t_0)=\lambda) =\left(\prod_{i=1}^N \frac{1}{g(\xi_i)}\right)^{t_0}
  \det_{i,j=1,\dots,N}\biggl[c_{\lambda_i-i+j}^{t_0}\biggr] {s_\lambda(\xi_1,\dots,\xi_N)},
$$
where $c^{t_0}_k$ are the coefficients of the Laurent series for $\bigl(g(x)\bigr)^{t_0}$.

If we view $\mathcal X_{N,g}(t_0)$ as a point configuration in $\mathbb{Z}$, i.e.\ as an element
of $\mathbb M_N$, then we may speak about its \emph{correlation functions}:
$$
 \rho_n(x_1,\dots,x_n)={\rm Prob}(x_1\in\mathcal X_{N,g}(t_0),\dots, x_n\in \mathcal X_{N,g}(t_0)).
$$

As shown in \cite{BK}, the correlation functions of the two-sided Schur measures have a
determinantal form {(for the one-sided Schur measure this was proved earlier, see \cite{Ok_wedge})}
$$
\rho_n(x_1,\dots,x_n)=\det_{i,j=1,\dots n}[K(x_i,x_j)],
$$
with \emph{correlation kernel} $K$ given by a double contour integral
\begin{multline}
\label{eq_x10}
 K(x_1;x_2)=\frac{1}{(2\pi i)^2} \oint_{\Co} dw \oint_{\Co'} dz \left(\frac{
 g(w)}{g(z)}\right)^{t_0} \prod_{\ell=1}^{N}\frac{ (1- w/\xi_\ell)}{(1-
 z/\xi_{\ell})} \frac{z^{x_2}}{w^{x_1+1}} \frac{1}{w-z},
\end{multline}
where the (positively oriented) contour $\Co$ includes only the pole at $0$ and $\Co'$ includes
only the poles at $\xi_i$. A similar formula exists for dynamical correlation functions (see
\cite[Section 2.8]{BF} and Section \ref{Subsection_correlation_functions} below), describing
finite-dimensional distributions of $\mathcal X_{N,g}(t)$ and also for $\mathcal
Y_{1,\gamma_+,\gamma_-}(t)$.

Note that if $\xi_i$'s grow fast enough as $i\to\infty$, then one may formally pass to the limit
$N\to\infty$ in \eqref{eq_x10}. Thus it is natural to expect that there is an \emph{infinite
dimensional} Markov process which is a $N\to\infty$ limit of processes $\mathcal X_{N,g}(t)$ (or
$\mathcal Y_{N,\gamma_+,\gamma_-}(t)$). One goal of this paper is to define this limiting process
rigorously and to show that its finite-dimensional distributions are indeed given by limits of
\eqref{eq_x10}.



\section{Commutation relations}

\label{Section_Commutation_relations}

In this section we show that transition probabilities $P_N(\lambda\to\mu; g(x))$ are in a
certain way consistent for various $N$.

We start by introducing stochastic matrices with rows and columns indexed by elements of
$\mathbb{GT}_N$ and $\mathbb{GT}_{N-1}$, $N\ge 1$. In other words, we want to define transition
probabilities from $\mathbb{GT}_N$ to $\mathbb{GT}_{N-1}$. As above, it is convenient to pass from
stochastic matrices to maps between spaces of Schur generating functions of probability measures.
Thus, we want to introduce a map
$$
 \tilde P_N^{\downarrow}: \mathcal F_N\to \mathcal F_{N-1}.
$$

\begin{proposition} The specialization map
$$
 \widetilde P_N^{\downarrow}: f(x_1,\dots,x_N)\to f(x_1,\dots,x_{N-1},\xi_N)
$$
 is a bounded linear operator between appropriate spaces of continuous functions, and
$\widetilde P_N^{\downarrow}(\mathcal F_N)\subset \mathcal
 F_{N-1}$.
\end{proposition}
\begin{proof}
 The fact that this is a bounded linear operator is straightforward. Using well-known branching
rules for Schur functions (see e.g. \cite{Mac}) we see that:
$$
\widetilde
P_N^{\downarrow}\left(\frac{s_{\lambda}(x_1,\dots,x_N)}{s_\lambda(\xi_1,\dots,\xi_N)}\right)=
\sum_{\mu\prec\lambda}
\frac{s_{\mu}(x_1,\dots,x_{N-1})}{s_\mu(\xi_1,\dots,\xi_{N-1})}
 \xi_N^{|\lambda|-|\mu|} \frac{s_\mu(\xi_1,\dots,\xi_{N-1})}{s_{\lambda}({\xi_1,\dots,\xi_N})},
$$
where $\mu\prec\lambda$ means the following interlacing condition:
\begin{equation}
 \label{eq_interlacing}
 \lambda_1\ge\mu_1\ge\lambda_2\ge\dots\ge\mu_{N-1}\ge\lambda_N,
\end{equation}
and $|\lambda|=\lambda_1+\dots+\lambda_N$, $|\mu|=\mu_1+\dots+\mu_{N-1}$. Since all the
coefficients
$$
 \xi_N^{|\lambda|-|\mu|} \frac{s_\mu(\xi_1,\dots,\xi_{N-1})}{s_{\lambda}({\xi_1,\dots,\xi_N})}
$$
are positive, we immediately conclude that $P_N^{\downarrow}(\mathcal F_N)\subset \mathcal
F_{N-1}$.
\end{proof}

Let us denote by $P_N^{\downarrow}$ a stochastic matrix of transition probabilities
corresponding to $\widetilde P_N^{\downarrow}$, i.e. $P_N^{\downarrow}=\mathcal L_{N-1}^{-1}\circ
\widetilde P_N^{\downarrow}\circ L_N$. We call this matrix a \emph{stochastic link} between levels $N$ and $N-1$.

{Using the definition of Schur functions we conclude that the matrix elements
$P_N^{\downarrow}(\lambda\to\mu)$ are given by the following formula: {
\begin{equation}
\label{eq_explicit_link}
 P_N^{\downarrow}(\lambda\to\mu) =\begin{cases} \xi_N^{|\lambda|-|\mu|} \dfrac{\det_{i,j=1,\dots,N-1}\bigl[{\xi_i^{\mu_j+N-1-j}}\bigr]}
{\det_{i,j=1,\dots,N}\bigl[{\xi_i^{\lambda_j+N-j}}\bigr]} \prod_{i=1}^{N-1}(\xi_i-\xi_N), &
\mu\prec\lambda,\\ 0, &\text{otherwise}.\end{cases}
 \end{equation}
 }
Note that if $\xi_i$'s form a geometric progression then the determinants in
\eqref{eq_explicit_link} turn into $q$-Vandermonde determinants, cf.\ \eqref{eq:princ_spec}.}

\begin{proposition}
 \label{proposition_commutation_1} Matrices of transition probabilities $P_N(\,\cdot\,; g(x))$
 commute with links $P_N^{\downarrow}$.
\end{proposition}
\begin{proof}
 This is equivalent to commutativity relations between maps $\widetilde Q^g_N$ and  $\widetilde {
P}_N^{\downarrow}$, which is straightforward.
\end{proof}

For our further constructions it is necessary to extend the space $\mathbb{GT}_N$ and the definition
of links $P_N^{\downarrow}$.

The extended level $N$, $\overline{\mathbb{GT}}_N$ consists of all sequences
$\lambda_1\ge\lambda_2\ge\dots\ge\lambda_N$, where $\lambda_i\in \mathbb{Z}\cup \{+\infty\}$. We
identify every such sequence with a shorter sequence $\mu_1\ge\mu_2\ge\dots\ge\mu_k$ with
$\mu_i=\lambda_{i+N-k}$ and $k$ being the smallest integer such that $\lambda_{N-k}=+\infty$, and
with the corresponding $k$--point configuration in $\mathbb{Z}$.

We have {
\begin{equation}\label{eq:split}
\overline{\mathbb{GT}}_N=\mathbb{GT}_N^{(0)}\cup\mathbb{GT}_N^{(1)}\cup
\mathbb{GT}_N^{(2)}\cup\dots\cup \mathbb{GT}_N^{(N)},
\end{equation}
} where
$$\mathbb{GT}_N^{(k)}=\{\lambda:\lambda_1=\dots=\lambda_{N-k}=+\infty,\lambda_{N-k+1},\dots,\lambda_N\in\mathbb{Z}\},$$
{in particular, $\mathbb{GT}_N^{(0)}$ consists of a single signature with all
infinite coordinates.} It is convenient to use the obvious identification
$$
 \mathbb{GT}_N^{(k)}\simeq \mathbb{GT}_{N-1}^{(k)} \simeq \dots \simeq \mathbb{GT}_k^{(k)} =
 \mathbb{GT}_k.
$$

In order to define the extended matrix of transition probabilities $\overline
P_N^{\downarrow}(\lambda\to\mu)$, $\lambda\in{\overline {\mathbb{GT}}_N},\, \mu\in{\overline{
\mathbb{GT}}_{N-1}}$ we first introduce for any $k<N$ an auxiliary stochastic matrices $Q_N^k$
with rows and columns indexed by elements of $\mathbb{GT}_k$ by $Q_N^k=Q^g_k$ with
$g=1/(1-\xi_N^{-1}x)$. {And $Q_N^0$ is the unique $1\times 1$ stochastic matrix.}

Now we are ready to define $\overline P_N^{\downarrow}(\lambda\to\mu)$. This matrix has a block
structure with respect to splittings (\ref{eq:split}) on levels $N$ and $N-1$. For
$\lambda\in\mathbb{GT}_N^{(N)}$
$$
 \overline P_N^{\downarrow}(\lambda\to\mu)=\begin{cases}
  P_N^{\downarrow}(\lambda\to\mu),\text{ if } \mu\in\mathbb{GT}_{N-1}^{(N-1)},\\
  0,\text{ otherwise.}
 \end{cases}
$$

For $\lambda\in\mathbb{GT}_N^{(k)}$ with $k<N$, we define
$$
 \overline P_N^{\downarrow}(\lambda\to\mu)=\begin{cases}
  Q_N^{k}(\lambda\to\mu),\text{ if } \mu\in\mathbb{GT}_{N-1}^{(k)},\\
  0,\text{ otherwise.}
 \end{cases}
$$

Let us also extend the definition of stochastic matrices $P_N(\,\cdot\,; g(x))$  to larger
matrices $\overline P_N(\,\cdot\,; g(x))$ with rows and columns indexed by elements of $\overline
{\mathbb{GT}}_N$. These matrices also have a block structure with respect to (\ref{eq:split}). We
define
$$
\overline P_N(\lambda\to\mu ; g(x))=\begin{cases}
 P_k(\,\cdot\,; g(x)),\text{ if }
 \lambda\in\mathbb{GT}_N^{(k)}\text{ and }\mu\in\mathbb{GT}_N^{(k)},\\
 0,\text{ otherwise,}
\end{cases}
$$
where the arguments of $P_k$ on the right are suitable truncations of $\lambda$ and $\mu$.

\begin{proposition}
  Matrices of transition probabilities $\overline P_N(\,\cdot\,; g(x))$ commute with links
$\overline P_N^{\downarrow}$.
\end{proposition}

\begin{proof}
This is immediate from the definitions.
\end{proof}

\section{Infinite--dimensional dynamics.}

\label{Section_Infinite-dimensional}

In this section we introduce infinite--dimensional dynamics --- the main object of study of the
present paper. We start from general theorems and then we specialize them to the probability
measures and Markov chains on the sets $\mathbb {GT}_N$.

\subsection{General constructions}

\label{subsection_commuting_op_constructions} For any topological space $W$ we denote by $\mathcal
M (W)$ the Banach space of signed measures on $W$ with total variation norm, and by $\mathcal M_p
(W)$ the closed convex subset of probability measures.

Suppose that we have a sequence of countable sets {$\Gamma_0,\Gamma_1,\dots$, and
for any $N\ge 0$} we have a stochastic matrix $\Lambda^{N+1}_N$ with rows enumerated by elements
of $\Gamma_{N+1}$ and columns enumerated by elements of $\Gamma_N$. We call these matrices
\emph{links}. $\Lambda^{N+1}_N$ induces a linear operator mapping $\mathcal M (\Gamma_{N+1})$ to
$\mathcal M(\Gamma_N)$ and we keep the same notation for this operator:
$$
 (\Lambda^{N+1}_N M)(y)=\sum_{x\in\Gamma_{N+1}} M(x) \Lambda^{N+1}_N(x,y),\quad M\in\mathcal
 M(\Gamma_N).
$$

The projective limit $\plim \mathcal
 M(\Gamma_N)$ with respect to the maps $\Lambda^{N+1}_N$ is a Banach space with norm {
$$
 \Vert (M_0,M_1,\dots)\Vert =\sup_N\Vert M_N\Vert .
$$}


Suppose that the sets $\Gamma_N$ are equipped with some topology (which can be, in principle, the
discrete topology). To avoid pathologies we assume that these topologies are metrizable.  Then the
set $\plim \mathcal
 M(\Gamma_N)$ has another natural topology which we call the \emph{weak} topology. The weak topology on $\plim \mathcal
M(\Gamma_N)$ is the minimal topology such that for every $N\ge 0$ and every bounded continuous
function $f(w)$ on $\Gamma_N$ the map
$$
 \{M_N\}\mapsto \sum_{w\in \Gamma_N} f(w) M_N(w)
$$
is continuous.


We equip $\plim \mathcal
 M(\Gamma_N)$ with Borel $\sigma$--algebra spanned by open sets in the norm-topology. One proves
 that this is the same algebra as Borel $\sigma$--algebra spanned by open sets in the weak
 topology. (This is the only place where we need metrizability. It may be unnecessary here as well
 but we do not have a proof of that.)

A projective limit $\plim \mathcal
 M_p(\Gamma_N)$ is a closed (in both topologies) convex subset of $\plim \mathcal
 M(\Gamma_N)$. Elements of $\plim \mathcal
 M_p(\Gamma_N)$ are called \emph{coherent systems}. Note that if $M$ is a coherent system, then
{
 $$
  \Vert M_0\Vert =\Vert M_1\Vert =\dots=\Vert M\Vert =1.
 $$}

\begin{definition}
\label{definition_boundary}
 A topological space $\mathcal Q$ is a \emph{boundary} of a sequence
 $\{\Gamma_N,\Lambda^{N+1}_N\}$ if
 \begin{enumerate}
 \item There exists a bijective map
   $$
    E: \mathcal M(\mathcal Q)\to \plim \mathcal
 M(\Gamma_N);
   $$
  \item $E$ and $E^{-1}$ are bounded linear operators in the corresponding norms;
  \item $E$ maps $\mathcal M_p(\mathcal Q)$ bijectively onto $\plim \mathcal
 M_p(\Gamma_N)$;
   \item $x\to E(\delta_x)$ is a bijection between $\mathcal Q$ and extreme points of the convex set of coherent systems, and
    this bijection
   is a homeomorphism on its image, where we use the restriction of the the weak topology of
 $\plim \mathcal M(\Gamma_N)$.
 \end{enumerate}
\end{definition}

{\bf Remark 1. } If $\mathcal Q_1$ and $\mathcal Q_2$ are two boundaries, then they are
homeomorphic.

{\bf Remark 2. } As follows from $4.$ of the above definition, the boundary can be always
identified with the set of all extreme coherent systems. However, finding a more explicit
description of the boundary can be complicated.

{ {\bf Remark 3. } Some authors \emph{define} the boundary to be the set of all
extreme coherent systems (see e.g. \cite{Kerov_book}, \cite{DF}, \cite{Olsh}), then 1.--4. become
the \emph{properties} of the boundary. Also note that a slightly different definition of the
boundary was used in \cite{BO}. }
\begin{theorem}
\label{theorem_boundary}
 For any sequence $\{\Gamma_N,\Lambda^{N+1}_N\}$ there exists a boundary $\mathcal Q$.
\end{theorem}
\begin{proof}
 Statements similar to Theorem \ref{theorem_boundary} were proved in \cite{DF} and \cite{Olsh}.
We use \cite{Olsh} as our main reference.

 Let $\mathcal Q$ be the set of extreme points of the convex set $\plim \mathcal M_p(\Gamma_N)$. We
 equip $\mathcal Q$ with weak topology inherited from $\plim \mathcal M(\Gamma_N)$. Note
 that condition $4.$ of Definition \ref{definition_boundary} is satisfied automatically.

 Theorem 9.2 in \cite{Olsh} says that $\mathcal Q$ is a Borel subset of $\plim \mathcal
 M_p(\Gamma_N)$, and for any $M\in \plim \mathcal
 M_p(\Gamma_N)$ there exist a unique probability measure $\pi_M$ on $\mathcal Q$ such that
$$
 M=\int_{\mathcal Q} M^{\mathsf q} \pi_M(d\mathsf q),
$$
meaning that for any {$N\ge 0$} and any subset $A$ of $\Gamma_N$ we have
$$
 M_N(A)=\int_{\mathcal Q} M^{\mathsf q}_N(A) \pi_M(d\mathsf q).
$$

For any $\pi\in\mathcal M_p(\mathcal Q)$, set
$$
E(\pi)=\int_{\mathcal Q} M^{\mathsf q} \pi(d\mathsf q).
$$
If $\pi$ is a signed measure, i.e.\ $\pi\in\mathcal M(\mathcal Q)$, then there exist
$\pi_1,\pi_2\in \mathcal M_p(\mathcal Q)$ and two non-negative numbers $c_1$, $c_2$ such that
$$
 \pi=c_1\pi_1-c_2\pi_2
$$
and $c_1+c_2=\Vert \pi\Vert $. We define
$$
 E(\pi)=c_1E(\pi_1)-c_2 E(\pi_2).
$$

Clearly, $E$ is a linear operator from $\mathcal M(\mathcal Q)$ to $\plim \mathcal
 M(\Gamma_N)$ of norm 1. Using Theorem 9.2 from \cite{Olsh} we conclude that condition $3.$ of Definition \ref{definition_boundary}
  is satisfied and, moreover, $E$ is an injection.

 In order to prove that $E$ is a surjection it is enough to show that for any $M\in\plim \mathcal
 M(\Gamma_N)$ there exist $K,L\in \plim \mathcal M_p(\Gamma_N)$ and $k,l\ge 0$ such that
 $M=kK-lL$. Without loss of generality we may assume that neither $M$, nor $-M$ are positive measures.
Let $M_N= K^N_N-L^N_N$ be a decomposition of $M_N$ into a difference of two
 positive measures such that $\Vert M_N\Vert =\Vert K^N_N\Vert +\Vert L^N_N\Vert $. Set $K^N_{N-1}=\Lambda^{N}_{N-1}
 K^N_N$, $K^{N}_{N-2}=\Lambda^{N-1}_{N-2} K^{N}_{N-1}$ and so on, and similarly for $L^N_N$. Note
 that for any {$N\ge 0$}, $K^i_N$ monotonically increases as $i\to\infty$. Also $\Vert K^i_N\Vert \le \Vert M\Vert $.
 Hence, there exists a limit
 $$
  K^{\infty}_N=\lim_{i\to\infty} K^i_N.
 $$
 In the same way there is a limit
 $$
  L^{\infty}_N=\lim_{i\to\infty} L^i_N.
 $$
 Note that for any $i$, $K^i_N$ and $L^i_N$ are positive measures and $M_N=K^i_N-L^i_N$.
 Therefore, similar statements hold for $M$, $K^{\infty}_N$ and $L^{\infty}_N$.
 Setting $k=\Vert K^{\infty}\Vert $, $l=\Vert L^{\infty}\Vert $ (neither $k$ nor $l$ can vanish
because we assumed $M$ is not positive), $K=K^{\infty}/k$, $L=L^{\infty}/l$ we get the
 required decomposition of $M$.

 We have proved that $E$ is a bounded linear operator in Banach spaces which is a bijection. Then
it follows from Banach Bounded Inverse Theorem that $E^{-1}$ is also a bounded linear operator.
\end{proof}

Next, we want to introduce a way to define a Markov chain on the boundary $\mathcal Q$ of
$\{\Gamma_N,\Lambda^{N+1}_N\}$ using Markov chains on the sets $\{\Gamma_N\}$.


Let $\mathcal W_1,\mathcal W_2$ be topological spaces. A function $P(u,A)$ of a point
$u\in\mathcal W_1$ and Borel subset $A$ of $\mathcal W_2$ is \emph {a Markov transition kernel} if
\begin{enumerate}
\item For any $u\in\mathcal W_1$, $P(u,\,\cdot\,)$ is a probability measure on $\mathcal W_2$,
\item For any Borel subset $A\subset\mathcal W_2$, $P(\,\cdot\,,A)$ is a measurable function on $W_1$.
\end{enumerate}
If $\mathcal W_1=\mathcal W_2=\mathcal W$, then we say that $P(u,A)$ is a Markov transition kernel
on $\mathcal W$.

Suppose that for every $N\ge 0$ we have a Markov transition kernel on $\Gamma_N$ given by a
stochastic matrix $\mathcal P_N(x\to y)$, $x,y\in\Gamma_N$. Assume that these matrices commute
 with  links $\Lambda^{N+1}_N$, i.e. for $N=0,1,2,\dots$ we have
\begin{equation}
\label{eq_commutation_in_general_form}
 \mathcal P_N\Lambda^{N+1}_N=\Lambda^{N+1}_N \mathcal P_{N+1}.
\end{equation}
Now we define a kernel $\mathcal P_\infty$ on $\mathcal Q$ in the following way.
 Take a point $u\in\mathcal Q$; by Theorem \ref{theorem_boundary} Dirac $\delta$--measure
$\delta_u$ corresponds to a certain coherent system $\{M^u_N\}$. Then relations
\eqref{eq_commutation_in_general_form} yield that $\{\mathcal P_N M_N\}$ is also a coherent system
and, consequently, it corresponds to a probability measure on $\mathcal Q$. We \emph{define}
$\mathcal P_{\infty}(u, \,\cdot\,)$ to be equal to this measure.

\begin{proposition}
\label{Proposition_is_a_Markov_kernel} $\mathcal P_{\infty}$ is a Markov transition kernel on
$\mathcal Q$.
\end{proposition}
\begin{proof}
 $\mathcal P_{\infty}(u, \,\cdot\,)$ is a probability measure on $\mathcal Q$ by the very
definition. Thus, it remains to check that $\mathcal P_{\infty}(\,\cdot\,, A)$ is a measurable
function. Theorem \ref{theorem_boundary} yields that the map $u\mapsto M_N^u(x)$ (where
$M_N^u=E(\delta_u)_N$) is continuous and, thus, measurable for any {$N\ge 0$} and any
$x\in\Gamma_N$. Then the map $u\mapsto (\mathcal P_N M_N^u) (x)$ is also measurable. Since, the
last property holds for any $x$, we conclude that $u\mapsto \{\mathcal P_N M_N^u\} $ is a
measurable map from $\mathcal Q$ to coherent systems. The definition of the boundary implies that
the correspondence between probability measures on $\mathcal Q$ and coherent systems is
bi-continuous and, therefore, it is bi-measurable. We conclude that the correspondence $u\mapsto
\mathcal P_{\infty}(u, \,\cdot)$ is a measurable map between $\mathcal Q$ and $\mathcal
M_p(\mathcal Q)$ (the Borel structure in the latter space corresponds to the total variation
norm).
 It remains to note that for any Borel $A\subset \mathcal Q$ the map
 $$
  {\rm Ev}_A: \mathcal M(\mathcal Q) \to \mathbb R,\quad {\rm Ev}_A( P) =P(A)
 $$
 is continuous (in the total variation norm topology), thus measurable. Therefore, for any
 Borel $A$, the map $u\mapsto
 \mathcal P_{\infty}(u, A)$ is measurable.
\end{proof}

We see that transition kernels $\mathcal P_{N}$ naturally define Markov chains on $\mathcal Q$.
However, Theorem \ref{theorem_boundary} tells us nothing about the actual state space of these
Markov chains, and their properties can be very different. Next we state sufficient conditions for
these Markov chains on $\mathcal Q$ to enjoy the \emph{Feller property} that we now recall.

\smallskip

Let $\mathcal W$ be a locally compact topological space with a countable base. Let $\mathcal
B(\mathcal W)$ be the Banach space of real valued bounded measurable functions on $\mathcal W$,
and let $C_0(\mathcal W)$ be the closed subspace of all continuous function tending to zero at
infinity. In other words, these are continuous function $f(w)$ such that for any $\varepsilon>0$
there exist a compact set $V\subset \mathcal W$ such that $|f(w)|<\varepsilon$ for all
$w\in\mathcal W\setminus V$. Note that $C_0(\mathcal W)$ is separable and its Banach dual is
$\mathcal M(\mathcal W)$.

 Let $P(u,A)$, $u\in W_1$, $A\subset W_2$ be a Markov transition kernel. It induces a linear
contracting operator $P^*: \mathcal B(\mathcal W_2) \to  \mathcal B(\mathcal W_1)$ via
$$
 (P^*f)(x)=\int_{\mathcal W_2} f(w) P(x,dw).
$$
We say that $P(u,A)$ is a \emph{Feller kernel} if $P^*$ maps $C_0(\mathcal W_2)$ to $C_0(\mathcal
W_1)$.

Now let $X(t)$ be a homogeneous continuous time Markov process on $\mathcal W$ with transition
probabilities given by a semigroup of kernels $P_t(u,A)$. $X(t)$ is a \emph{Feller process} if the
following conditions are satisfied:
\begin{enumerate}
\item $P_t(u,A)$ is a Feller kernel, i.e.\ $P_t^*$ preserves $C_0(\mathcal W)$.
\item For every $f\in C_0(\mathcal W)$ the map $f\to P_t^*f$ is continuous at $t=0$.
\end{enumerate}

Feller processes have certain good properties. For instance, they have a modification with
c\`{a}dl\`{a}g sample paths,  they are strongly Markovian, they have an infinitesimal generator,
{see e.g.\, \cite[Section 4.2]{EK}}.

\medskip

\begin{proposition} \label{proposition_feller}
Let, as above,  $\mathcal Q$ be the boundary of $\{\Gamma_N,\Lambda^{N+1}_N\}$, $\mathcal P_N$ be
 stochastic matrices on $\Gamma_N$ commuting with $\Lambda^{N+1}_N$, and let $\mathcal P_\infty$
 be the corresponding Markov transition kernel on $\mathcal Q$. Suppose that
 \begin{enumerate}
  \item $\Gamma_N$ and $\mathcal Q$ are locally compact topological spaces with countable bases;
  \item For every $N\ge 0$ the function $\Lambda^{\infty}_N(u,A):=E(\delta_u)_N(A)$, $u\in \mathcal Q$, $A\subset \Gamma_N$
  is a Feller kernel;
  \item Matrices $\mathcal P_N$ define Feller kernels on $\Gamma_N$.
 \end{enumerate}
 Then $\mathcal P_\infty$ is a Feller kernel on $\mathcal Q$.
\end{proposition}

\begin{proof}
Let $h\in C_0(\mathcal Q)$. We need to check that $\mathcal P^*_{\infty}(h)\in C_0(\mathcal Q)$.
Since $C_0(\mathcal Q)$ is closed and $\mathcal P^*_{\infty}(g)$ is a contraction, it is enough to
check this property on a dense set of functions $h$. Let us find a suitable dense set of
functions.

We claim that the union of the sets $(\Lambda^{\infty}_N)^*(C_0(\Gamma_N))$ over {$N\ge 0$} is
dense in $C_0(\mathcal Q)$. Indeed, $\mathcal M(\mathcal Q)$ is a Banach dual to $C_0(\mathcal
Q)$, thus, it is enough to check that if $\pi\in\mathcal M(\mathcal Q)$ is such that $\int f
d\pi=0$ for all $N$ and all $f\in (\Lambda^\infty_N)^* (C_0(\Gamma_N))$, then $\pi\equiv 0$.
 The latter property is equivalent (by Fubini's theorem) to the following one: For any $N$ and any
 $f\in C_0(\Gamma_N)$, $\int f d \Lambda^\infty_N(\pi)=0$.  But then $\Lambda^\infty_N(\pi)\equiv
 0$, therefore, $E(\pi)\equiv 0$ and $\pi\equiv 0$. The claim is proved.

Now let $h=(\Lambda^\infty_N)^*(f)$. Then by the definitions
$$
\mathcal P^*_{\infty}(h)= \mathcal P^*_{\infty}((\Lambda^\infty_N)^*(f)) =
\Lambda^\infty_N(\mathcal P^*_N(f)).
$$

But $\mathcal P^*_N(f)\in C_0(\Gamma_N)$ by the condition $3.$ of Proposition
\ref{proposition_feller}. Thus, $\mathcal P^*_{\infty}(h)\in C_0(\mathcal Q)$ by the condition
$2.$
\end{proof}

\begin{proposition}
\label{proposition_Feller_process} Let $X_N(t)$ be a homogeneous continuous time Markov process on
$\Gamma_N$ with transition probabilities given by a semigroup of stochastic matrices $\mathcal
P_{t,N}$. Suppose as above,  that $\mathcal Q$ is the boundary of $\{\Gamma_N,\Lambda^{N+1}_N\}$,
the matrices $\mathcal P_{t,N}$ commute with links $\Lambda^{N+1}_N$ and let $\mathcal
P_{t,\infty}$ be the corresponding semigroup of Markov transition kernels on $\mathcal Q$.

If $X_N(t)$ is a Feller process for every $N\ge 0$ and, furthermore, assumptions $1.$ and $2.$ of
Proposition \ref{proposition_feller} are also satisfied, then a Markov process on $\mathcal Q$
with semigroup of transition probabilities $\mathcal P_{t,\infty}$ and arbitrary initial
distribution is Feller.
\end{proposition}
\begin{proof}
 The first defining property of the Feller process is contained in Proposition
 \ref{proposition_feller}. As for the second one, since $\mathcal P_{t,\infty}$ is a contraction,
 we may check this property on a dense subset.  If $h=(\Lambda^\infty_N)^*(f)$ and $f\in
 C_0(\Gamma_N)$, then
$$\mathcal P_{t,\infty}^*
(h)=(\Lambda^\infty_N)^*( \mathcal P_{t,N}^*(f)).
$$
And this is a continuous map as a composition of a continuous map and a contraction.
\end{proof}

\subsection{Specialization}

Now let us specialize the general situation by setting $\Gamma_N=\mathbb{GT}_N$ (with discrete
topology) and $\Lambda^{N+1}_N=P_{N+1}^{\downarrow}$, where these matrices implicitly depend on
the sequence $\{\xi_N\}$. Denote by $\mathcal Q(\xi_1,\xi_2,\dots)$ the boundary of the sequence
$\{ \mathbb{GT}_N, P_{N+1}^{\downarrow}\}$.

Given an admissible function $g(x)\in\mathcal G(\infty,\xi)$, we set $\mathcal P_N:= P_{N}(u, A;
g(x))$. Proposition \ref{proposition_commutation_1} yields that these stochastic matrices commute
with links $P_{N+1}^{\downarrow}$, thus, the constructions of Section
\ref{subsection_commuting_op_constructions} give as  a Markov transition kernel $\mathcal
P_\infty$ on $\mathcal Q(\xi_1,\xi_2,\dots)$ which we denote $P_{\infty}(u, A; g(x))$.

As far as the authors know, an explicit description of the set $\mathcal Q(\xi_1,\xi_2,\dots)$ is currently known
in two special cases. Namely,
\begin{enumerate}
\item $\xi_1=\xi_2=\dots=1$. This case was studied by Voiculescu \cite{Vo}, Boyer \cite{Bo}, Vershik--Kerov \cite{Vk_unitary},
Okounkov--Olshanski \cite{OkOlsh}. It is related to representation theory of the
infinite--dimensional unitary group and to the classification of totally--positive Toeplitz
matrices. However, it turns out that in this case the Markov operators $P_{\infty}(g(x))$
correspond to \emph{deterministic} dynamics on $\Omega$, cf.\ \cite{B-Schur}.

\item $\xi_N=q^{1-N}, \, 0<q<1$. This case was studied in \cite{Gor}. As we  show below,
$P_{\infty}(g(x))$ leads to a non-trivial stochastic dynamics. (Note that the case $q>1$
is essentially the same as $0<q<1$ up to a certain simple transformation of spaces $\mathbb{GT}_N$.)
\end{enumerate}

From now on we restrict ourselves to the case $\xi_N=q^{1-N}$, $N\ge 1$. The following results
were proven in \cite[Theorem 1.1, Proposition 1.2]{Gor}.

\begin{theorem}
\label{theorem_qGT_boundary}
 For $\xi_N=q^{1-N}$, $0<q<1$, the boundary $\mathcal Q(\xi_1,\xi_2,\dots)$ is homeomorphic to the set
 $\mathcal N$ of infinite increasing sequences of integers
 $$
  \mathcal{N}=\{\nu_1\le\nu_2\le\dots,\, \nu_i\in\mathbb{Z}\}
 $$
 with the topology of coordinate--wise convergence.

 Denote by $E_q$ the bijective map from (signed) measures on $\mathcal N$ to $\plim {\mathcal M}(
\mathbb{GT}_N)$, and let $ \mathcal E^\nu=E_q(\delta^\nu)$. Then the coherent system $\mathcal
E^\nu_N$, $N=0,1,2,\dots$, has the following property: If we view $\lambda_N$, $\lambda_{N-1}$,
\dots as random variables on the probability space $(\mathbb{GT}_N, \mathcal E^{\nu}_N)$, then for
every $k\ge 1$, $\lambda_{N-k+1}$ converges (in probability) to $\nu_{k}$ as $N\to\infty$.
\end{theorem}

Similarly to $\mathbb{GT}_N$, we identify elements of $\mathcal N$ with point configurations in
$\mathbb{Z}$:
$$
 \nu_1\le\nu_2\le\dots \longleftrightarrow \{\nu_i+i-1,\, i=1,2,\dots\}.
$$
Note that in this way we get all semiinfinite point configurations in $\mathbb Z$; we denote the
set of such configurations by $\mathbb M_{\infty}$.

\smallskip

The space $\mathcal N$ is \emph{not} locally compact. This introduces certain technical
difficulties in studying continuous time Markov chains on this space. To avoid these difficulties
we seek a natural local compactification of $\mathcal N$.

Let
$$
 \overline {\mathcal N}=\bigsqcup_{N=0}^{\infty} \mathbb{GT}_N\ \sqcup \mathcal N.
$$
We identify elements of $\mathbb{GT}_N$ with infinite sequences $\nu_1\le\nu_2\dots$ such that
$\nu_1,\dots,\nu_N$ are integers and $\nu_{N+1}=\nu_{N+2}=\dots=+\infty$. Thus,
$$
 \overline {\mathcal N}=\{\nu_1\le\nu_2\le\dots,\, \nu_i\in\mathbb{Z}\cup\{+\infty\}\}.
$$

In the same way we set
$$
 \overline {\mathbb M_\infty} =\mathbb M_{\infty}\cup \mathbb M_0\cup \mathbb M_1 \cup \dots
$$
Clearly, $\overline {\mathbb M_\infty}$ is a set of all point configurations in $\mathbb{Z}$ which
have finitely many points to the left from zero. There is a natural bijection between $
 \overline {\mathcal N}$ and $\overline {\mathbb M_\infty}$.

We equip $\overline {\mathcal N}$ with the following topology. The base consists of
the neighborhoods
$$
 A_{\eta,k}=\{\nu\in \overline{\mathcal N}: \nu_1=\eta_1,\dots,\nu_k=\eta_k\},\quad
 \eta_i\in\mathbb{Z},
$$
and
$$
 B_{\eta,k,\ell}=\{\nu\in \overline{\mathcal N}: \nu_1=\eta_1,\dots,\nu_k=\eta_k, \nu_{k+1}\ge\ell \},\quad
 \eta_i\in\mathbb{Z},\quad l\in \mathbb{Z}.
$$

The following proposition is straightforward.
\begin{proposition}
 Topological space $\overline {\mathcal N}$ is locally compact, the natural inclusion $\mathcal
 N\hookrightarrow \overline {\mathcal N}$ is continuous, and its image is dense in $\overline {\mathcal
 N}$.
\end{proposition}

Now we are ready to define a kernel $\overline{P_{\infty}} (u, A; g(x))$ on $\overline {\mathcal
N}$: If $u\in\mathbb{GT}_N\subset \overline {\mathcal N}$, then $\overline{P_{\infty}} (u, A;
g(x))$ is a discrete probability measure concentrated on $\mathbb{GT}_N$ with weight of a
signature $\lambda\in\mathbb{GT}_N$ given by $P_N(u\to\lambda; g(x))$; if $u\in\mathcal N\subset
\overline{\mathcal N}$, then measure $\overline{P_{\infty}} (u, A; g(x))$ is concentrated on
$\mathcal N$ and coincides with ${P_{\infty}} (u, A; g(x))$ on it.

\begin{proposition}
 $\overline{P_{\infty}} (u, A; g(x))$ is a Markov transition kernel on $\overline {\mathcal
N}$.
\end{proposition}

\begin{proof}
 For every $u$, $\overline{P_{\infty}} (u, \,\cdot\, ; g(x))$ is a probability measure by the
 definition. The measurability of $\overline{P_{\infty}} (\,\cdot\,, A ; g(x))$ follows from the
 measurability of ${P_{\infty}} (\,\cdot\,, A; g(x))$ and $P_N(\cdot\to A; g(x))$.
\end{proof}

In Section \ref{Section_feller} we prove that $\overline{P_{\infty}} (u, \,\cdot \,; g(x))$ is a
Feller kernel by identifying $\overline{\mathcal N}$ with a boundary of $\{
\overline{\mathbb{GT}_N},\overline P_N^{\downarrow}\}$ and then using Proposition
\ref{proposition_feller}.

\section{Description of the limiting processes}
\label{Section_boundary_process_distributions}

In this section we study finite--dimensional distributions of Markov processes that correspond to
the Markov kernels ${P_{\infty}} (u, A; g(x))$.

\subsection{A general convergence theorem}

Let $g_k(x)$, $k=0,1,2,\dots$, be a sequence of admissible functions, and let $\mathcal Z_0$ be an
arbitrary probability distribution on $\mathcal N$.

Denote by $\mathcal Z_N(t)$, $t=0,1,2,\dots$, a discrete time Markov chain on $\mathbb{GT}_N$ with
initial distribution $\mathcal Z_N(0)\stackrel{D}{=} E_q(\mathcal Z_0)_N$ and transition
probabilities
$$
 {\rm Prob}\{\mathcal Z_N(t+1)\in A\mid\mathcal  Z_N(t)\}= {P_{N}} (\mathcal Z_N(t), A; g_t(x)).
$$

Also let $\mathcal Z(t)$, $t=0,1,2,\dots$, be a discrete time Markov process on $\mathcal N$ with
initial distribution $\mathcal Z(0)\stackrel{D}{=} Z_0$ and transition measures
$$
 {\rm Prob}\{\mathcal Z(t+1)\in A\mid \mathcal Z(t)\}= {P_{\infty}} (\mathcal Z(t), A; g_t(x)).
$$

Note that the processes $\mathcal Z(t)$ and $\mathcal Z_N(t)$ will always depend on $q$ and on the
$\{g_k\}$, although we omit these dependencies in the notations.


We want to prove that finite-dimensional distributions of  processes $\mathcal Z(t)$  are limits
of the distributions of processes $\mathcal Z_N(t)$.

More formally, introduce embeddings:
$$
 \iota_N:\mathbb{GT}_N\hookrightarrow\mathcal N, \quad (\lambda_1\ge\dots\ge\lambda_N) \mapsto
(\lambda_N\le\dots\le\lambda_1\le\lambda_1\le\lambda_1\dots).
$$
We also use the same notations for the induced maps $\mathcal M(\mathbb{GT}_N)\to\mathcal
M(\mathcal N)$. Note that these maps are isometric in total variation norm.

Cylindrical subsets of $\mathcal N$ have the form
$$
 U=\{(\nu_1\le\nu_2\le\dots)\in \mathcal N\mid \nu_1\in H_1,\dots,\nu_k\in H_k\}
$$
for aribitrary subsets $H_1,\dots ,H_k$ of $\mathbb Z$.

The following statement is the main result of this section.

\begin{theorem}
\label{Theorem_convergence_of_joint_distributions} For every $k\ge 1$, the joint distribution of
random variables $\bigl(\iota_N(\mathcal Z_N(0)),\dots, \iota_N(\mathcal Z_N(k))\bigr)$ weakly
converges to the joint distribution of $(\mathcal Z(0),\dots,\mathcal Z(k))$ as $N\to\infty$.

Equivalently, if $A_0,\dots, A_k$ are arbitrary cylindrical subsets of $\mathcal N$, then
\begin{multline*}
\lim_{N\to\infty} {\rm Prob}\{\iota_N(\mathcal Z_N(0))\in A_0,\dots,\iota_N(\mathcal Z_N(k))\in
A_k\}
\\= {\rm Prob}\{\mathcal Z(0)\in A_0,\dots,\mathcal Z(k)\in A_k\}.
\end{multline*}
\end{theorem}

We start the proof with two lemmas.

\begin{lemma}
\label{Lemma_approximation_of_boundary_measures} Let $\mu$ be a finite measure on $\mathcal N$, and
let $A$ be any cylindrical subset of $\mathcal N$. We have
$$
\mu(A)=\lim_{N\to\infty} \iota_N(E_q(\mu)_N)(A).
$$
\end{lemma}
\begin{proof}
 It suffices to prove the lemma for cylindrical sets of the form
 $$
  A=A_{b(1),\dots,b(\ell)}=\{\nu_1\le\nu_2\le\dots\mid  \nu_1=b(1),\dots,\nu_\ell=b(\ell);\ b(j)\in \mathbb{Z}\}.
 $$
 First, suppose that $\mu=\delta^\nu$ for a certain $\nu\in\mathcal N$, then $\mu(A)=1$ if
$\nu_1=b(1),\dots,\nu_\ell=b(\ell)$, and $\mu(A)=0$ otherwise. The statement of Lemma
\ref{Lemma_approximation_of_boundary_measures} follows from Theorem \ref{theorem_qGT_boundary}.

 For an arbitrary measure we have
\begin{multline*}
 \mu(A)=\int_{\mathcal N} \delta^{\nu}(A) \mu(d\nu)=\int_{\mathcal N} \lim_{N\to\infty}
 \iota_N(E_q(\delta^{\nu})_N)(A) \mu(d\nu)\\ \stackrel{(*)}{=}\lim_{N\to\infty}
 \int_{\mathcal N} \iota_N(E_q(\delta^{\nu})_N)(A) \mu(d\nu) =
 \lim_{N\to\infty}\iota_N(E_q(\mu)_N)(A),
\end{multline*}
where the equality $(*)$ follows from the dominated convergence theorem.
\end{proof}

Let us denote by $I_A$ the indicator function of set $A$:
$$
 I_A(x)= \begin{cases} 1,& x\in A,\\ 0,& \text{otherwise.} \end{cases}
$$
If $\mu$ is a measure, then $I_A\mu$ stands for the measure given by
$$
 (I_A\mu) (B)=\mu(A\cap B).
$$
\begin{lemma}
\label{lemma_total_variation_dist}
{ Let $\mu$ be a probability measure on $\mathcal N$} and let $A$
 be any cylindrical set. Then the total variation distance between measures $\iota_N (E_q(\mu
 I_A)_N)$ and $I_A \iota_N (E_q(\mu)_N)$ tends to zero as $N\to\infty$.
\end{lemma}
\begin{proof}
First, suppose that $\mu=\delta^\nu$ for a certain sequence $\nu\in\mathcal N$. If $\nu\in A$,
then $E_q(I_A\mu)_N=E_q(\mu)_N$, consequently,
\begin{multline*}
\bigl\Vert \iota_N (E_q(I_A\mu)_N) -I_A \iota_N (E_q(\mu)_N)\bigr\Vert =\bigl\Vert
(1-I_A)\iota_N(E_q(\mu)_N)\bigr\Vert = \iota_N(E_q(\mu)_N)({\bar A}),
\end{multline*}
where $\bar A=\mathcal N\setminus A$. The right-hand side tends to zero by Lemma
\ref{Lemma_approximation_of_boundary_measures}.

If $\nu$ does not belong to $A$, then $( I_A\mu)_N =0$ and
$$
\bigl\Vert \iota_N (E_q(I_A \mu)_N) -I_A \iota_N (E_q(\mu)_N)\bigr\Vert =\bigl\Vert I_A \iota_N(
E_q(\mu)_N)\bigr\Vert =\iota_N (E_q(\mu)_N)(A)\to 0
$$
by Lemma \ref{Lemma_approximation_of_boundary_measures}.

To prove the claim for a general measure $\mu$ we observe the following property of the total
variation norm. { Suppose that $\mathcal W$ and $\mathcal V$ are measurable spaces,
and $f_N$ is a sequence of measurable maps from $\mathcal W$ to $\mathcal M(\mathcal V)$, such that $\Vert
f_N(w)\Vert \le 1$ and $\Vert f_N(w)\Vert \to 0$ for any $w\in\mathcal W$.} Then for any
probability measure $\pi$ on $\mathcal W$, we have
$$
 \left\Vert \int_{\mathcal W} f_N(w) \pi(dw)\right\Vert \to 0.
$$

We obtain
\begin{multline*}\bigl\Vert \iota_N( E_q(
 I_A\mu )_N)-I_A \iota_N (E_q(\mu)_N)\bigr\Vert \\=\left\Vert \int_{\mathcal N} \biggl(\iota_N( E_q(I_A \delta^{\nu})_N)-I_A
 \iota_N (E_q(\delta^{\nu})_N)\biggr) \mu(d\nu)\right\Vert \to 0.
\end{multline*}



\end{proof}



\begin{proof}[Proof of Theorem \ref{Theorem_convergence_of_joint_distributions}]

Let us write ${P_{\infty}^g}$ for the linear operator on $\mathcal M(\mathcal N)$ corresponding
to the kernel $P_{\infty} (u, A; g(x))$:
$$
 (P_\infty^g\pi)(A)=\int_{\mathcal N} P_{\infty} (u, A; g(x)) \pi(du).
$$

Let ${P_{N}^g}$ be the operator acting on measures of the form $\iota_N (\mathcal M(\mathbb{GT}_N))$
via the kernel ${P_N}(g(x))$:
$$
 (P_N^g\pi)(\iota_N(\eta))=\sum_{\lambda\in\mathbb{GT}_N} P_N (\lambda\to\eta; g(x)) \pi(\lambda),
\qquad
\eta\in\mathbb{GT}_N.
$$
Note that these operators are contractions in total variation norm.

We have
$$
 {\rm Prob}\{\mathcal Z(0)\in A_0,\dots,\mathcal Z(k)\in A_k\}= (I_{A_k} P_{\infty}^{g_{k-1}}\dots
(I_{A_1}P_\infty^{g_1}(I_{A_0} \mathcal Z_0))\dots)(\mathcal N)
$$
and
\begin{multline*}
 {\rm Prob}\{\iota_N(\mathcal Z_N(0))\in A_0,\dots,\iota_N(\mathcal Z_N(k))\in A_k\} \\ =
(I_{A_k}P_{N}^{g_{k-1}}\dots (I_{A_1}P_N^{g_1}(I_{A_0} \mathcal Z_0))\dots)(\mathcal N).
\end{multline*}

Applying Lemmas \ref{Lemma_approximation_of_boundary_measures}, \ref{lemma_total_variation_dist}
and definitions we obtain

\begin{align*}
 (I_{A_k} P_{\infty}^{g_{k-1}}&\dots (I_{A_1}P_\infty^{g_1}(I_{A_0} \mathcal Z_0))\dots)(\mathcal
N)\\&=\lim_{N\to\infty} \iota_N (E_q(I_{A_k} P_{\infty}^{g_{k-1}}\dots
(I_{A_1}P_\infty^{g_1}(I_{A_0} \mathcal Z_0))\dots)_N) (\mathcal N)
\\&=\lim_{N\to\infty} (I_{A_k}\iota_N (E_q( P_{\infty}^{g_{k-1}}\dots
(I_{A_1}P_\infty^{g_1}(I_{A_0} \mathcal Z_0))\dots)_N)) (\mathcal N) \\&=\lim_{N\to\infty}
(I_{A_k}P_{N}^{g_{k-1}} I_{A_{k-1}} \iota_N (E_q( P_{\infty}^{g_{k-2}}(\dots
(I_{A_1}P_\infty^{g_1}(I_{A_0} \mathcal Z_0))\dots)_N)) (\mathcal N)\\&=\dots =
\lim_{N\to\infty}(I_{A_k}P_{N}^{g_{k-1}}\dots (I_{A_1}P_N^{g_1}(I_{A_0} \mathcal
Z_0))\dots)(\mathcal N)
\end{align*}
as desired. \end{proof}

\subsection{Correlation functions}

\label{Subsection_correlation_functions}

Let $Q(t)$ be a stochastic process taking values in subsets of $\mathbb Z$ ($=$\ point
configurations in $\mathbb Z$). For $n\ge 1$, the $n$th correlation function $\rho_n$ of $Q(t)$ is the following
function of $n$ distinct pairs $(x_i,t_i)$:
$$
 \rho_n(x_1,t_1;x_2,t_2;\dots;x_n,t_n)={\rm Prob}\{x_1\in Q(t_1),\dots,x_n\in Q(t_n)\}.
$$

Recall that we identify $\mathbb{GT}_N$ with $N$--point configurations in $\mathbb Z$, and
$\mathcal N$ is identified with infinite subsets of $\mathbb Z$ that do not have $-\infty$ as
their limit point. The following statement is a corollary of Theorem
\ref{Theorem_convergence_of_joint_distributions}.

\begin{theorem}
 The correlation functions of processes $\mathcal Z_N(t)$ pointwise converge as $N\to\infty$ to the correlation
functions of processes $\mathcal Z(t)$.
\end{theorem}

{
\begin{proof}
 Let us proof this statement for the first correlation function, for all other
 correlation functions the proof is analogous. Let $\rho_1^N(x,t)$ denote the first correlation
 function of the point configuration corresponding to $\mathcal Z_N(t)$ and let $\rho_1(x,t)$ denote the first correlation
 function of the point configuration corresponding to $\mathcal Z(t)$.

 Choose $\varepsilon>0$. Let $m$ be a number such that ${\rm Prob}\{\mathcal Z(t)_1>m\}
 >1-\varepsilon$. Since the distribution of $\iota _N(\mathcal Z_N(t))_1$ converges to that of
 $\mathcal Z(t)_1$, for sufficiently large $N$ we have ${\rm Prob}\{(\iota _N \mathcal
 Z_N(t))_1>m\}
 >1-2\varepsilon$. Now let $k>x-m$. Note that $(\iota _N \mathcal Z_N(t))_1>m$ implies $(\iota _N\mathcal
 Z_N(t))_{k+1}+(k+1)-1>x$. Thus,
 $$
 \biggl|{\rm Prob}\Bigl\{ (\iota _N \mathcal Z_N(t))_1 = x \text{ or } \dots \text{
  or } (\iota _N \mathcal Z_N(t))_k+k-1=x\Bigl\}-\rho_1^N(x,t)\biggr|<2\varepsilon
 $$
 and similarly for $\mathcal Z(t)$.

 It follows from Theorem \ref{Theorem_convergence_of_joint_distributions} that for sufficiently
 large $N$
 \begin{multline*}
  \biggl|{\rm Prob}\Bigl\{ (\iota _N \mathcal Z_N(t))_1 = x \text{ or } \dots \text{
  or } (\iota _N \mathcal Z_N(t))_k+k-1=x \Bigr\} \\- {\rm Prob}\Bigl\{ \mathcal Z(t)_1 = x \text{ or } \dots \text{
  or }  \mathcal Z(t)_k+k-1=x\Bigr\}\biggr|<\varepsilon.
 \end{multline*}

 Therefore, for sufficiently large $N$ we have
 \begin{align*}
 \bigl|\rho_1^N(x,t) &- \rho_1(x,t)\bigr|\\&\le  \biggl|\rho_1^N(x,t)- {\rm Prob}\Bigl\{ (\iota _N
\mathcal Z_N(t))_1 = x \text{ or }
 \dots \text{ or } (\iota _N \mathcal Z_N(t))_k+k-1\Bigr\}\biggr| \\
&+ \biggl|{\rm Prob}\Bigl\{ (\iota _N \mathcal Z_N(t))_1 = x
  \text{ or }\dots \text{ or } (\iota _N \mathcal Z_N(t))_k+k-1=x\Bigr\}\\&- {\rm Prob}\Bigl\{
\mathcal Z(t)_1 = x \text{ or }
  \dots \text{ or }  \mathcal Z(t)_k+k-1=x\Bigr\}\biggr| \\&+ \biggl|{\rm Prob}\Bigl\{ \mathcal
Z(t)_1 = x  \text{ or }\dots \text{
  or }  \mathcal Z(t)_k+k-1=x\Bigr\}-\rho_1(x,t)\biggr|\\&< 2\varepsilon +\varepsilon
+\varepsilon= 4\varepsilon.
 \end{align*}
  Since $\varepsilon$ is arbitrary, the proof is complete.
\end{proof}
}

Now we are in position to actually compute the correlation functions of $\mathcal Z(t)$. From now
on we confine ourselves to the case $\mathcal Z_0=\delta^{\mathbf{0}}$. In other words, $\mathcal
Z(0)= \mathbf{0}=(0\le 0\le 0\le \dots)$. As shown in \cite[Theorem 1.1]{Gor}, this implies that
for every $N\ge 1$, $\mathcal Z_N(0)=\mathbf{0}_N=(0\le 0\le\dots\le 0)$.

\begin{proposition}
\label{Proposition_correlation_kernel_prelimit} For any $n,N\ge 1$, the $n$th correlation function
$\rho_n^N$ of the process $\mathcal Z_N(t)$ started from $\mathcal Z_N(0)=\mathbf{0}_N$ admits the
following determinantal formula:

$$
\rho_n^N(x_1,t_1;x_2,t_2;\dots;x_n,t_n)=\det_{i,j=1,\dots,n}[K^N(x_i,t_i;x_j,t_j)],
$$

where
\begin{multline}
\label{eq_prelimit_kernel}
 K^N(x_1,t_1;x_2,t_2)=-\frac{1}{2\pi i} \oint_{\Co} \frac{dw}{w^{x_1-x_2+1}}
 \prod_{t=t_2}^{t_1-1} {g_t(w)} {\mathbf
 1}_{t_1>t_2}\\
 +\frac{1}{(2\pi i)^2} \oint_{\Co} dw \oint_{\Co'_N} dz \frac{\prod_{t=0}^{t_1-1}
 g_t(w)}{\prod_{t=0}^{t_2-1} {g_t(z)}} \prod_{\ell=0}^{N-1}\frac{ (1-q^{\ell} w)}{(1-q^{\ell} z)}
 \frac{z^{x_2}}{w^{x_1+1}} \frac{1}{w-z},
\end{multline}
the contours $\Co$ and $\Co'_N$ are closed and positively oriented; $\Co$ includes only the pole
$0$ and $\Co'_N$ includes only the poles $q^{-i}$, $i=0,\dots,N-1$ of the integrand.
\end{proposition}
\begin{proof}
See Theorem 2.25, Corollary 2.26, Remark 2.27 in \cite{BF}, see also Proposition 3.4 in
\cite{BF-Push}. To match the notations, one needs to set $\alpha_\ell$ of \cite{BF} to be
$q^{1-\ell}$, $\ell\ge 1$, {set the symbol of the Toeplitz matrix $F_t(z)$ of  \cite{BF} to
$g_t(z^{-1})$}, change the integration variables via $\zeta\mapsto \zeta^{-1}$, and shift the
particles of \cite{BF} to the right by $N$.
\end{proof}

In what follows we use the standard notation
$$
(w;q)_\infty=\prod_{i=0}^{\infty}(1-wq^i).
$$

\begin{theorem}\label{th:corrfunc} For any $n\ge 1$, the $n$th correlation function $\rho_n$ of process $\mathcal Z(t)$
started from $\mathcal Z(0)=\mathbf{0}$ has the form
$$
\rho_n(x_1,t_1;x_2,t_2;\dots;x_n,t_n)=\det_{i,j=1,\dots,n}[K(x_i,t_i;x_j,t_j)],
$$
where
\begin{multline*}
 K(x_1,t_1;x_2t_2)=-\frac{1}{2\pi i} \oint_{\Co} \frac{dw}{w^{x_1-x_2+1}}
 {\prod_{t=t_2}^{t_1-1} {g_t(w)}} {\mathbf
 1}_{t_1>t_2}\\
 +\frac{1}{(2\pi i)^2} \oint_{\Co} dw \oint_{\Co'} dz \frac{\prod_{t=0}^{t_1-1}
 g_t(w)}{\prod_{t=0}^{t_2-1} {g_t(z)}} \frac{(w;q)_\infty}{(z;q)_\infty} \frac{z^{x_2}}{w^{x_1+1}}
 \frac{1}{w-z},
\end{multline*}
$\Co$ is positively oriented and includes only the pole $0$ of the integrand; $\Co'$ goes from
$+i\infty$ to $-i\infty$ between $\Co$ and point $1$.
\end{theorem}

\begin{figure}[h]
\begin{center}
\noindent{\scalebox{0.7}{\includegraphics{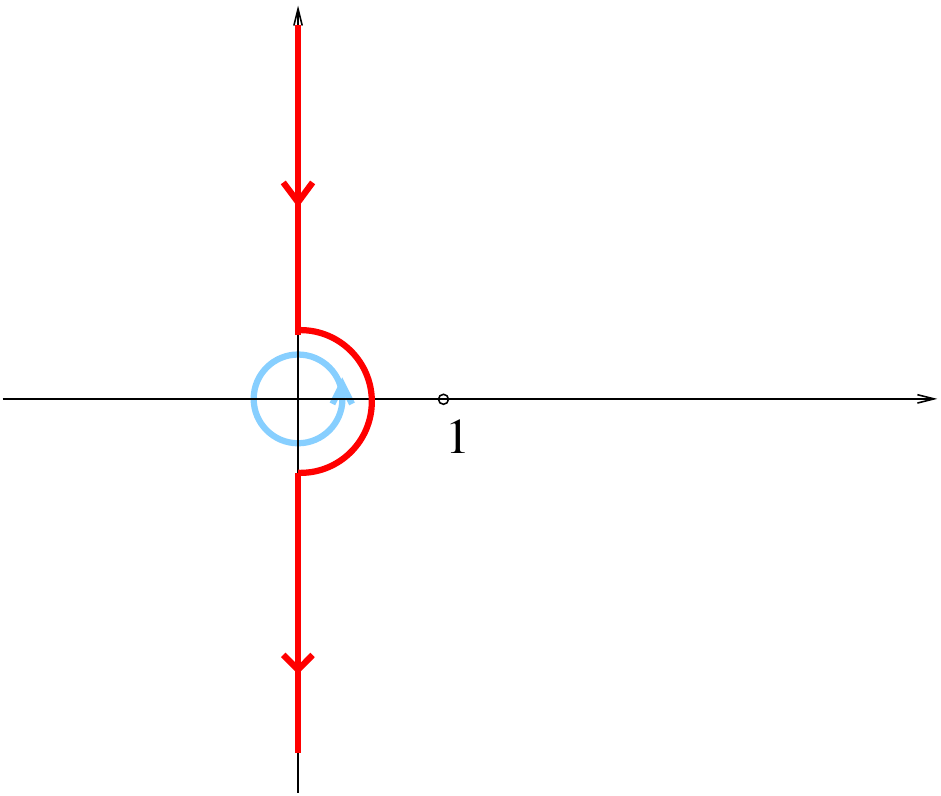}}} \caption{Contours of integration for
Theorem \ref{th:corrfunc}: $\Co$ in blue and $\Co'$ in red. \label{Figure_Contours} }
\end{center}
\end{figure}

\begin{proof}
 It is sufficient to prove that for every $(x_1,t_1;x_2,t_2)$, $K^N(x_1,t_1;x_2,t_2)\to
K(x_1,t_1;x_2,t_2)$ as $N\to\infty$.

Denote by $f_N(z,w)$ the integrand in the second (double) integral in \eqref{eq_prelimit_kernel}.
Note that if $N$ is large enough, then for every $w$, $|z^2 f_N(z,w)|$ goes to zero as
$|z|\to\infty$. Consequently, the replacement of the contour of integration $\Co'_N$ by $\Co'$
does not change the integral.

Observe that the integral
$$
  \oint_{\Co'} dz f_N(z,w)
$$
converges uniformly in $w\in\Co$ because of the rapid decay of $f_N(z,w)$ when $z\to\pm i\infty$.
Moreover, the functions $f_N(z,w)$ uniformly converge on $\Co \times \Co'$ as $N\to\infty$.
Therefore,
$$
 \oint_{\Co} dw \oint_{\Co'} dz f_N(z,w)\to \oint_{\Co} dw \oint_{\Co'}
 \frac{\prod_{t=0}^{t_1-1} g_t(w)}{\prod_{t=0}^{t_2-1} {g_t(z)}} \frac{\phi(w)}{\phi(z)}
 \frac{z^{x_2}}{w^{x_1+1}} \frac{1}{w-z}.
$$
\end{proof}

\subsection{First coordinate of the process}

\label{Subsection_first_coordinate}

In this section we give an independent interpretation for the evolution of the first coordinate of
$\mathcal Z(t)$. Similar interpretations also exist for the evolutions of first $k$ coordinates
for every $k$. Theorems of this section are based on the results of \cite[Sections 2.6-2.7]{BF}
 and we are not giving their proofs.

Although all constructions make sense for general processes introduced in the previous sections,
for simplicity of the exposition we restrict ourselves to homogeneous Markov processes
started from the delta measure at $\mathbf{0}$.

Denote by $\mathcal X_{g}(t)$ a discrete time homogeneous Markov process on $\mathcal N$ with
transition probabilities $P_\infty(u,A; g(x)) $ started from the delta measure at $\mathbf{0}$.


Also let $\mathcal Y_{\gamma_+,\gamma_-}(t)$ be continuous time homogeneous  Markov process on
$\mathcal N$ with transition probabilities $P_\infty\left(u, A; \exp(t(\gamma_+
x+\gamma_-/x))\right) $ started from the delta measure at $\mathbf{0}$.

Denote by $\mathcal X_{g}(t)_1$ and $\mathcal Y_{\gamma_+,\gamma_-}(t)_1$ projections of these
processes to the first coordinate (i.e.\ we consider the position of the leftmost particle). Note
that $\mathcal X_{g}(t)_1$ and $\mathcal Y_{\gamma_+,\gamma_-}(t)_1$ do not have to be Markov processes.

Similarly, denote by $\mathcal X_{N,g}(t)_N$ and $\mathcal Y_{N,\gamma_+,\gamma_-}(t)_N$ the
projections of processes $\mathcal X_{N,g}(t)$ and $\mathcal Y_{N,\gamma_+,\gamma_-}(t)$ to the
$N$th coordinate (which again corresponds to the position of the leftmost particle).

As in \cite{BF}, we introduce a state space $\mathcal S$ of interlacing variables
$$
 \mathcal S=\left\{\{x_k^m\}_{\stackrel{k=1,\dots,m}{m=1,2,\dots}}\subset \mathbb Z^{\infty}\mid
x^m_{k-1}<x_{k-1}^{m-1}\le x_k^m \right\}.
$$
We interpret $x_1^m<x_2^m<\dots<x_m^m$ as positions of $m$ particles at horizontal line $y=m$. An
example is shown in Figure \ref{Figure_Interlacing}.

\begin{figure}[h]
\begin{center}
\noindent{\scalebox{1.0}{\includegraphics{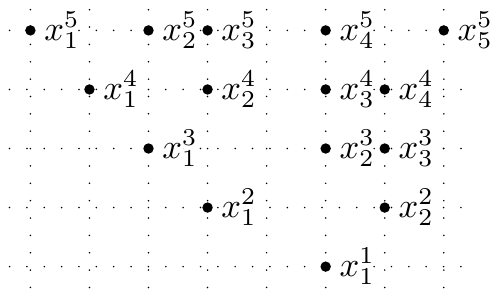}}} \caption{Interlacing particles
\label{Figure_Interlacing} }
\end{center}
\end{figure}

Introduce projection maps
$$
 \pi_N: \mathcal S \to \mathbb{M}_N \cong \mathbb{GT}_N,
$$
$$
 \{x_k^m\}_{m\ge 1,\,  1\le k\le m}\mapsto x_1^N+N-1< x_2^N+N-1<\dots<x_N^N+N-1.
$$

Let $\mathcal H$ be a set of decreasing sequences of integers:
$$
 \mathcal H=\{y_1>y_2>y_3>\dots \mid  y_i\in\mathbb Z\}.
$$

Denote by $\Pi$ the projection from $\mathcal S$ to $\mathcal H$:

$$
 \Pi(\{x_k^m\})= x_1^1> x_1^2 > x_1^3 >\dots.
$$

Finally, let $\Pi^{\infty}$ denote the map from $\mathcal S$ to $\mathbb Z \cup \{-\infty\}$ given
by
$$
 \Pi^{\infty}(\{x_k^m\})= \lim_{N\to\infty} x_1^N + N -1 = \lim_{N\to\infty} \pi_N(\{x_k^m\})[1],
$$
where by $(\,\cdot\,)[1]$ we mean a coordinate of the leftmost particle (it corresponds to $N$th
coordinate in $\mathbb{GT}_N$).

An algebraic formalism which leads to a family of Markov processes on $\mathcal S$ was introduced
in \cite[Chapter 2]{BF}. Among the processes in \cite{BF} there were processes $\mathcal
X^{\mathcal S}_{g}(t)$ and $\mathcal Y^S_{\gamma_+,\gamma_-}(t)$ such that the projections
$\pi_N(\mathcal X^{\mathcal S}_{g}(t))$ and $\pi_N(\mathcal Y^S_{\gamma_+,\gamma_-}(t))$ coincide
with $\mathcal X_{N,g}(t)$ and $\mathcal Y_{N,\gamma_+,\gamma_-}(t)$, respectively. Moreover, the
projections $\Pi(\mathcal X^{\mathcal S}_{g}(t))$ and $\Pi(\mathcal Y^S_{\gamma_+,\gamma_-}(t))$
are also Markov chains that we explicitly describe below.

The following theorem explains the connection to our processes.

\begin{theorem}
\label{theorem_identification_of_first_coordinate}
 Finite dimensional distributions of the stochastic processes $\Pi^{\infty}(\mathcal X^{\mathcal
S}_{g}(t))$ and  $\Pi^{\infty}(\mathcal Y^{\mathcal S}_{\gamma_+,\gamma_-}(t))$  coincide with
those of $\mathcal X_{g}(t)_1$ and $\mathcal Y_{\gamma_+,\gamma_-}(t)_1$.
\end{theorem}
In other words, Theorem \ref{theorem_identification_of_first_coordinate} states
that the first coordinate in the stochastic process $\mathcal
X_{g}(t)$ (or $\mathcal Y_{\gamma_+,\gamma_-}(t)$) evolves as the limiting value of coordinates of
the Markov
process $\Pi(\mathcal X^{\mathcal S}_{g}(t))$ (or $\Pi(\mathcal Y_{\gamma_+,\gamma_-}(t))$) .

Note that, in particular, Theorem \ref{theorem_identification_of_first_coordinate} guarantees that
if $\{x^k_m\}$ is distributed as $\mathcal X_{N,g}(t)_N$ or $\mathcal Y_{N,\gamma_+,\gamma_-}(t)_N$,
then the limit $\lim_{N\to\infty} (x_1^N-N+1)$ is almost surely finite.

\begin{proof}[Proof of Theorem \ref{theorem_identification_of_first_coordinate}] The proofs for
$\mathcal X_{g}(t)_1$ and $\mathcal Y_{\gamma_+,\gamma_-}(t)_1$ are the same, and we will work with
$\mathcal X_{g}(t)_1$.

 Theorem \ref{Theorem_convergence_of_joint_distributions} implies that finite dimensional
distributions of the processes $\mathcal X_{N,g}(t)_N$ converge as $N\to\infty$ to finite
dimensional distributions of $\mathcal X_{g}(t)_1$.

Therefore,
\begin{align*}
 {\rm Prob}\{&\Pi^{\infty}(\mathcal X^{\mathcal S}_{g}(t_1))\in A_1,\dots, \Pi^{\infty}(\mathcal
X^{\mathcal S}_{g}(t_k))\in A_k\}
\\&={\rm Prob}\{ \lim_{N\to\infty} \pi_N(\{x_k^m(t_1)\})[1]  \in A_1,\dots, \pi_N(\{x_k^m(t_k)\})[1]  \in A_k
\}
\\&=\lim_{N\to\infty}{\rm Prob}\{\mathcal X_{N,g}(t_1)_N\in A_1,\dots,\mathcal X_{N,g}(t_k)_N\in
A_k\}
\\&={\rm Prob}\{\mathcal X_{g}(t_1)_1\in A_1,\dots,\mathcal X_{g}(t_k)_1\in A_k\}.
\end{align*}
\end{proof}

It remains to describe the processes $\Pi(\mathcal X^{\mathcal S}_{g}(t))$ and $\Pi(\mathcal
Y^S_{\gamma_+,\gamma_-}(t))$. We have $\mathcal X^{\mathcal S}_{g}(0)=\mathcal
Y^S_{\gamma_+,\gamma_-}(0)=(0>-1>-2>-3>\dots)$. We view coordinates of the process as positions of
particles in $\mathbb Z$. Thus, $\mathcal X^{\mathcal S}_{g}(0)=\mathcal
Y^S_{\gamma_+,\gamma_-}(0)$ is the densely packed configuration of particles in $\mathbb{Z}_{\le
0}$.

Let us start from $\Pi(\mathcal X^{\mathcal S}_{g}(t))$. The description depends on $g(x)$, and we
discuss only the cases $g(x)=1+\beta x$ and $g(x)=1+\beta x^{-1}$; for other possibilities we
refer to \cite[Section 2.6]{BF}.

 Given a configuration $\{y_i\}$ at time
moment $t$, to construct a configuration at moment $t+1$ that we denote by $\{z_i\}$, we perform a
sequential update from right to left.

For $g(x)=1+\beta x$, the first particle jumps to the right by one with probability
$p_1:=\beta/(1+\beta)$ and stays put with probability $1-p_1$. In other words, $z_1=y_1$ with probability
$1-p_1$ and $z_1=y_1+1$ with probability $p_1$. For particle number $k$ we do the following: If
$y_k=z_{k-1}-1$, then this particle is \emph{blocked} and we set $z_k=y_k$. Otherwise, the
particle jumps to the right with probability $p_k:=q^{1-k}\beta/(1+q^{1-k}\beta)$ and stays put with
probability $1-p_k$.

For $g(x)=1+\beta/x$, the first particle jumps to the left by one with probability $p_1:=\beta/(1+\beta)$ and
stays put with probability $1-p_1$. For particle number $k$ we do the following: If $y_k=z_{k-1}$, then
this particle is forced to jump to the left and we set $z_k=y_{k}-1$ (one might say that
particle number $k$ was pushed by particle number $k-1$). Otherwise, the particle jumps to the
left by one with probability $p_k=q^{k-1}\beta/(1+q^{k-1}\beta)$ and stays put with probability $1-p_k$.

Observe that the above two update rules are closely related. Indeed, $1+\beta x=
(\beta x) (1+\beta^{-1}/x)$. From the probabilistic viewpoint, this equality means that one
stochastic step for $g(x)=1+\beta x$ is equivalent to the composition of the stochastic step with
$g(x)=1+\beta^{-1}/x$ and deterministic shift $y_i\to y_i+1$ for all $i\ge 1$.

For continuous time Markov processes $\Pi(\mathcal Y^S_{\gamma_+,\gamma_-}(t))$ the interpretation
is quite similar. Particle number $k$ has two exponential  clocks,  ``right clock'' and ``left
clock'' with parameters $q^{1-k}\gamma_+$ and $q^{k-1}\gamma_-$, respectively. These two numbers
are \emph{intensities} of the right and left jumps. When the right clock of a particle number $k$
rings, it checks whether the position to its right is occupied (i.e.\ if $y_k=y_{k-1}-1$). If the
answer is ``Yes'' then nothing happens, otherwise the particle jumps to the right by one. When the left
clock of the particle number $k$ rings, then this particle jumps to the left and pushes the (maybe
empty) block of particles sitting next to it.

The processes above with one-sided jumps are versions of the totally asymmetric simple exclusion
process (known as TASEP) and long range TASEP, cf. \cite{Spitzer}, \cite{L1} \cite{L2}. The
processes with two-sided jumps was defined and studied in \cite{BF-Push} under the name of
PushASEP.

\section{The Feller property}
\label{Section_feller}

In this section we show that transition probabilities
$$P_{\infty}(u;A;\exp(t(\gamma_+x+\gamma_-/x)))$$ of the Markov process $\mathcal
Y_{\gamma_+,\gamma_-}(t)$ satisfy the Feller property. The notion of a Feller process makes sense
only for the processes in a locally compact space, and this is the reason why we embed $\mathcal
N$ into the bigger locally compact space $\overline{\mathcal N}$ as in Section
\ref{Section_Infinite-dimensional} above.

\subsection{Extended boundary}
The aim of this section is to identify the local compactification $\overline{\mathcal N}$ of $\mathcal N$
with a boundary of the sequence of sets $\overline{\mathbb{GT}_N}$ and
links $\overline{P}^\downarrow_N(\lambda\to\mu)$ with the sequence of parameters
$\xi_i=q^{1-i}$, $i\ge 1$. These sets and links were introduced in the second part of Section
\ref{Section_Commutation_relations}.

Let {$M=(M_0,M_1,\dots)$} be an element of $\plim \mathcal
M(\overline{\mathbb{GT}_N})$. We say that $M$ is of class $k$ if for $N>k$ the support of $M_N$ is
a subset of $\mathbb{GT}_N^{(k)}$. We say that $M$ is of class $\infty$ if for any $N$ the support
of $M_N$ is a subset of $\mathbb{GT}_N^{(N)}$.

\begin{proposition}
\label{proposition_decomposition_of_ext_coherent_system} For any $M\in \plim \mathcal
M(\overline{\mathbb{GT}_N})$ there exist unique $M^{\infty}, M^{0},M^1,M^2,\dots\in \plim \mathcal
M(\overline{\mathbb{GT}_N})$ such that $M^i$ is of class $i$ and $M=M^{\infty}+M^0+M^1+M^2+\dots$.

If $M$ is a nonnegative then $M^{\infty}, M^{0},M^1,M^2,\dots$ are also
nonnegative.
\end{proposition}

\begin{proof}

Decompose $M_N$ into the sum
$$
 M_N=M_N^{(0)}+M_N^{(1)}+\dots+M_N^{(N)}
$$
with
$$
 \mathrm{supp}(M_N^{(k)})\subset \mathbb{GT}_N^{(k)}.
$$
Clearly, such decomposition is unique and \begin{equation} \label{eq_x1}
 \Vert M_N\Vert =\sum_k \Vert M_N^{(k)}\Vert .
\end{equation}

Note that
$$
 \overline{P}^{\downarrow}_N M_N^{(k)} =M_{N-1}^{(k)}
$$
for $k\le N-2$. Set $M^k_N=M^{(k)}_N$ for $N>k$, $M^k_k=\overline{P}^{\downarrow}_{k+1}
M^{k}_{k+1}$, $M^{k-1}_k=\overline{P}^{\downarrow}_{k} M_k^k$, and so on. One proves that for every
$k$, $M^k=(M^k_n)_{n\ge 1}$ is of class $k$.

Furthermore, \eqref{eq_x1} yields that $\sum_{k=0}^{r}\Vert M_N^{k}\Vert \le \Vert M_{r+1}\Vert
\le \Vert M\Vert $. Consequently, the sum $\sum_{k=0}^{\infty} M^k$ is well defined. Set
$M^{\infty}=M-\sum_{k=0}^{\infty} M^k$. It follows that $M^{\infty}$ is of class $\infty$ and,
thus,
$$
 M=M^{\infty}+\sum_{k=0}^{\infty} M^k.
$$
The uniqueness of the decomposition and nonnegativity are immediate.
\end{proof}

Now we fix $k$ and aim to describe the set of all class $k$ elements of $\plim \mathcal
M(\overline{\mathbb{GT}_N})$. Observe that these are just elements of $\plim \mathcal M(\Gamma_N)$
with $\Gamma_N=\mathbb{GT}_N^{(\min(k,N))}$ with links given by the restrictions of matrices
$\overline{P}_N^\downarrow(\lambda\to\mu)$. Thus, we may use Theorem \ref{theorem_boundary}.
Therefore, it remains to identify the set of all extreme coherent systems, i.e.\ the extreme points of
$\plim \mathcal M_p(\mathbb{GT}_N^{(\min(k,N))})$.

\begin{theorem}
\label{theorem_finite_dim_boundary} The extreme points of $\plim \mathcal
M_p(\mathbb{GT}_N^{(\min(k,N))})$ are enumerated by signatures $\lambda\in\mathbb{GT}_k$. Let
$\mathcal E^{\lambda}$ be an element of $\plim \mathcal M_p(\mathbb{GT}_N^{(\min(k,N))})$
corresponding to $\lambda$. Then for $N\ge k$, the Schur generating function of the measure
$\mathcal E^{\lambda}_N$ is
$$
 \mathcal S(x_1,\dots,x_k; \mathcal E^\lambda_N)= \frac{s_\nu(x_1,\dots,
x_k)}{s_\nu(1,q^{-1},\dots, q^{1-k})}\prod_{n=N+1}^{\infty} G_n^k,
$$
and for $N\in \{1,\dots,k-1\}$ we have
\begin{multline*}
 \mathcal S(x_1,\dots,x_N; \mathcal E^\lambda_N)\\= \frac{s_\nu(x_1,\dots,
x_N,q^{1-N},\dots,q^{1-k})}{s_\nu(1,q^{-1},\dots, q^{1-k})}\prod_{n=k+1}^{\infty} G_n^k
(x_1,\dots, x_N,q^{1-N},\dots,q^{1-k}),
\end{multline*}
where
$$
 G_N^{\ell}(x_1,\dots,x_\ell)=\prod_{i=1}^\ell \frac{1-q^{N-1}q^{1-i}}{1-q^{N-1}x_i}.
$$
{For $k=0$, the unique element of $\plim \mathcal M_p(\mathbb{GT}_N^{(\min(k,N))})$
is $\mathcal E^{\varnothing}$, such that for every $N$, $\mathcal E^{\varnothing}_N$ is the unique
probability measure on the singleton $\mathbb{GT}_N^{(0)}$.}
\end{theorem}

\begin{proof}
{The case $k=0$ is obvious. We use the general \emph{ergodic method} to prove this
theorem for $k\ge 1$ (see \cite{V_ergodic}, \cite[Section 6]{OkOlsh} and also \cite[Theorem
1.1]{DF}).}

Choose $\lambda\in\mathbb{GT}^{(k)}_N$. There exists a unique system of measures
$M^{\lambda}_1,\dots,M^{\lambda}_N$ such that for every $i\le N$, $M^{\lambda}_i$ is a measure on
$\mathbb{GT}^{(\min(k,N))}_i$, $M^{\lambda}_i=\overline{P}_{i+1}^\downarrow M^{\lambda}_{i+1}$, and
$M^{\lambda}_N$ is the delta-measure at $\lambda$. We call such a system of measures a
\emph{primitive} system of $\lambda$. The ergodic method states that every extreme coherent system
$M\in \plim \mathcal M_p(\mathbb{GT}_N^{(\min(k,N))})$ is a weak limit of primitive systems. In
other words, there exist sequences $N_i\to\infty$ and
$\lambda^i\in\mathbb{GT}^{(\min(N_i,k))}_{N_i}$ such that for every $N$ and every
$\mu\in\mathbb{GT}^{(\min(k,N))}_N$, we have
$$
 M_N(\mu)=\lim_{i\to\infty} M_N^{\lambda^i}(\mu).
$$

According to the definition of links $\overline{P}_{i+1}^\downarrow$, the Schur generating
function of the measure $M_N^{\lambda^i}$ has the following form for $N_i\ge N \ge k$:

$$
 \mathcal{S}(x_1,\dots,x_k;
M_N^{\lambda^i})=\frac{s_{\lambda^i}(x_1,\dots,x_k)}{s_{\lambda^i}(1,q^{-1},\dots,q^{1-k})}\prod_{n=N+1}^{N_i}
G_n^k.
$$

Proposition 4.1 of \cite{Gor} implies that the weak convergence of $M_N^{\lambda^i}(\mu)$ as $i\to\infty$ is
equivalent to the uniform convergence of the Schur generating functions of measures
$M_N^{\lambda^i}$ to those of $M_N$. Observe that
$$
 \prod_{n=N+1}^{N_i} G_n^k \to \prod_{n=N+1}^{\infty} G_n^k
$$
uniformly on $T_k$ as $i\to\infty$. Therefore, functions
$$
 \frac{s_{\lambda^i}(x_1,\dots,x_k)}{s_{\lambda^i}(1,q^{-1},\dots,q^{1-k})}
$$
should also uniformly converge. But this happens if and only if the sequence of signatures
$\lambda^i$ stabilize to a certain $\lambda\in\mathbb{GT}_k$. Then measures $M_N^{\lambda^i}$
converge precisely to $\mathcal E_N^{\lambda}$.

Thus, coherent systems $\mathcal E^{\lambda}$ contain all extreme points of $\plim \mathcal
M_p(\mathbb{GT}_N^{(\min(k,N))})$. It remains to prove that all of them are indeed extreme. But
this follows from linear independence of the measures $\mathcal E_N^{\lambda}$ which, in turn,
follows from linear independence of Schur polynomials $s_{\lambda}(x_1,\dots,x_k)$.
\end{proof}

Now we are ready to describe the map $E$ from $\mathcal M(\overline{\mathcal N})$ to $\plim
M(\overline{\mathbb{GT}_N})$, cf. Definition \ref{definition_boundary}. Let $\pi$ be a finite signed measure
on $\overline{\mathcal N}$.
There is a unique decomposition
$$
 \pi=\pi^{\infty}+\pi^{0}+\pi^1+\pi^2+\dots
$$
such that $\textrm{supp}(\pi^k)\subset \mathbb{GT}_N\subset\overline{\mathcal N}$ and
$\textrm{supp}(\pi^{\infty})\subset \mathcal N\subset \overline{\mathcal N}$.

By Theorems \ref{theorem_boundary} and \ref{theorem_qGT_boundary}, $\pi^{\infty}$ corresponds to a
unique element of $\plim \mathcal M(\mathbb{GT}_N)$ which can be viewed as an element $M^{\infty}$
of $\plim \mathcal M(\overline{\mathbb{GT}_N})$. Similarly, by Theorems \ref{theorem_boundary} and
\ref{theorem_finite_dim_boundary}, $\pi^k$ corresponds to a unique element of $\plim \mathcal
M(\mathbb{GT}_N^{(\min(k,N)})$ which can be viewed as an element $M^k$ of $\plim \mathcal
M(\overline{\mathbb{GT}_N})$. Note that {
$$
 \Vert M^\infty\Vert +\Vert M^0\Vert +\Vert M^1\Vert +\dots=\Vert \pi^\infty\Vert +\Vert \pi^0\Vert +\Vert \pi^1\Vert +\dots<\infty.
$$}
Therefore, we may define {
$$
 M:=M^{\infty}+M^0+M^1+\dots.
$$}
Set $E(\pi)=M$.

\begin{theorem}
\label{Theorem_extended_boundary_1}
 The set $\overline{\mathcal N}$ and the map $E$ satisfy the first 3 conditions of Definition
\ref{definition_boundary}.
\end{theorem}

\begin{proof}

\begin{enumerate}

\item Proposition \ref{proposition_decomposition_of_ext_coherent_system} implies that $E$ is surjective.
$E$ is injective by the construction.
  \item  $E$ is a direct sum of norm $1$ operators, thus it is a norm
$1$ operator. Furthermore, $E$ is a bijection, thus $E^{-1}$ is also bounded.
  \item Second part of Proposition \ref{proposition_decomposition_of_ext_coherent_system} guarantees
   that $E$ maps $\mathcal M_p(\overline{\mathcal N})$ bijectively onto $\plim \mathcal
M_p(\overline{\mathbb{GT}_N})$,

\end{enumerate}
\end{proof}

As for the fourth condition of Definition \ref{definition_boundary}, its proof is
nontrivial and we present it as a separate theorem in the next section.

\subsection{The topology of the extended boundary}

In this section we prove the following theorem.
\begin{theorem}
\label{theorem_equivalence_of_topologies}
 The map
 $$
  \nu\to\mathcal E^\nu=E(\delta^\nu)
 $$
 from $\overline{\mathcal N}$ to $\plim \mathcal M(\overline{\mathbb{GT}_N})$ (equipped with weak
topology) is a homeomorphism on its image.\footnote{Note that the image of this map consists of the extreme points of $\plim
\mathcal M_p(\overline{\mathbb{GT}_N})$. }
\end{theorem}

We present a proof in a series of lemmas.

Let $a[1..N]$ be an element of $\mathbb M_N$, i.e.\ $a[1..N]$ is the $N$-point subset of $\mathbb Z$
consisting of points $a[1]<a[2]<\dots<a[N]$. Set
$$
 A_N^a(x_1,\dots,x_N)=Alt( x_1^{a[1]}\cdots x_N^{a[N]})=\sum_{\sigma\in S(N)} (-1)^{\sigma}
x_1^{a[\sigma(1)]}\cdots x_N^{a[\sigma(N)]}.
$$

\begin{lemma}
\label{lemma_factorization} The following factorization property holds:
\begin{multline}
 \frac{s_\lambda(x_1,\dots,x_N)}{s_\lambda(1,\dots,q^{1-N})}=
 \frac{s_{\lambda_{N-k+1},\dots,\lambda_N}(x_1,\dots,x_k)
 s_{(\lambda_1+k,\dots,\lambda_{N-k}+k)}(x_{k+1},\dots,x_N)}{\prod_{i=1,\dots,k}\prod_{j=k+1,\dots,N}
 (x_i-x_j)} \\ \times \frac{\prod_{i=1,\dots,k}\prod_{j=k+1,\dots,N}
 (q^{1-i}-q^{1-j})}{s_{\lambda_{N-k+1},\dots,\lambda_N}(1,\dots,q^{1-k})
 s_{(\lambda_1+k,\dots,\lambda_{N-k}+k)}(q^{-k},\dots,q^{1-N})}+Q,
\end{multline}
where if we expand $Q$ as a sum of monomials
$$
 Q=\sum_{m_1,\dots,m_N} c^{m_1,\dots,m_N} x_1^{m_1}\cdots x_N^{m_N},
$$
then
$$
 \sum_{m_1,\dots,m_N} |c^{m_1,\dots,m_N}| 1^{m_1}\cdots q^{(1-N)m_N}
<R(k,\lambda_{N-k},\lambda_{N-k+1}),
$$
and for any fixed $k$ and bounded $\lambda_{N-k+1}$, $R\to 0$ as $\lambda_{N-k}\to\infty$.

\end{lemma}
{\bf Remark 1.} The statement should hold for more general sequences of $\xi$'s, but our proof
works only for geometric progressions.

\begin{proof} We have (the first equality is the definition of the Schur polynomial)
\begin{multline*}
 (-1)^{N(N-1)/2} s_\lambda(x_1,\dots,x_N)= \frac{Alt(x_1^{\lambda_N}x_2^{\lambda_{N-1}+1}\cdots
 x_N^{\lambda_1+N-1})}{\prod_{i<j}(x_i-x_j)}\\=\frac{Alt(x_1^{\lambda_N}\cdots
 x_k^{\lambda_{N-k+1}+k-1})Alt(x_{k+1}^{\lambda_{N-k}+k}\cdots
 x_N^{\lambda_1+N-1})}{\prod_{i<j}(x_i-x_j)}+B.
\end{multline*}
Let us estimate the remainder $B$. By the definitions
\begin{align}
\label{eq_x11} \nonumber B=&\sum_{\{j(1),\dots, j(k)\}\ne \{N-k+1,\dots,N\}} \pm
\dfrac{x_1^{\lambda_{j(1)}+N-j(1)}\cdots
 x_k^{\lambda_{j(k)}+N-j(k)}}{\prod_{i<j\le k}(x_i-x_j)}\\ \nonumber &\times
 \dfrac{A_{N-k}^{(\lambda+\delta)\setminus
 \{\lambda_{j(1)}+N-j(1),\dots,\lambda_{j(k)}+N-j(k)\}}(x_{k+1},\dots,x_N)}{\prod_{k<i<j}(x_i-x_j)}
 \\& \times \frac{1}{\prod_{i=1,\dots,k}\prod_{j=k+1,\dots,N} (x_i-x_j)},
\end{align}
where  $(\lambda+\delta)\setminus
 \{\lambda_{j(1)}+N-j(1),\dots,\lambda_{j(k)}+N-j(k)\}$ stands for the $(N-k)$-element subset of $\mathbb
Z$ which is the set-theoretical difference of
$\lambda+\delta=\{\lambda_1+N-1,\lambda_2+N-2+\dots,\lambda_{N-1}+1,\lambda_N\}$ and
$\{\lambda_{j(1)}+N-j(1),\dots,\lambda_{j(k)}+N-j(k)\}.$

Denote the three factors in \eqref{eq_x11} by $B_1$, $B_2$ and $B_3$, respectively. For any
function $L$ on $T_N$ that has a Laurent expansion, let $Est(L)$ be the following number: Decompose $L$ into the sum of
monomials
$$
 L=\sum_{m_1,\dots,m_N} \ell^{m_1,\dots,m_N} x_1^{m_1}\cdots x_N^{m_N}
$$
and set
$$
 Est(L)=\sum_{m_1,\dots,m_N} |\ell^{m_1,\dots,m_N}| 1^{m_1}\cdots q^{(1-N)m_N}.
$$
Clearly, $Est(L_1L_2)\le Est(L_1)Est(L_2)$. Therefore,
$$
 Est(B)\le Est(B_1)Est(B_2)Est(B_3).
$$
In what follows $const_k$ denotes various constants depending solely on $k$. We have

$$
Est(B_1)\le const_k \prod_{m=1}^k q^{(1-m)(\lambda_{j(m)}+N-j(m))}.$$

  Observe that $B_2$ is a Schur polynomial. As follows from the combinatorial formula(see e.g.\
Section I.5 of \cite{Mac}), these functions are sums of monomials with non-negative
 coefficients, thus $Est(B_2)=B_2(1,q^{-1},\dots,q^{1-N})$.

 Finally, decomposing the denominators in $B_3$ into geometric series and then converting them
back, we get
$$
 Est(B_3)\le \left(\prod_{i=1}^{k} \prod_{j=k+1,\dots,N} (q^{1-j}-q^{1-i})^{-1}\right)\le const_k
 \prod_{j=k+1}^N q^{(j-1)k}.
$$

Now set $a[N-j+1]=\lambda_j+N-j$. We have

\begin{multline*}
 Est(B)\le \sum_{\{j(1),\dots j(k)\}\ne \{1,\dots,k\}} const_k \left(\prod_{j=k+1}^N q^{(j-1)k}
 \prod_{i=1}^k q^{(1-i)a[j(i)]}\right)
 \\ \times
 A_{N-k}^{a[1..N] \setminus \{a[j(1)],\dots, a[j(k)]\}}(q^{-k},\dots,q^{1-N})
 \frac{1}{\prod_{i<j\le N-k} (q^{-k-i+1}-q^{-k-j+1})}
\end{multline*}

Let $ C(j(1),\dots,j(k))$ denote the right-hand side of the above inequality.
Then
$$
\frac{
 C(j(1),\dots,j(k))}{C(1,\dots,k)}= \prod_{i=1}^k q^{(1-i)(a[j(i)]-a[i])} \dfrac{A_{N-k}^{a[1..N]
 \setminus \{a[j(1)],\dots,
 a[j(k)]\}}(q^{-k},\dots,q^{1-N})}{A_{N-k}^{a[k+1..N]}(q^{-k},\dots,q^{1-N})}\,.
$$

For any increasing sequence $b[1..N-k]$ we have
\begin{align*}
A_{N-k}^{b[1..N-k]}&(q^{-k},\dots,q^{1-N})=q^{(1-N)(b[1]+\dots+b[N-k])}A_{N-k}^{b[1..N-k]}(1,\dots,q^{N-k-1})
 \\&\stackrel{(*)}{=}q^{(1-N)(b[1]+\dots+b[N-k])} \prod_{i<j\le (N-k)} (q^{b[i]}-q^{b[j]})
 \\&=q^{(1-N)(b[1]+\dots+b[N-k])+(N-k-1)b[1]+(N-k-2)b[2]+\dots+0b[N-k]} \\ &\quad\times \prod_{i<j\le (N-k)} \left(1-q^{b[j]-b[i]}\right)
 \\&=q^{(-k)b[1]+(-k-1)b[2]+\dots+(1-N)b[N-k]} \prod_{i<j\le (N-k)} (1-q^{b[j]-b[i]}),
\end{align*}
where the equality $(*)$ is the Vandermonde determinant evaluation.

To analyze $A_{N-k}^{a[1..N] \setminus \{a[j(1)],\dots, a[j(k)]\}}$ we think of the set
$$
b[1..N-k]:=a[1..N] \setminus
 \{a[j(1)],\dots, a[j(k)]\}
$$ as of a small modification of the set $a[k+1..N]$. Note that
under this modification only finite number of members  (up to $k$) of the sequence $b[i]$ change.
Using the finiteness of $\prod_{n\ge 1} (1-q^n)$, one easily sees that the `modified' product
$\prod_{1\le i<j\le (N-k)} (1-q^{b[j]-b[i]})$ differs from the `unmodified' one $\prod_{k+1\le
i<j\le N} (1-q^{a[j]-a[i]})$ by a constant that is bounded away from 0 and $\infty$ (note that $k$
is fixed, while $N$ can be arbitrarily large).

Hence,
\begin{multline}
\label{eq_x12}
 \dfrac{A_{N-k}^{a[1..N] \setminus \{a[j(1)],\dots,
 a[j(k)]\}}(q^{-k},\dots,q^{1-N})}{A_{N-k}^{a[k+1..N]}(q^{-k},\dots,q^{1-N})}
 \\ \le const_k\cdot q^{(-k)b[1]+(-k-1)b[2]+\dots+(1-N)b[N-k]} q^{ka[k+1]+\dots+(N-1)a[N]}
\end{multline}

 The next step is to estimate the exponent of $q$ above by
\begin{multline*}
 k(a[k+1]-b[1])+\dots+(N-1)(a[N]-b[N-k]) \\\ge
k\sum_{m=1}^{N-k} (a[m+k]-b[m])\ge
k\sum_{m:j(m)>k}(a[j(m)]-a[k]).
\end{multline*}

Hence,
\begin{multline*}
\frac{C(j(1),\dots,j(k))}{C(1,\dots,k)}\le const_k \prod_{i=1}^k q^{(1-i)(a[j(i)]-a[i])}
q^{k\sum_{m:j(m)>k}(a[j(m)]-a[k])}
\\ \le const(k,a[k]) \cdot q^{\sum_{m:j(m)>k}a[j(m)]}.
\end{multline*}
Here and below we use $const(k,a[k])$ to denote any constant depending only on $k$ and $a[k]$.


Thus, the sum over all $\{j(1),\dots,j(k)\}\ne\{1,\dots,k\}$ can be bounded by geometric series
and
$$
 Est(B)\le C(1,\dots,k)\cdot const(k,a[k])\, q^{a[k+1]}.
$$

Hence, substituting the definition of $C(1,\dots,k)$ and using \eqref{eq:princ_spec},
\begin{align*}
 Est&\left(\frac{B}{s_\lambda(1,\dots,q^{1-N})}\right)\\&\le const(k,a[k]) q^{a[k+1]}
 \prod_{j=k+1}^N q^{(j-1)k} \prod_{i=1}^k q^{(1-i)(\lambda_{N+1-i}+i-1)}
 \\ &\quad\times
 q^{(-k)(\lambda_{N-k}+k)+\dots+(1-N)(\lambda_1+N-1)} \prod_{1\le i<j\le N-k}
 \frac{1-q^{\lambda_i-i-\lambda_j+j}} {q^{-k-i+1}-q^{-k-j+1}}
 \\ &\quad\times
 q^{(N-1)(\lambda_1+\dots+\lambda_N)} q^{-0\lambda_1-\dots-(N-1)\lambda_N} \prod_{1\le i<j\le N}
 \frac{1-q^{j-i}}{1-q^{\lambda_i-i-\lambda_j+j}}
 \\
 &=const(k,a[k]) q^{a[k+1]} \prod_{j=k+1}^N q^{(j-1)k} \\ &\quad\times \prod_{i=1}^N q^{-(i-1)^2}
 \prod_{1\le i<j\le(N-k)} \frac{1-q^{j-i}}{q^{-k-i+1}-q^{-k-j+1}}\\ &\quad \times
 \prod_{i=1}^{N-k}\prod_{j=N-k+1}^{N} \frac {1-q^{j-i}}{1-q^{\lambda_i-i-\lambda_j+j}}
 \prod_{(N-k+1)\le i<j\le N} \frac{1-q^{j-i}}{1-q^{\lambda_i-i-\lambda_j+j}}
 \\
 &\le const(k,a[k]) q^{a[k+1]} \prod_{j=k+1}^N q^{(j-1)k} \prod_{i=1}^N q^{-(i-1)^2} \prod_{1\le
 i<j\le(N-k)} q^{k+j-1}
 \\&=
const(k,a[k]) q^{a[k+1]} \prod_{j=k+1}^{N} q^{(j-1)k} \prod_{i=1}^N q^{-(i-1)^2} \prod_{j=k+1}^{N}
q^{(j-1)(j-1-k)}
\\&=const(k,a[k]) q^{a[k+1]}
\end{align*}
The statement of the lemma immediately follows.
\end{proof}

{
\begin{lemma}
\label{lemma_total_variation_estimate}
 Let $P$ be a signed finite measure on $\mathbb{GT}_N$ with Schur generating function $\mathcal
 S$:
 $$
  \mathcal S (x_1,\dots,x_N; P)=\sum_{\lambda\in\mathbb{GT}_N}
  P(\lambda)\frac{s_\lambda(x_1,\dots,x_N)}{s_\lambda(1,\dots,q^{1-N})}.
 $$
 The total variation norm of $P$ can be estimated as
 $$
  \Vert P \Vert \le const_N Est(\mathcal S).
 $$
\end{lemma}
\begin{proof}
 As before, we use the notation $const_N$ below for various positive constants depending on $N$.
 We have
 $$
  Est\left(\mathcal S \prod_{i<j} (x_i-x_j)\right) = \sum_{\lambda\in\mathbb{GT}_N}
  \frac{|P(\lambda)|}{s_\lambda(1,\dots,q^{1-N})} Est\left(Alt(x_1^{\lambda_1+N-1}\cdots
  x_N^{\lambda_N})\right).
 $$
 Note that
  $$Est\left(Alt(x_1^{\lambda_1+N-1}\cdots
  x_N^{\lambda_N})\right)\ge q^{0\lambda_N} q^{(-1)(\lambda_{N-1}+1)}\cdots
  q^{(1-N)(\lambda_1+N-1)}.
 $$
 Also, using \eqref{eq:princ_spec}
 $$
  {s_\lambda(1,\dots,q^{-N})} \le {\frac{1}{const_N}} \,{q^{0\lambda_N} q^{(-1)(\lambda_{N-1}+1)}\cdots
  q^{(1-N)(\lambda_1+N-1)}}
 $$

 Therefore,
 $$
   Est\left(\mathcal S \prod_{i<j} (x_i-x_j)\right) \ge {const_N}\sum_{\lambda\in\mathbb{GT}_N}
  |P(\lambda)| = {const_N} \Vert P \Vert.
 $$
 It remains to observe that
 $$
  Est\left(\mathcal S \prod_{i<j} (x_i-x_j)\right)\le Est(\mathcal S) Est\left( \prod_{i<j} (x_i-x_j)\right) \le const_N
  Est(\mathcal S).
 $$
\end{proof}
}

\begin{lemma}
\label{lemma_convergence_on_the_boundary_case1} Let $\nu^n$ be a sequence of points of ${\mathcal
N}\subset \overline{\mathcal N}$ converging to $\nu\in \overline{\mathcal N}\setminus\mathcal N$.
Then $\mathcal E^{\nu^n}$ weakly converges to $\mathcal E^\nu$.
\end{lemma}

\begin{proof} Since $\nu\in \overline{\mathcal N}\setminus\mathcal N$, there exists $k\ge 1$ such that
 $\nu\in\mathbb{GT}_k\subset \overline{\mathcal N}$. First, take $m>k$. Recall that the
measure ${\mathcal E}^{\nu}_{m}$ is supported by $\mathbb{GT}_m$ and ${\mathcal S}(x_1,\dots,x_m;
{\mathcal E}^{\nu}_{m})$ is its Schur generating function.

Theorem 1.3 of \cite{Gor} yields
$$
 {\mathcal S}(x_1,\dots,x_m; {\mathcal E}^{\nu^n}_{m})=\lim_{N\to\infty}
  S_N^{\nu^n}(x_1,\dots,x_m),
$$
where
$$
  S_N^{\nu^n}(x_1,\dots,x_m)
=\dfrac{s_{(\nu_{N}^n,\dots,\nu_1^n)}(x_1,\dots,x_m,q^{-m},\dots,q^{1-N})}{s_{(\nu_{N}^n,\dots,\nu_1^n)}(1,\dots,q^{1-N})}.
$$


Since $\nu^n\to\nu$, we have $\nu^n_{k+1}\to\infty$, while $\nu^n_1,\dots,\nu^n_{k}$ stabilize. Thus, we
may use Lemma \ref{lemma_factorization}
\begin{align}
\nonumber&\dfrac{s_{(\nu_{N}^n,\dots,\nu_1^n)}(x_1,\dots,x_m,q^{-m},\dots,q^{1-N})}{s_{(\nu_{N}^n,\dots,\nu_1^n)}(1,\dots,q^{1-N})}
 \\&\quad\nonumber=\dfrac{s_{(\nu_{k}^n,\dots,\nu_1^n)}(x_1,\dots,x_k)}{s_{(\nu_{k}^n,\dots,\nu_1^n)}(1,\dots,q^{1-k})}
  \frac{\prod_{i=1,\dots,k}\prod_{j=m+1,\dots,N}
  (q^{1-i}-q^{1-j})}{\prod_{i=1,\dots,k}\prod_{j=m+1,\dots,N} (x_i-q^{1-j})}
 \\&\quad\nonumber \times
  \dfrac{s_{(\nu^n_n+k,\dots,\nu^n_{k+1}+k)}(x_{k+1},\dots,x_m,q^{-m},\dots,q^{1-N})}
  {s_{(\nu^n_n+k,\dots,\nu^n_{k+1}+k)}(q^{-k},\dots,q^{1-N})}
 \\&\quad \times
  \frac{\prod_{i=1,\dots,k}\prod_{j=k+1,\dots,m}
  (q^{1-i}-q^{1-j})}{\prod_{i=1,\dots,k}\prod_{j=k+1,\dots,m} (x_i-x_j)} +Q. \label{eq_x13}
\end{align}

Note that the measure with Schur generating function  $S_N^{\nu^n}(x_1,\dots,x_m)$ is supported on
signatures $\lambda$ such that $\lambda_{m-k}\ge\nu^n_{k+1}$. Indeed, this readily follows from
formulas \eqref{eq_interlacing} and \eqref{eq_explicit_link}, see also Proposition 5.5 in
\cite{Gor}.

 Let us find the projection of the measure with Schur generating function
$S_N^{\nu^n}(x_1,\dots,x_m)$ to the last $k$ coordinates
 of the signature (all other coordinates tend to infinity as $n\to\infty$). We claim that the
Schur generating function of this projection uniformly (in $x$'s and in $N$) tends (as
$n\to\infty$) to
\begin{equation}
\label{eq_x2}
\dfrac{s_{(\nu_{k}^n,\dots,\nu_1^n)}(x_1,\dots,x_k)}{s_{(\nu_{k}^n,\dots,\nu_1^n)}(1,\dots,q^{1-k})}
\frac{\prod_{i=1,\dots,k}\prod_{j=m+1,\dots,N}
(q^{1-i}-q^{1-j})}{\prod_{i=1,\dots,k}\prod_{j=m+1,\dots,N} (x_i-q^{1-j})}\,.
\end{equation}

 To prove this claim, we first expand \eqref{eq_x2} into a \emph{finite} sum of normalized Schur
 polynomials
$$
 \sum_{\lambda\in \Lambda} c_\lambda^N \frac{s_\lambda(x_1,\dots,x_k)}{s_\lambda(1,\dots,q^{1-k})}
+ Q_1
$$
It is clear, that for any $\varepsilon$ we may choose such a finite set $\Lambda$ not depending on
$N$, that $Est(Q_1)<\varepsilon$.


We also expand
$$
  \dfrac{s_{(\nu^n_n+k,\dots,\nu^n_{k+1}+k)}(x_{k+1},\dots,x_m,q^{-k},\dots,q^{1-N})}
  {s_{(\nu^n_n+k,\dots,\nu^n_{k+1}+k)}(q^{-k},\dots,q^{1-N})}
$$
into the \emph{full} sum of normalized  Schur polynomials in variables $x_{k+1},\dots,x_m$ (with
normalization in the point $(q^{-k},\dots,q^{1-m})$ as
$$
 \sum_{\mu\in\mathbb{GT}_{m-k}} u_\mu
 \frac{s_{\mu}(x_{k+1},\dots,x_m)}{s_{\mu}(q^{-k},\dots,q^{1-m})}.
$$
Observe that $\sum_\mu u_\mu=1$. Note that, if $n$ is large enough, then coordinates of all
signatures in the support of $u_\mu$ has larger coordinates than those of $\Lambda$. Moreover,
as $n\to\infty$ these coordinates tend to infinity.

We substitute these two expansions into {\eqref{eq_x13}}
 and use Lemma \ref{lemma_factorization} yet
again (in the reverse direction, for polynomials in {$m$} variables). This gives
\begin{equation}
\label{eq_x14}
 \sum_{\lambda\in \Lambda} c_\lambda^N \sum_\mu u_{\mu}
 \frac{s_{\lambda\cup\mu}(x_1,\dots,x_m)}{s_{\lambda\cup\mu}(1,\dots,q^{1-m})} + Q_2
\end{equation}

If we forget about $Q_2$, then we arrive at a measure on $\mathbb{GT}_m$ which assigns to signature
$\lambda\cup\mu$ the weight $c_\lambda u_{\mu}$. Clearly, its projection to the lowest
{$k$}
 coordinates assigns to signature $\lambda$ the weight $c_\lambda$, as needed.

\smallskip

{ It remains to work out the impact of $Q_2$. Let us list all the terms contributing
to $Q_2$.
\begin{enumerate}
 \item $Q$ from \eqref{eq_x13}. By Lemma \ref{lemma_factorization}, $Est(Q)\to 0$ as $n\to\infty$.
 \item $Q_1$ gives the term
  \begin{multline*}
   Q_1'=Q_1  \dfrac{s_{(\nu_N^n+k,\dots,\nu_N^{k+1}+k)}(x_{k+1},\dots,x_m,q^{-m},\dots,q^{1-N})}
  {s_{(\nu_N^n+k,\dots,\nu_N^{k+1}+k)}(q^{-k},\dots,q^{1-N})}
 \\ \times
  \frac{\prod_{i=1,\dots,k}\prod_{j=k+1,\dots,m}
  (q^{1-i}-q^{1-j})}{\prod_{i=1,\dots,k}\prod_{j=k+1,\dots,m} (x_i-x_j)}.
  \end{multline*}
  We have $Est(Q_1') \le Est(Q_1) \le \varepsilon$.
 \item When we use Lemma \ref{lemma_factorization} the second time, we get a term $Q^{\lambda,\mu}$
 for each pair $(\lambda,\mu)$. Thus, the total impact is
 $$\sum_{\lambda,\mu} Q^{\lambda,\mu} u_\mu
 c_{\lambda}^N.$$
  Since by Lemma \ref{lemma_factorization} we have uniform bounds for $Est(Q^{\lambda,\mu})$
  (here we use the fact that $\Lambda$ is finite) and
  $\sum_{\lambda,\mu} |u_\mu c_{\lambda}^N|\le 1$, thus, the contribution of these terms also
  tends to zero as $n\to\infty$.
\end{enumerate}
 Summing up, we have
$$
 Est(Q_2)\le \varepsilon +R,
$$

where $R\to 0$ as $\nu^n\to\nu$ (uniformly in $N$). Consequently,
$$
 Est\Bigl(Q_2\prod_{1\le i<j\le m} (x_i-x_j)\Bigr)\le const_m(\varepsilon +R).
$$

Now let $H$ be a (signed) measure on $\mathbb{GT}_m$ with Schur--generating function $Q_2$. Using
Lemma \ref{lemma_total_variation_estimate} we conclude that total variation norm of $H$ can be
bounded by
$$
 \Vert H \Vert \le const_m (\varepsilon +R).
$$

Since $\varepsilon$ is arbitrary and $R\to 0$ as $n\to\infty$, the influence of $H$ and, thus, the
influence of $Q_2$ in \eqref{eq_x14} is negligible.

We have proved that the Schur generating function of the projection to the last $k$ coordinates of
measure on $\mathbb{GT}_m$ with Schur generating function $S_N^{\nu^n}(x_1,\dots,x_m)$, tends as
$n\to\infty$ to the function given by \eqref{eq_x2}. Now sending $N$ to infinity we see that the
Schur generating function of the projection to the last $k$ coordinates of measure $\mathcal
E^{\nu_n}_m$ uniformly tends to the Schur generating function of $\mathcal E^{\nu}_m$. Next, we
use  Proposition 4.1 of \cite{Gor} which yields that weak convergence of measures is equivalent to
uniform convergence of their Schur generating functions. We conclude that  the projection to the
last $k$ coordinates of measure $\mathcal E^{\nu^n}_m$ weakly tends to $\mathcal E^{\nu}_m$. On
the other hand, all other coordinates of the signature distributed according to $\mathcal
E^{\nu^n}_m$ tend to infinity as $n\to\infty$. It follows that $\mathcal E^{\nu^n}_m \to \mathcal
E^{\nu}_m$.

}

\medskip

For $m\le k$, Lemma \ref{lemma_factorization} yields

\begin{multline*}
  S_N^{\nu^n}(x_1,\dots,x_m)=
\dfrac{s_{(\nu_{N}^n,\dots,\nu_1^n)}(x_1,\dots,x_m,q^{-m},\dots,q^{1-N})}{s_{(\nu_{N}^n,\dots,\nu_1^n)}(1,\dots,q^{1-N})}
 \\=\dfrac{s_{(\nu_{k}^n,\dots,\nu_1^n)}(x_1,\dots,x_m,q^{-m},\dots,q^{1-k})}{s_{(\nu_{k}^n,\dots,\nu_1^n)}(1,\dots,q^{1-k})}
  \frac{\prod_{i=1,\dots,m}\prod_{j=k+1,\dots,N}
  (q^{1-i}-q^{1-j})}{\prod_{i=1,\dots,m}\prod_{j=k+1,\dots,N} (x_i-q^{1-j})}  +Q
\end{multline*}

and
\begin{multline*}
 \lim_{n\to\infty} {\mathcal S}(x_1,\dots,x_m; {\mathcal E}^{\nu^n}_{m})= \lim_{n\to\infty}
\lim_{N\to\infty}S_N^{\nu^n}(x_1,\dots,x_m)=
\\
=\dfrac{s_{(\nu_{k}^n,\dots,\nu_1^n)}(x_1,\dots,x_m,q^{-m},\dots,q^{1-k})}{s_{(\nu_{k}^n,\dots,\nu_1^n)}(1,\dots,q^{1-k})}
  \frac{\prod_{i=1}^{m}\prod_{j=k+1}^{\infty}
  (q^{1-i}-q^{1-j})}{\prod_{i=1}^{m}\prod_{j=k+1}^{\infty}(x_i-q^{1-j})}\\= {\mathcal
  S}(x_1,\dots,x_m; {\mathcal E}^{\nu}_{m}).
\end{multline*}
By Proposition 4.1 of \cite{Gor} we conclude that $\mathcal E^{\nu^n}_m \to \mathcal
E^{\nu}_m$.
\end{proof}

\begin{lemma}
\label{lemma_convergence_on_the_boundary_case2} Let $\nu^n$ be a sequence of points of
$\overline{\mathcal N}\setminus {\mathcal N}$ converging to $\nu\in \overline{\mathcal N}\setminus {\mathcal N}$.
 Then $\mathcal E^{\nu^n}$ weakly converges to $\mathcal E^\nu$.
\end{lemma}
The proof is similar to that of Lemma \ref{lemma_convergence_on_the_boundary_case1} and we
omit it.

\begin{lemma}
\label{lemma_convergence_on_the_boundary_case3} Let $\nu^n$ be a sequence of points of
$\overline{\mathcal N}\setminus\mathcal{N}$ converging to $\nu \in \mathcal{N}$. Then $\mathcal E^{\nu^n}$ weakly converges to $\mathcal E^\nu$.
\end{lemma}

\begin{proof}
Fix $m\ge 1$. Since $\nu^n\to\nu\in\mathcal{N}$, for large enough $n$ we have
$\nu^n\in\mathbb{GT}_{k_n}$ with $k_n\ge m$. Moreover, $k_n\to\infty$ as $n\to\infty$. Therefore,
$\mathcal E^{\nu^n}_m$ is supported on $\mathbb{GT}_m$, and its Schur generating function is
$$
 {\mathcal S}(x_1,\dots,x_m; {\mathcal
E}^{\nu^n}_{m})=\frac{s_{\nu^n}(x_1,\dots,x_m,q^{-m},\dots,q^{1-k_n})}{s_{\nu^n}(1,\dots,q^{1-k_n})}\prod_{q=k_n+1}^{\infty}
G_q^m.
$$
Therefore,
\begin{multline*}
\lim_{n\to\infty} {\mathcal S}(x_1,\dots,x_m; {\mathcal
E}^{\nu^n}_{m})=\lim_{n\to\infty}\frac{s_{\nu^n}(x_1,\dots,x_m,q^{-m},\dots,q^{1-k_n})}{s_{\nu^n}(1,\dots,q^{1-k_n})}\\=
{\mathcal S}(x_1,\dots,x_m; {\mathcal E}^{\nu}_{m}),
\end{multline*}
where the last equality follows from Theorem 1.3 of \cite{Gor}. Using Proposition 4.1 of
\cite{Gor} we conclude that $\mathcal E^{\nu^n}_m \to \mathcal E^{\nu}_m$.
\end{proof}

\begin{lemma}
\label{lemma_convergence_on_the_boundary_case4} Let $\nu^n$ be a sequence of points of
$\mathcal N\subset\overline{\mathcal N}$ converging to $\nu\in \mathcal N$.  Then
$\mathcal E^{\nu^n}$ weakly converges to $\mathcal E^\nu$.
\end{lemma}

\begin{proof}
See Proposition 5.16 of \cite{Gor}.
\end{proof}

\begin{lemma}
\label{lemma_bounded_first_coordinate}
 Let $\nu^n$ be a sequence of points of $\overline{\mathcal N}$ and  $\nu\in\overline{\mathcal
N}$. If $\mathcal E^{\nu^n}$ weakly converges to $\mathcal E^\nu$, then there exists $m\in \mathbb{Z}$ such that
$\nu^n_1\ge m$ for every $n\ge 1$.
\end{lemma}
\begin{proof}
We start by proving that that there exists a constant $c>0$ such that $\mathcal
E^{\nu^n}_1(\{\nu^n_1\})>c$ for every $n\ge 1$. More exactly, one can take $c=(q;q)_\infty^2$.

For $\nu\in\mathcal N$ this was proved in Lemma 5.15 of
\cite{Gor}. If $\nu^n\in\mathbb{GT}_k$, then the support of $\mathcal E^{\nu_n}_1$
consists of numbers greater or equal then $\nu^n_1$. Therefore, $x^{-\nu^n_1}{\mathcal
S}(x;\mathcal E^{\nu^n}_1)$ is a power series (without negative powers of $x$) and $
 \mathcal E^{\nu^n}_1(\{\nu^n_1\}) $ equal the value of this series at $x=0$.

We have
$$
 {\mathcal S}(x;\mathcal
E^{\nu^n}_1)=\frac{s_{\nu^n}(x,q^{-1},\dots,q^{1-k})}{s_{\nu^n}(1,q^{-1},\dots,q^{1-k})}
\prod_{\ell=k+1}^{\infty} G^1_\ell.
$$
Observe that $s_{\mu+q}(x_1,\dots,x_k)=(x_1,\dots,x_k)^{q}s_\mu(x_1,\dots,x_k)$ and use this
equality for $\mu=\nu^n$, $q=\nu^n_1$. We obtain
$$
 {\mathcal S}(x;\mathcal
E^{\nu^n}_1)=x^{\nu^n_1}\frac{s_{\lambda}(x,q^{-1},\dots,q^{1-k})}{s_{\lambda}(1,q^{-1},\dots,q^{1-k})}
\prod_{\ell=k+1}^{\infty} G^1_\ell.
$$
where $\lambda=\nu^n-\nu^n_1$ is a signature with $\lambda_k=0$. Therefore
$$
 \mathcal E^{\nu^n}_1(\{\nu^n_1\})=
\frac{s_{\lambda}(0,q^{-1},\dots,q^{1-k})}{s_{\lambda}(1,q^{-1},\dots,q^{1-k})}
\prod_{\ell=k+1}^{\infty} G^1_\ell(0).
$$

Using \eqref{eq:princ_spec} we obtain
\begin{multline*}
 \frac{s_\lambda(q^{-1},\dots,q^{1-k})}{s_\lambda(1,q^{-1},\dots,q^{1-k})}=\dfrac{q^{-|\lambda|}\prod\limits_{1\le
 i<j \le k-1} \dfrac{1-q^{-\lambda_i+\lambda_j+i-j}}{1-q^{i-j}}}{\prod\limits_{1\le i<j \le k}
 \dfrac{1-q^{-\lambda_i+\lambda_j+i-j}}{1-q^{i-j}}}
 =\dfrac{q^{-|\lambda|}}{\prod\limits_{i=1}^{k-1}\dfrac{1-q^{-\lambda_i+i-k}}{1-q^{i-k}}}\\
 =\prod_{i=1}^{k-1}\dfrac{1-q^{k-i}}{1-q^{\lambda_i-i+k}}\ge\prod_{i=1}^{k-1}(1-q^{k-i})\ge\prod_{i=1}^{\infty}{(1-q^i)}.
\end{multline*}

Also
$$
 \prod_{\ell=k+1}^{\infty} G^1_\ell(0)=(1-q^k)(1-q^{k+1})\cdots.
$$
Hence,
$$
 \mathcal E^{\nu^n}_1(\{\nu^n_1\})\ge\prod_{i=1}^{\infty}{(1-q^i)} \prod_{i=k}^{\infty}(1-q^i).
$$

Now let $C_q(u)$ be a function on $\overline{\mathbb{GT}_1}=\mathbb Z\cup \{\infty\}$ that vanishes
at all points $> q$ and equals $1$ at all points $\le q$.
Choose $q$ so that
$$
 \sum_{u\in \overline{\mathbb{GT}_1}} C_q(u) \mathcal E^{\nu}_1(u) <c/2.
$$

Note that $C_q$ is a continuous function on $\overline{\mathbb{GT}_1}$. Thus,
$$
 \sum_{u\in \overline{\mathbb{GT}_1}} C_q(u) \mathcal E^{\nu}_1(u)=\lim_{n\to\infty} \sum_{u\in
\overline{\mathbb{GT}_1}} C_q(u) \mathcal E^{\nu^n}_1(u).
$$
Hence, for large enough $n$ we have
$$
 \sum_{u\in \overline{\mathbb{GT}_1}} C_q(u) \mathcal E^{\nu^n}_1(u)<c.
$$
But if $\nu^n_1<q$ then
$$
  \sum_{u\in \overline{\mathbb{GT}_1}} C_q(u) \mathcal E^{\nu^n}_1(u)\ge  \mathcal
E^{\nu^n}_1(\{\nu^n_1\})>c.
$$
This contradiction proves that $\nu^n_1 \ge q$.
\end{proof}

Now we are ready to give a proof of Theorem \ref{theorem_equivalence_of_topologies}.

\begin{proof}[Proof of Theorem \ref{theorem_equivalence_of_topologies}]

Let $\{\nu^n\}_{n\ge 1}\subset\overline{\mathcal N}$ and $\nu\in\overline{\mathcal N}$. Our goal
is to prove that $\lim_{n\to\infty}\nu^n=\nu$ if and only if $\mathcal E^{\nu^n}$ weakly converges
to $\mathcal E^\nu$ as $n\to\infty$.

First, suppose that $\nu^n\to\nu$. Without loss of generality we may assume that either for every
$n$ we have $\nu^n\in\mathcal N$, or for every $n$, $\nu^n\in\overline{\mathcal N}\setminus
\mathcal N$. Also either $\nu\in{\mathcal N}$ or $\nu\in\overline{\mathcal N}\setminus\mathcal N$.
Thus, we have four cases and they are covered by Lemmas
\ref{lemma_convergence_on_the_boundary_case1}-\ref{lemma_convergence_on_the_boundary_case4}.

Now suppose that $\mathcal E^{\nu^n}$ weakly converges to $\mathcal E^\nu$. By Lemma
\ref{lemma_bounded_first_coordinate}, there exists $m\in \mathbb{Z}$ such that $\nu^n\in D_m$ for
every $n\ge 1$, where
$$
 D_m=\{\nu\in \overline{\mathcal N}: \nu_1\ge m\}.
$$
Observe that the set $D_m$ is compact. Therefore, the sequence $\nu^n$ has a converging
subsequence $\nu^{n_h}\to\nu'$. But then $\mathcal E^{\nu^n}\to\mathcal E^{\nu'}$ and, thus,
$\nu=\nu'$. We see that $\{\nu^n\}$ is a sequence in a compact set such that all its converging
subsequences converge to $\nu$. This implies $\nu^n\to\nu$ as $n\to\infty$.
\end{proof}
\subsection{Feller Markov processes on the boundary}

 In this section we prove that for any admissible function $g(x)$,
$\overline{P_\infty}(u,A; g)$ is a Feller kernel on $\overline{\mathcal N}$. Moreover, we show
that a Markov process on $\overline{\mathcal N}$ with semigroup of transition probabilities
$\overline{P_\infty}(u,A; \exp(t(\gamma_+x+\gamma_-/x)))$ and arbitrary initial distribution is
Feller.

The section is organized as follows: First we prove in Proposition
\ref{proposition_Feller_kernel_N} that for any $N\ge 0$ and any admissible $g(x)$,
$\overline{P_N}(\lambda\to\mu;g)$ is a Feller kernel on $\overline{\mathbb{GT}_N}$. As a
corollary, we show in Theorem \ref{theorem_Feller_kernel} that $\overline{P_\infty}(u,A;g)$ is a
Feller kernel on $\overline{\mathcal N}$. Finally, we prove (Proposition
\ref{proposition_Feller_process_N} and Theorem \ref{theorem_Feller_process}) that Markov processes
with semigroups of transition probabilities $\overline{P_N}(\mu\to\lambda;
\exp(t(\gamma_+x+\gamma_-/x)))$ and $\overline{P_\infty}(u,A; \exp(t(\gamma_+x+\gamma_-/x)))$ are
Feller.

\begin{proposition}
 \label{proposition_Feller_kernel_N} For any admissible function $g(x)$,
$\overline{P_N}(\mu\to\lambda;g)$ is a Feller kernel on $\overline{\mathbb{GT}_N}$.
\end{proposition}

First, we prove a technical lemma.

\begin{lemma}
\label{lemma_exponential_estimate} For any elementary admissible function $g(x)$, the transition
probabilities $P_N(\mu\to\lambda ;g(x))$ admit an exponential tail estimate: There exist positive
constants $a_1$ and $a_2$ such that
$$
 P_N(\mu\to\lambda ;g(x)) < a_1 \exp\bigl( - a_2 \max_{1\le i\le N} |\lambda_i-\mu_i|).
$$
\end{lemma}
\begin{proof}
 Proposition \ref{Proposition_transition_prob_explicit} gives
$$
P_N(\mu\to\lambda ;g(x))=\left(\prod_{i=1}^N \frac{1}{g(q^{1-i})}\right)
  \det_{i,j=1,\dots,N}\biggl[c_{\lambda_i-i-\mu_j+j}\biggr]
  \frac{s_\lambda(1,\dots,q^{1-N})}{s_\mu(1,\dots,q^{1-N})},
$$
where
$$
 g(x)=\sum_{k\in\mathbb{Z}} c_k x^k.
$$

 If $g(x)=(1+\beta x^{\pm 1})$, then $P_N(\mu\to\lambda ;g(x))=0$ as soon as
$|\lambda_i-\mu_i|>1$ for any $i$, and we are done.

For the remaining two cases note that \eqref{eq:princ_spec} implies
$$
\frac{s_\lambda(1,\dots,q^{1-N})}{s_\mu(1,\dots,q^{1-N})}< const_N\cdot
q^{\sum_i(\lambda_i-\mu_i)(i-N)}< const_N\cdot q^{-N^2 \max_i|\lambda_i-\mu_i|}.
$$

It follows that

\begin{multline*}
 \left| \det_{i,j=1,\dots,N}\biggl[c_{\lambda_i-i-\mu_j+j}\biggr] \frac{s_\lambda(1,\dots,q^{1-N})}{s_\mu(1,\dots,q^{1-N})} \right|
\\ < const_N  \sum_{\sigma\in S(N)} q^{-N^2 \max_i|\lambda_i-\mu_i|} \prod_{i=1}^N c_{\lambda_i-i-\mu_{\sigma(i)} +
\sigma(i)}.
\end{multline*}

If $g(x)=\exp(\gamma x^{\pm 1})$  then $c_k r^k \to 0$ for any $r>0$ as $k\to\infty$.
Note that $\max_i |\lambda_i-\mu_i|=m$ implies that for any permutation $\sigma$ there exists $i$
such that $|\lambda_i-\mu_{\sigma(i)}| \ge m$. For this $i$ the product $$q^{-N^2
\max_i|\lambda_i-\mu_i|} c_{\lambda_i-i-\mu_{\sigma(i)} + \sigma(i)}$$ is  exponentially
small (in $m$). Therefore, each term
$$
 q^{-N^2 \max_i|\lambda_i-\mu_i|} \prod_{i=1}^N c_{\lambda_i-i-\mu_{\sigma(i)} + \sigma(i)}
$$
 tends to zero exponentially fast as $\max_i |\lambda_i-\mu_i|\to \infty$ and we are done.

Finally, if $g(x)=(1-\alpha x^{-1})^{-1}$ then $P_N(\mu\to\lambda ;g(x))=0$ unless $\mu_i\ge
\lambda_i$ for $1\le i\le N$. In the latter case,
$$\frac{s_\lambda(1,\dots,q^{1-N})}{s_\mu(1,\dots,q^{1-N})}<const.
$$
Expanding the determinant
$$
 \det_{i,j=1,\dots,N}\biggl[c_{\lambda_i-i-\mu_j^n+j}\biggr]= \sum_{\sigma} \prod_{i=1}^N c_{\lambda_i-i-\mu_{\sigma(i)} + \sigma(i)},
$$
by the same argument as above we see that each term in the sum tends to zero exponentially
fast as $\max_i |\lambda_i-\mu_i|\to \infty$. Therefore,
$$
\left(\prod_{i=1}^N \frac{1}{g(q^{1-i})}\right)
  \det_{i,j=1,\dots,N}\biggl[c_{\lambda_i-i-\mu_j+j}\biggr]
  \frac{s_\lambda(1,\dots,q^{1-N})}{s_\mu(1,\dots,q^{1-N})}
$$
tends to zero exponentially fast.
\end{proof}

\begin{proof}[Proof of Proposition \ref{proposition_Feller_kernel_N}]

Let $h\in C_0(\overline{\mathbb{GT}_N})$; we want to check that $P^*_{N}(g)(h)\in
C_0(\overline{\mathbb{GT}_N})$. By the definition of admissible functions, it suffices to check
this property for elementary admissible functions $g(x)=(1+\beta x^{\pm 1})$, $g(x)=\exp(\gamma
x^{\pm 1})$ and $g(x)=(1-\alpha x^{-1})^{-1}$. Moreover, since $C_0(\overline{\mathbb{GT}_N})$ is
closed and $P^*_{N}(g)$ is a contraction, it is enough to check this property on a set of
functions $h$ whose linear span is dense.

We choose the following system of functions ($\lambda\in\mathbb{GT}_N$, $k\ge 1$):
$$
 a_{\lambda}(\mu)=\begin{cases}1,\quad \mu=\lambda,\\ 0,\text{ otherwise}; \end{cases}
$$
$$
 b_{\lambda,k}(\mu)=\begin{cases} 1,\quad \mu_1\ge\lambda_1,\dots,\mu_k\ge \lambda_k,\,
\mu_{k+1}=\lambda_{k+1},\dots,\mu_N= \lambda_N,\\ 0,\quad\text{ otherwise}.\end{cases}
$$

Let us start from $h=a_{\lambda}(\mu)$. We have
$$
 \Bigl(\overline{P^*_{N}}(g)(h)\Bigr) (\mu) = \overline{P_N}(\mu\to\lambda;g).
$$

By definition, $\overline{P_N}(\mu\to\lambda;g)=0$ if $\mu\in\overline{\mathbb{GT}_N}\setminus
\mathbb{GT}_N$. Note that for a sequence $\mu^n$ of elements of $\mathbb{GT}_N$, $\mu^n \to
\infty$ (in topology of $\overline{\mathbb{GT}_N}$) means that $\mu_N^n\to-\infty$, while $\mu^n
\to \mu\in {\overline{\mathbb{GT}_N}}\setminus {\mathbb{GT}_N} $ implies $\mu_1^n \to +\infty$.
Thus, to show that $\overline{P^*_{N}}(g)(h)\in C_0(\overline{\mathbb{GT}_N})$ we should prove
that if $\mu^n$ is a sequence of elements of $\mathbb{GT}_N$ such that
 either $\mu_N^n\to-\infty$ or $\mu_1^n\to +\infty$ as $n\to\infty$, then
$\overline{P_N}(\mu^n\to\lambda;g)={P_N}(\mu^n\to\lambda;g)\to 0$ as $n\to\infty$. But if
$\mu_N^n\to-\infty$  then $|\mu_N^n-\lambda_N|\to\infty$ and we may use Lemma
\ref{lemma_exponential_estimate}. If $\mu_1^n\to +\infty$, then $|\mu_1^n-\lambda_1|\to\infty$ and
Lemma \ref{lemma_exponential_estimate} also provides the required estimate.

\smallskip

Now let $h=b_{\lambda,k}(\mu)$. The fact that $\Bigl(\overline{P^*_{N}}(g)(h)\Bigr) (\mu)\to 0$ as
$\mu\to\infty$ (in other words, as $\mu_N\to -\infty$) again follows from Lemma
\ref{lemma_exponential_estimate}. However, we need an additional argument to prove that the
function $\overline{P^*_{N}}(g)(h)$ is continuous on $\overline{\mathbb{GT}_N}$.

Let us prove that if  $\mu^n$ is a sequence of elements of $\overline{\mathbb{GT}_N}$ converging
to $\mu$, then $\Bigl(\overline{P^*_{N}}(g)(h)\Bigr) (\mu^n)\to \Bigl(\overline{P^*_{N}}(g)(h)\Bigr) (\mu)$.
One readily sees that there are two principal cases (the result for all other cases is
a simple corollary of the results in one of these two cases):
\begin{itemize}
\item $\mu^n\in \mathbb{GT}_N\subset\overline{\mathbb{GT}_N}$ and $\mu\in
\mathbb{GT}_N\subset\overline{\mathbb{GT}_N}$;
\item $\mu^n\in \mathbb{GT}_N\subset\overline{\mathbb{GT}_N}$ and $\mu\in
\overline{\mathbb{GT}_N}\setminus \mathbb{GT}_N$, i.e.\ $\mu_1=\mu_2=\dots=\mu_m=\infty$ while
$\mu_{m+1}<\infty$ for some $1\le m\le N$.
\end{itemize}
In the former case the sequence $\mu^n$ stabilizes, i.e.\ $\mu^n=\mu$ for large enough $n$,
therefore, $\Bigl(\overline{P^*_{N}}(g)(h)\Bigr) (\mu^n)= \Bigl(\overline{P^*_{N}}(g)(h)\Bigr) (\mu)$ for large enough
$n$. In the rest of the proof we concentrate on the latter case. From the definition of
convergence in $\overline{\mathbb{GT}_N}$ we conclude that $\mu^n_i\to +\infty$ for $i=1,\dots m$,
and for $n>n_0$ we have $\mu^n_i=\mu_i$ for $i=m+1,\dots,N$. Without loss of generality
we may assume $n_0=0$.

If $m>k$ then by Lemma \ref{lemma_exponential_estimate} $$\Bigl(\overline{P^*_{N}}(g)(h)\Bigr)
(\mu^n)\to 0 = \Bigl(\overline{P^*_{N}}(g)(h)\Bigr) (\mu).$$ Now suppose that $m=k$. We have:

\begin{equation}
\label{eq_x16}
 \Bigl(\overline{P^*_{N}}(g)(h)\Bigr) (\mu^n)=\sum_{\nu\in B_{\lambda,k}} P_N(\mu^n\to\nu;g),
\end{equation}
where $B_{\lambda,k}=\{u\in \mathbb{GT}_N: b_{\lambda,k}(u)=1\}$.

Note that in $\mu^n$ the first $k$ coordinates are large while the last $N-k$ coordinates are
fixed, and choose an integral sequence $r^n$ such that $r^n\to +\infty$ and $\mu^n_k-r^n\to\infty$.
Let
$$
 \mathcal A^n=\{\nu\in \mathbb{GT}_N \mid \nu_1 \ge r^n,\dots, \nu_k \ge r^n\}.
$$
 For every $n$ we divide the set $B_{\lambda,k}$ in the sum \eqref{eq_x16} into two disjoint
parts
$$
 \sum_{\nu\in B_{\lambda,k}} P_N(\mu^n\to\nu;g) =\sum_{\nu\in \mathcal B^n_{\lambda,k}} P_N(\mu^n\to\nu;g)+\sum_{\nu\in \mathbb B^n_{\lambda,k}} P_N(\mu^n\to\nu;g)
$$
where
$$
 \mathcal B^n_{\lambda,k} = B_{\lambda,k} \cap \mathcal A^n, \quad \mathbb B^n_{\lambda,k}=  B_{\lambda,k} \setminus \mathcal
 A^n.
$$

Observe that, as follows from Lemma \ref{lemma_exponential_estimate},
$$
\sum_{\nu\in \mathbb B_{\lambda,k}^n}
 P_N(\mu^n\to\nu;g) \to 0.
$$

In the remaining sum
\begin{equation} \label{eq_x9}
 \sum_{\nu\in \mathcal B_{\lambda,k}^n} P_N(\mu^n\to\nu;g),
\end{equation}
we use Lemma \ref{lemma_factorization} for every term, i.e.\ for
$$
P_N(\mu^n\to\nu ;g(x))=\left(\prod_{i=1}^N \frac{1}{g(q^{1-i})}\right)
  \det_{i,j=1,\dots,N}\biggl[c_{\nu_i-i-\mu^n_j+j}\biggr]
  \frac{s_{\nu}(1,\dots,q^{1-N})}{s_{\mu^n}(1,\dots,q^{1-N})}.
$$
As we have shown in the proof of Lemma \ref{lemma_exponential_estimate} for all
elementary admissible functions, the coefficients $c_k$ decay rapidly as $k$ grows. Therefore, the
determinant of the matrix $\biggl[c_{\mu^n_i-i-\nu_j+j}\biggr]$ factorizes. We conclude that
\begin{align}
\label{eq_x8}\nonumber& P_N(\mu^n\to\nu ;g(x))=(1+o(1))\\\nonumber& \times \left(\prod_{i=1}^{N-k}
\frac{1}{g(q^{1-i})}\right)
  \det_{i,j=k+1,\dots,N}\biggl[c_{\nu_i -i-\mu^n_j+j}\biggr]
  \frac{s_{(\nu_{k+1},\dots,\nu_N)}(1,\dots,q^{1-N+k})}{s_{(\mu^n_{k+1},\dots,\mu^n_N)}(1,\dots,q^{1-N+k})}
\\& \times \left(\prod_{i={N-k+1}}^{N} \frac{1}{g(q^{1-i})}\right)
  \det_{i,j=1,\dots,k}\biggl[c_{\nu_i -i-\mu^n_j+j}\biggr]
  \frac{s_{(\nu_{1}+N-k,\dots,\nu_k+N-k)}(q^{k-N},\dots,q^{1-N})}{s_{(\mu^n_1+N-k,\dots,\mu^n_k+N-k)}(q^{k-N},\dots,q^{1-N})},
\end{align}
where the term $o(1)$ uniformly tends to zero as $n\to\infty$.

Note that the second line in \eqref{eq_x8} is the transition probability {
$$
 P_N\Bigl(\mu\to(\infty,\dots,\infty,\lambda_{k+1},\dots,\lambda_N); g(x)\Bigr) = \bigl(\overline{P^*_{N}}(g)(h)\bigr) (\mu).
$$
 }

 As for the third line, observe that Lemma \ref{Lemma_decomposition_of_multiplied_s} implies
\begin{multline*}
  \sum_{(\nu_1,\dots,\nu_k)\in\mathbb{GT}_k} \left(\prod_{i={N-k+1}}^{N}
  \frac{1}{g(q^{1-i})}\right) \det_{i,j=1,\dots,k}\biggl[c_{\nu_i -i-\mu^n_j+j}\biggr] \\ \times
  \frac{s_{(\nu_{1}+N-k,\dots,\nu_N+N-k)}(q^{k-N},\dots,q^{1-N})}{s_{(\mu^n_1+N-k,\dots,\mu^n_k+N-k)}(q^{k-N},\dots,q^{1-N})}=1.
\end{multline*}

Arguing as in Lemma \ref{lemma_exponential_estimate}, we conclude that we may replace the
summation set by $\mathcal B_{\lambda,k}^n$, i.e.~as $n\to\infty$,
\begin{multline*}
  \sum_{\nu\in\mathcal B_{\lambda,k}^n} \left(\prod_{i={N-k+1}}^{N}
  \frac{1}{g(q^{1-i})}\right) \det_{i,j=1,\dots,k}\biggl[c_{\nu_i -i-\mu^n_j+j}\biggr] \\ \times
  \frac{s_{(\nu_{1}+N-k,\dots,\nu_N+N-k)}(q^{k-N},\dots,q^{1-N})}{s_{(\mu^n_1+N-k,\dots,\mu^n_k+N-k)}(q^{k-N},\dots,q^{1-N})}\to 1.
\end{multline*}

Summing up, we proved that
\begin{multline*}
 \bigl(\overline{P^*_{N}}(g)(h)\bigr) (\mu^n)= \sum_{\nu\in \mathcal B^n_{\lambda,k}} P_N(\mu^n\to\nu;g)+\sum_{\nu\in \mathbb B^n_{\lambda,k}} P_N(\mu^n\to\nu;g)
 \\ =  \bigl(\overline{P^*_{N}}(g)(h)\bigr) (\mu) (1+o(1)) +o(1) \to \bigl(\overline{P^*_{N}}(g)(h)\bigr) (\mu).
\end{multline*}

It remains to consider the case $m<k$. The statement in this case essentially follows from the
case $m=k$. Indeed, decompose $B_{\lambda,k}$ into a disjoint union
$$
 B_{\lambda,k}=\bigcup_{\theta \in \Theta}  B_{\theta,m},
$$
where the union is taken over the set $\Theta\subset \mathbb{GT}_N$ consisting of all $\theta$ such
that $\theta_i=\lambda_i$ for $i>k$, $\theta_i\ge\lambda_i$ for $m<i\le k$ and
$\theta_i=\max(\lambda_i,\theta_{m+1})$ for $i\le m$. Now choose large enough $s$ and denote
$$
 \Theta^s=\Theta \cap \{\nu\in\mathbb{GT}_N \mid \nu_1 <s\}.
$$
  We write
$$
 \bigl(\overline{P^*_{N}}(g)(b_{\lambda,k}\bigr)(\mu^n)=
 \sum_{\theta \in \Theta^s} \bigl(\overline{P^*_{N}}(g)(b_{\theta,m})\bigr) (\mu^n)  +
 \sum_{\theta \in \Theta \setminus \Theta^s}  \bigl(\overline{P^*_{N}}(g)(b_{\theta,m})\bigr) (\mu^n),
$$
Note that the set $\Theta^s$ is finite. Therefore, using the already proven case $m=k$ we conclude
that
$$
\sum_{\theta \in \Theta^s} \bigl(\overline{P^*_{N}}(g)(b_{\theta,m})\bigr) (\mu^n) \to \sum_{\theta \in
\Theta^s} \bigl(\overline{P^*_{N}}(g)(b_{\theta,m})\bigr) (\mu).
$$
To finish the proof it remains to note that as follows from Lemma
\ref{lemma_exponential_estimate},
$$\sum_{\theta \in \Theta \setminus \Theta^s}  \bigl(\overline{P^*_{N}}(g)(b_{\theta,m})\bigr) (\mu^n)$$
uniformly (in $n$) tends to zero as $s\to\infty$.
\end{proof}

Let $P^*_{\infty\to N}$ be a contraction operator from $\mathcal B(\overline{\mathbb{GT}_N})$ to
$\mathcal B(\overline{\mathcal N})$ given by:
$$
 (P^*_{\infty\to N} f) (\nu)=\sum_{v\in\overline{\mathbb{GT}_N}} \mathcal E^{\nu}_N(v) f(v),
$$
where $\mathcal E^{\nu}$ is the extreme coherent system corresponding to the measure $\delta^\nu$
on $\overline{\mathcal N}$.
\begin{lemma}
\label{lemma_link_is_Feller}
 $P^*_{\infty\to N}$ maps $C_0(\overline{\mathbb{GT}_N})$ to $C_0(\overline{\mathcal N})$.
\end{lemma}
\begin{proof}
 Again it suffices to check the lemma on a set of continuous functions whose linear span is dense. We choose the
familiar system of functions
$$
 a_{\lambda}(\mu)=\begin{cases}1,\quad \mu=\lambda,\\ 0,\text{ otherwise,} \end{cases}
$$
$$
 b_{\lambda,k}(\mu)=\begin{cases} 1,& \mu_1\ge\lambda_1,\dots,\mu_k\ge\lambda_k,\,
\mu_{k+1}=\lambda_{k+1},\dots,\mu_N= \lambda_N,\\ 0,&\text{ otherwise},\end{cases}
$$
and their finite linear combinations. Theorem
\ref{theorem_equivalence_of_topologies} implies that $P^*_{\infty\to N}(a_\lambda)$ and $P^*_{\infty\to
N}(b_{\lambda,k})$ are continuous functions on $\overline{\mathcal N}$. It remains to check that they vanish at the infinity.
If $\nu^n$ is a sequence of elements of $\overline{\mathcal N}$ tending to infinity
(i.e.\ escaping from every compact set), then $\nu^n_1\to -\infty$. Thus, we should check that if
$\nu^n_1\to-\infty$ then
$$
 \bigl(P^*_{\infty\to N}(a_{\lambda})\bigr)(\nu^n)=\mathcal E^{\nu^n}_N(\lambda)\to 0.
$$
Assume the opposite. Then there exists a subsequence $\{n_\ell\}$ such that $\mathcal
E^{\nu^n}_N(\lambda)>c>0$ for $n=n_\ell$. Let $\psi^\ell=\nu^{n_\ell}-\nu^{n_\ell}_1$.
Then $\{\psi^{\ell}\}$ is a sequence of elements of a compact set. Hence, $\{\psi^{\ell}\}$ has a converging
subsequence. Without loss of generality assume that already $\psi^\ell$ is converging,
$\psi^{\ell}\to \psi$. Since, $\mathcal E^{\psi^{\ell}}_N$ is a probability measure on
$\mathbb{GT}_N$ and $\mathcal E^{\psi^{\ell}}_N\to\mathcal E^{\psi}_N$, we must have $\mathcal
E^{\psi^{\ell}}_N(\lambda-\nu^{n_\ell}_1) \to 0$.

Now observe the following property of measures $\mathcal E^\nu$ which was proved in
\cite[Proposition 5.12]{Gor}. For $e\in\mathbb{Z}$ and $\nu\in\mathcal N$ let $\nu-e$ be a
sequence with coordinates $(\nu-e)_i=\nu_i-e_i$. In the same way for $\lambda\in\mathbb{GT}_N$ set
$(\lambda-e)_i=\lambda_i-e$. Then we have  $\mathcal E^{\nu-e}_N(\lambda-e)=\mathcal
E^\nu_N(\lambda)$.

 We conclude that $\mathcal E^{\nu^{n_{\ell}}}(\lambda)=\mathcal
E^{\psi^{\ell}}_N(\lambda-\nu^{n_\ell}_1)\to 0$. Contradiction.

The argument for the functions $b_{\lambda,k}$ is similar and we omit it.
\end{proof}

\begin{theorem}
 \label{theorem_Feller_kernel} For an admissible $g(x)$, $\overline{P_{\infty}}(g)$ is a Feller
kernel on $\overline{\mathcal N}$.
\end{theorem}
\begin{proof}
This is an application of Proposition \ref{proposition_feller}. Indeed, $\overline
{\mathbb{GT}_N}$ and $\overline {\mathcal N}$ are locally compact with countable bases by
definition. The property $2.$ of Proposition \ref{proposition_feller} is Lemma
\ref{lemma_link_is_Feller} and the property $3.$ is Proposition \ref{proposition_Feller_kernel_N}.
\end{proof}

\begin{proposition}
\label{proposition_Feller_process_N}
 A Markov process on $\overline{\mathbb{GT}_N}$ with semigroup of transition probabilities
$\overline{P_N}(\lambda\to\mu; exp(t(\gamma_+x+\gamma_-/x)))$ and arbitrary initial distribution
is Feller.
\end{proposition}
\begin{proof}
 The first property from the definition of a Feller process is contained in Proposition
\ref{proposition_Feller_kernel_N}. As for the second property, it immediately follows from the
fact that $\exp(t(\gamma_+x+\gamma_-/x)))\to 1$ as $t\to 0$ and definitions.
\end{proof}

Now Proposition \ref{proposition_Feller_process} yields
\begin{theorem}
\label{theorem_Feller_process}
 A Markov process on $\overline{\mathcal N}$ with semigroup of transition probabilities
$\overline{P_\infty}(u,A; \exp(t(\gamma_+x+\gamma_-/x)))$ and arbitrary initial distribution is
Feller.
\end{theorem}


\section{PushASEP with particle-dependent jump rates}

\label{Section_pushASEP}

Fix parameters $\zeta_1,\dots,\zeta_N>0$, $a, b\ge 0$, and assume that at least one the numbers
$a$ and $b$ does not vanish.  Consider $N$ particles in $\mathbb Z$ located at different sites and
enumerated from left to right. The particle number $n$, $1\le n\le N$, has two exponential clocks
- the ``right clock'' of rate $a\zeta_n$ and the ``left clock'' of rate $b/\zeta_n$. When the
right clock of particle number $n$ rings, it checks whether the position to the right of it is
empty. If yes, then the particle jumps to the right by 1, otherwise it stays put. When the left clock of particle number $n$ rings, it jumps
to the left by 1 and pushes the (maybe empty) block of particles sitting next to it.

Note that if $\zeta_1=q^{1-N}$, $\zeta_2=q^{2-N}$, \dots, $\zeta_{N-1}=q^{-1}$, $\zeta_N=1$, and the
process is started from the initial configuration $1-N,2-N,\dots,-1,0$ then this dynamics describes the
evolution of $N$ leftmost particles of the process $\Pi(\mathcal Y^S_{\gamma_+,\gamma_-}(t))$
introduced in Section \ref{Subsection_first_coordinate}.

Let $P_t(x_1,\dots,x_N\mid y_1\dots, y_N)$ denote the transition probabilities of the above
process. The probabilities depend on $\zeta_1,\dots,\zeta_N,a,b$, but we omit these
dependencies from the notation. In this section we study the asymptotic behavior of particles as
time $t$ goes to infinity. In other words, we are interested in the asymptotics of
$P_t(x_1,\dots,x_N\mid y_1\dots, y_N)$ as $t\to\infty$.

One easily checks that if for some $r\in\mathbb{R}_{>0}$ and $1\le k\le N$ we have
$\zeta_k<r$ and $\xi_{k+1}>r$, \dots, $\zeta_N>r$, then at large times the first $k$ and the last
$N-k$ particles behave independently. Thus, it is enough to study the case $\xi_N \le
\min(\zeta_1,\dots,\zeta_{N-1})$. Moreover, without loss of generality we may also assume that
$\zeta_N=1$. Thus, it suffices to consider the situation when we have
$h\le N$ indices $n_1<n_2<\dots<n_h=N$ such that
$$
 \zeta_{n_1}=\zeta_{n_2}=\dots=\zeta_{n_{h}}=1,
$$
and $\zeta_k>1$ for $k$ not belonging to $\{n_1,\dots,n_{h}\}$. Set
$D=\{n_i\}_{i=1,\dots,h}$.

To state the result on asymptotic behavior we need to introduce a certain distribution from the
random matrix theory first. Let $M_n$ be a random $n\times n$ Hermitian matrix from the Gaussian
Unitary Ensemble, see e.g. \cite{Meh}, \cite{F}, \cite{AGZ} for the definition. {We use the
normalization for which diagonal matrix elements are real Gaussian random variables with variance
$1$.  Denote the eigenvalues of $M_n$ by $\lambda_1^n\le \lambda_2^n\le \dots\le \lambda_n^n$. Let
$M_k$ be the top left $k\times k$ submatrix of $M_n$, denote the eigenvalues of $M_k$ by
$\lambda_1^k\le\dots\le\lambda_k^k$. Finally, denote by $GUE_1^n$ the joint distribution of the
smallest eigenvalues of matrices $M_k$. In other words, it is the joint distribution of the vector
$(\lambda_1^n\le\lambda_1^{n-1}\le\dots\le\lambda_1^1)$.

Denote by ${\mathfrak G}_n(z)$, $n\in\mathbb Z$, the integral
$$
 {\mathfrak G}_n(z)=\frac 1{2\pi i} \int_{\Co_2} u^n \exp( u^2/2 + uz) du,
$$
where the contour of integration $\Co_2$ is shown in Figure \ref{Figure_Contours_3}.
\begin{figure}[h]
\begin{center}
\noindent{\scalebox{0.7}{\includegraphics{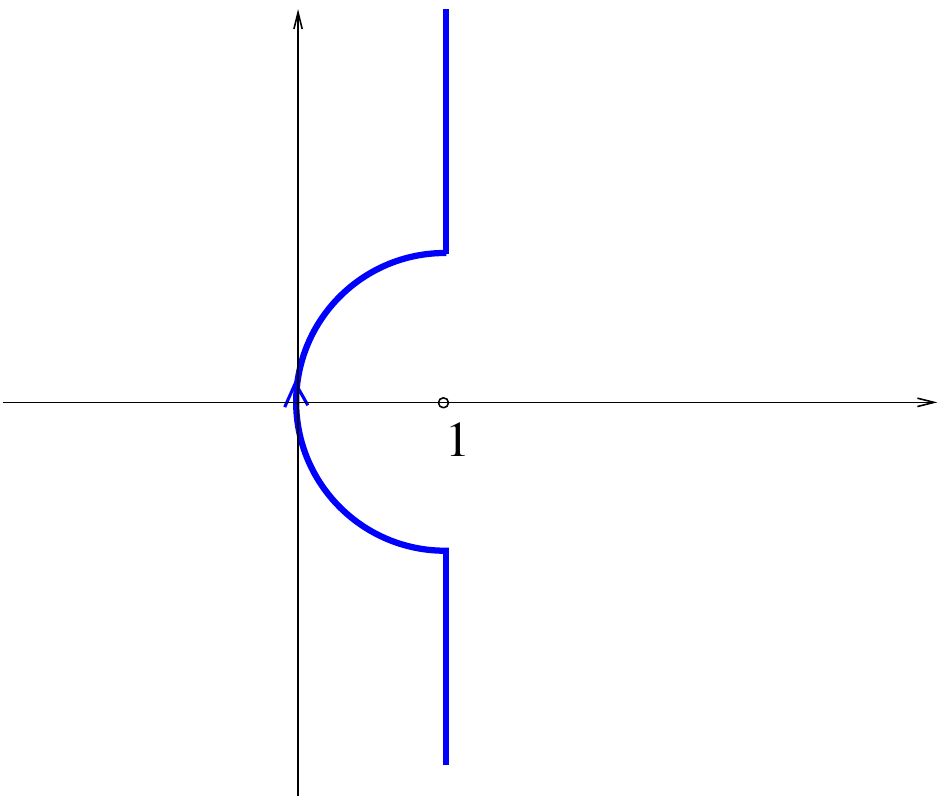}}} \caption{Contour of integration
$\Co_2$. \label{Figure_Contours_3} }
\end{center}
\end{figure}
Note that ${\mathfrak G}_0(z)$ is the density of the Gaussian distribution:
$$
 \frac 1{2\pi i} \int_{\Co_2} \exp( u^2/2 + uz) du = \frac{1}{\sqrt{2\pi}}\exp(-z^2/2).
$$
{
 More generally, for $n\ge 1$ we have:
$$
 {\mathfrak G}_n(z)=\left(\frac{\partial}{\partial z}\right)^n {\mathfrak G}_0(z) = (-1)^n
 \frac{1}{\sqrt{\pi}\, 2^{(n+1)/2}} H_n\left(\frac{z}{\sqrt 2}\right) \exp\left(-\frac{z^2}{2}\right),
$$
 where $H_n$ is the $n$th Hermite polynomial.
}

\begin{proposition}
\label{Proposition_GUE_distribution}
 The probability distribution $GUE_1^n$ of the smallest eigenvalues of the $GUE$ principal submatrices has density
$$
 \rho(y_1,\dots,y_n)=\det[{\mathfrak G}_{k-l}(y_l)]_{k,l=1,\dots,n},\quad y_1\le y_2\le\dots\le y_n.
$$
\end{proposition}

\begin{proof}
 We start from the formulas for the correlation functions of \emph{all} eigenvalues of GUE minors
found in {\cite{J_artic}, \cite{OR_birth},} \cite{JN}, \cite{JNe}.  Given a
$n\times n$ GUE matrix, one constructs a point process on $\mathbb N \times \mathbb R$ which has a
point $(m,y)$ if and only if $y$ is an eigenvalue of $m\times m$ top left submatrix. Let $\rho_n$
be the $n$th correlation function of this point process. Observe that the interlacing of the
eigenvalues of the nested submatrices guarantees that
$$
 \rho_n(1,y_n;2,y_{n-1};\dots ; n,y_1),\quad y_n\ge y_{n-1}\ge\dots\ge y_1,
$$
is precisely the density of the distribution $GUE_1^n$. Then {Theorem 2 in
\cite{OR_birth} and} Theorem 1.3 of \cite{JN} yield that the density of $GUE_1^n$ is
$$
 \det_{i,j=1\dots n}[ K(i,y_{n+1-i};j,y_{n+1-j})],
$$
where (we use a more convenient for us integral representation for the kernel, which can be found
after the formula $(6.17)$ in \cite{JN}; note that we use a different
normalization for GUE)
\begin{multline*}
 K(r,\xi; s,\eta)= \frac{\sqrt{2^{s-r}\exp(\frac{\eta^2-\xi^2}2)}}{2\sqrt2(\pi i)^2}\\ \times
\sum_{k=0}^{\infty} \int_{\Co_1} du \exp(\sqrt2\xi u-u^2)u^{k-r} \int_{\Co_2} dv
v^{s-k-1}\exp(v^2-\sqrt 2\eta v),
\end{multline*}
where $\Co_1$ is anticlockwise oriented circle around the origin  and $\Co_2$ is the line $\Re
v=2$ from $-i\infty$ to $+i\infty$. Changing variables $z=-\sqrt2 v$ and $w=-\sqrt 2 u$ we arrive
at
$$
(-1)^{s-r-1}\frac{\exp(\frac{\eta^2-\xi^2}2)}{(2\pi i)} \sum_{k=0}^{\infty} \int_{\Co_3} dw
\exp(-\xi w-w^2/2)w^{k-r} {\mathfrak G}_{s-k-1}(\eta),
$$
where $\Co_3$ is clockwise oriented circle around origin. Note that when $k\ge n\ge r$ the
integral $\int_{\Co_3} dw \exp(-\xi w-w^2/2)w^{k-r}$ vanishes. Thus, we may assume that the sum is
over $k=0,1,\dots,n-1$.

Hence, $K(r,\xi; s,\eta)$ is the matrix element of the product of two matrices and
$$
 \det_{i,j=1\dots n}[ K(i,y_{n+1-i};j,y_{n+1-j})]=(-1)^n\det_{r,k=1,\dots,n} [A(r,k)]
\det_{k,s=1,\dots,n} [{\mathfrak G}_{s-k}(y_{n+1-s})],
$$
where
$$
 A(r,k)=\frac{1}{2\pi i} \int_{\Co_3} dw \exp(y_{n+1-r} w-w^2/2)w^{k+1-r}.
$$
Again note that $A(r,k)=0$ unless $r>k$. Thus, the matrix $[A(r,k)]$ is triangular and
$$
 \det_{r,k=1,\dots,n} [A(r,k)]=\prod_{k=1}^n A(k,k)
$$
with
$$
A(k,k)=\frac{1}{2\pi i} \int_{\Co_3} \frac{dw}{w} \exp(y_{n+1-r} w-w^2/2)=-1.
$$
To finish the proof it remains to note that
$$
\det_{k,s=1,\dots,n}[{\mathfrak G}_{s-k}(y_{n+1-s})]=\det_{k,s=1,\dots,n} [{\mathfrak
G}_{k-s}(y_{s})].
$$
\end{proof}

{Now we are ready to state the main theorem of this section.}

\begin{theorem}\label{th:pushasep}
Fix parameters $\bigl(\{\zeta_i\}_{i=1}^n,a,b\bigr)$ as above and let  $y_1<\dots< y_N$ be arbitrary integers.
Denote by $(X_1(t),\dots,X_N(t))$ the random vector distributed as
$$
 P_t(x_1,\dots,x_N\mid y_1,\dots,y_N).
$$

The joint distribution of
\begin{align*}
 \biggl(&\frac{X_{n_1}(t)-(a-b)t}{\sqrt{(a+b)t}}, \frac{X_{n_2}(t)-(a-b)t}{\sqrt{(a+b)t}},\dots,
\frac{X_{n_h}(t)-(a-b)t}{\sqrt{(a+b)t}};\\
 &X_{n_1}(t)-X_{n_1-1}(t),\dots, X_{2}(t)-X_{1}(t); \\&X_{n_2}(t)-X_{n_2-1}(t),\dots,
 X_{n_1+2}(t)-X_{n_1+1}(t); \\&\dots
 \\&X_{n_h}(t)-X_{n_h-1}(t),\dots, X_{n_{h-1}+2}(t)-X_{n_{h-1}+1}(t) \biggr)
\end{align*}

converges to
\begin{align*}
 &GUE_1^h
 \\&\times {\rm Ge}(\zeta_{n_1-1}^{-1})\times\dots\times {\rm Ge}(\zeta_1^{-1})
 \\&\dots
 \\&
 \times {\rm Ge}(\zeta_{n_h-1}^{-1})\times\dots\times {\rm Ge}(\zeta_{n_{h-1}+1}^{-1}),
\end{align*}
 where ${\rm Ge}(p)$ is a geometric distribution on $\{1,2,\dots\}$ with parameter $p$.
\end{theorem}

Theorem \ref{th:pushasep} follows from Proposition
\ref{Proposition_GUE_distribution} and Proposition \ref{Proposition_Convergence_in_PushASEP}.

\begin{proposition}
\label{Proposition_Convergence_in_PushASEP} Set $v=a-b$ and denote
\begin{gather*}
x_1=vt+\sqrt{(a+b)t} \cdot\widetilde x_1 +\hat x_1,\dots,  x_{n_1}=vt+\sqrt{(a+b)t} \cdot\widetilde x_1 +\hat x_{n_1},\\
x_{n_1+1}=vt+\sqrt{(a+b)t} \cdot\widetilde x_2 +\hat x_{n_1+1},\dots, x_{n_2}=vt+\sqrt{(a+b)t} \cdot\widetilde x_2 +\hat x_{n_2},\\
\dots\\
x_{n_{h-1}+1}=vt+\sqrt{(a+b)t}\cdot \widetilde x_h +\hat x_{n_{h-1}+1},\dots, x_{n_h}=vt+\sqrt{(a+b)t}\cdot \widetilde x_{h} +\hat x_{n_h}.
\end{gather*}
Then
\begin{multline*}
\lim_{t\to\infty} ((a+b)t)^{h/2} P_t(x_1,\dots,x_N\mid y_1,\dots,y_N) \\=
\prod_{i\in\{1,\dots,N\}\setminus D} (1-\zeta^{-1}_i) \zeta_i^{1-\hat x_{i+1}+\hat x_i} \cdot
\det[{\mathfrak G}_{k-l}(\widetilde
 x_l)]_{k,l=1,\dots,h},
\end{multline*}
and the convergence is uniform for $\widetilde x_1,\dots,\widetilde x_h$ belonging to compact
sets.
\end{proposition}

\begin{proof}

Our starting point is an explicit formula for $P_t(x_1,\dots,x_N\mid y_1,\dots,y_N)$ from
\cite{BF-Push}; it is a generalization of a similar formula for TASEP from \cite{RS}. We have
\begin{multline}
\label{eq_PushASEP_transition} P_t(x_1,\dots,x_N\mid y_1,\dots,y_N)\\=\left(\prod_{i=1}^N
\nu_i^{x_i-y_i}\exp(-at\zeta_i-bt/\zeta_i)\right) \det[F_{k,l}(x_l-y_k)],
\end{multline}
where
$$
 F_{k,l}(x)=\frac1{2\pi i} \oint_{\Co_0} \dfrac{\prod_{i=1}^{k-1}(1-\zeta_i
 z)}{\prod_{j=1}^{l-1}(1-\zeta_j z)}\, z^{x-1} \exp(btz+atz^{-1}) dz,
$$
and the integration is over a positively oriented circle $\Co_0$ of a small radius centered at
$0$.

Let us study the asymptotic behavior of $F_{k,l}(x)$ when $t\to\infty$ and $x=x(t)=(a-b)t+
\sqrt{(a+b)t}\cdot  \tilde x + \hat x$.

\begin{lemma}
\label{lemma_asymptotic_of_kernel} For $t\to\infty$ with $x=x(t)=(a-b)t+ \sqrt{(a+b)t} \cdot
\tilde x + \hat x$, we have
\begin{align}
\label{eq_x4}
 \nonumber F_{k,l}(x)&= \exp \left((a+b)t\right) ((a+b)t)^{-1/2} t^{-n/2} (-1)^{n}\\ &\quad
\nonumber\times \dfrac{\prod_{i\in\{1,\dots, k-1\}\setminus D} (1-\zeta_i)}{\prod_{j\in\{1,\dots,
l-1\}\setminus D} (1-\zeta_j)} \left({\mathfrak G}_n(\tilde x) +o(1)\right)
 \\&+ \sum_{m\in\{k\dots l-1\}\setminus D } C_{k,l,m}(t,x) \zeta_m^{1-x} \exp(bt\zeta_m^{-1}+at\zeta_m),
\end{align}
 where {the summation is only over
indices $m$ corresponding to distinct $\zeta_m$},
$$
 C_{k,l,m}(t,x) = Res_{z=\zeta_m^{-1}}\left( \dfrac{\prod_{i=1}^{k-1}(1-\zeta_i
 z)}{\prod_{j=1}^{l-1}(1-\zeta_j z)} (\zeta_m z)^{x-1}
 \exp(bt(z-\zeta_m^{-1})+at(z^{-1}-\zeta_m))\right)
$$
{(in particular, $C_{k,l,m}(t,x)$ has polynomial growth (or decay) in $t$ as
$t\to\infty$)} and
$$n=|\{1,\dots, k-1\}\cap D|-|\{1,\dots, l-1\}\cap D|.$$
The remainder $o(1)$ above is uniformly small for $\tilde x$ belonging to compact sets.
\end{lemma}
{{\bf Remark. }
 Observe that in the limit regime of Lemma \ref{lemma_asymptotic_of_kernel} we have
 \begin{multline*}
  |\zeta_m^{-x} \exp(bt/\zeta_m+at\zeta_m)|=\exp\left(t\left(-\frac{x}{t}\ln(\zeta_m)  + b/\zeta_m+a\zeta_m\right)\right)
 \\ \approx \exp\biggl(t((b-a)\ln(\zeta_m)  + b/\zeta_m+a\zeta_m)\biggr).
 \end{multline*}
 Since the function  $(b-a)\ln(r)  + b/r+ar$ of $r\in \mathbb{R_+}$ has a minimum at $r=1$,
 for any $\zeta_m\ne 1$ we have (for $t\gg 1$) $$ |\zeta_m^{1-x} \exp(bt/\zeta_m+ar)| \gg
\exp((a+b)t).$$
 Using the fact that $C_{k,l,m}(t,x)$ have polynomial growth (or decay) in $t$ as $t\to\infty$ we conclude that
 in the asymptotic decomposition \eqref{eq_x4} the first term is small comparing to the other ones. }

\begin{proof}[Proof of Lemma \ref{lemma_asymptotic_of_kernel}]
{First, suppose that $b>0$.} Let us deform the integration contour so that it
passes near the point $1$. Since $\zeta_j\ge 1$, we need to add some residues:
\begin{multline}
\label{eq_x3}
 F_{k,l}(x)=\frac1{2\pi i} \oint_{\Co_1} \dfrac{\prod_{i=1}^{k-1}(1-\zeta_i
z)}{\prod_{j=1}^{l-1}(1-\zeta_j z)}\, z^{x-1} \exp(btz+atz^{-1}) dz \\+\sum_{m\in \{k,\dots, l-1\}\setminus
D} Res_{z=\zeta_m^{-1}}\left( \dfrac{\prod_{i=1}^{k-1}(1-\zeta_i
 z)}{\prod_{j=1}^{l-1}(1-\zeta_j z)}\, z^{x-1} \exp(btz+atz^{-1})\right),
\end{multline}
where the contour $\Co_1$ is shown in Figure \ref{Figure_Contours_2}. {Here the summation is only
over indices $m$ corresponding to distinct $\zeta_m$.}

\begin{figure}[h]
\begin{center}
\noindent{\scalebox{0.7}{\includegraphics{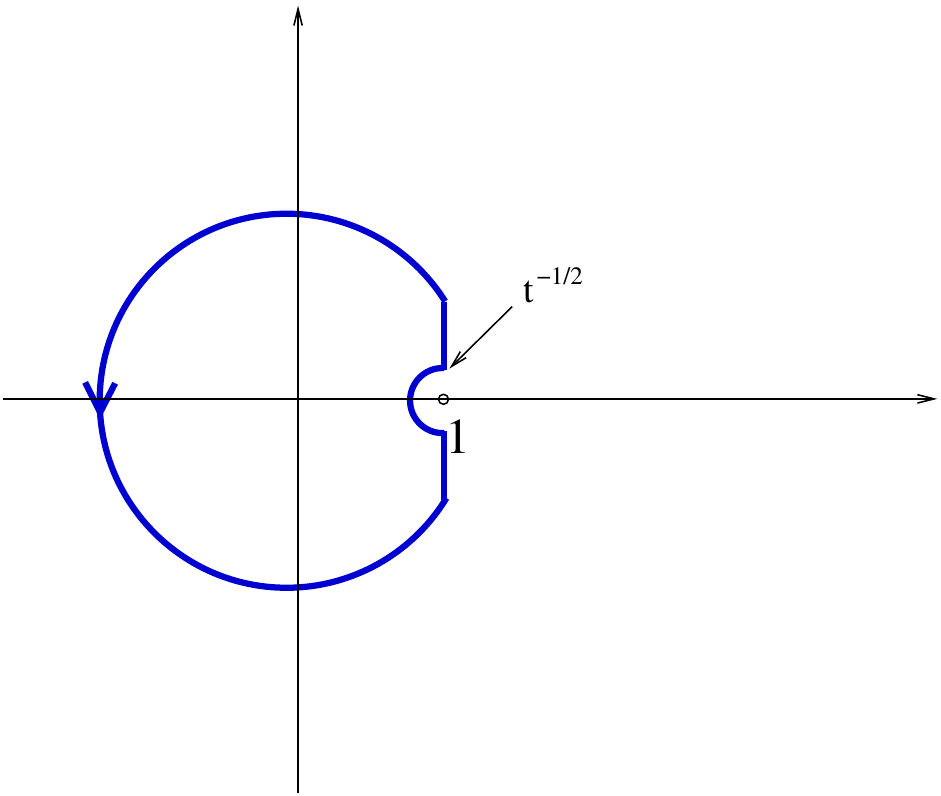}}} \caption{Contour of integration
$\Co_1$. {The radius of the small arc is $t^{-1/2}$ and the radius of larger arc is a parameter
$R$ that we choose later}. \label{Figure_Contours_2} }
\end{center}
\end{figure}

Observe that if $l\le m< k$ and $\zeta_m>1$, then
\begin{multline*}
Res_{z=\zeta_m^{-1}}\left( \dfrac{\prod_{i=1}^{k-1}(1-\zeta_i z)}{\prod_{j=1}^{l-1}(1-\zeta_j z)}\,
 z^{x-1} \exp(btz+atz^{-1})\right)\\ =C_{k,l,m}(t,x)  \zeta_m^{1-x} \exp(bt\zeta_m^{-1}+at\zeta_m).
\end{multline*}
In particular, if the pole at $\zeta_m^{-1}$ is simple (which is true, for instance, if all
$\zeta_m\ne 1$ are mutually distinct), then
\begin{multline*}
Res_{z=\zeta_m^{-1}}\left( \dfrac{\prod_{i=1}^{k-1}(1-\zeta_i z)}{\prod_{j=1}^{l-1}(1-\zeta_j z)}
\, z^{x-1} \exp(btz+at z^{-1})\right)\\=  \dfrac{-\zeta_m^{-1}}{\prod_{j=\{k,\dots, l-1\}\setminus
D}(1-\zeta_j/\zeta_m)}\, \zeta_m^{1-x} \exp(bt \zeta_m^{-1}+at\zeta_m).
\end{multline*}

Observe that
$$
 z^{x-1} \exp(btz+atz^{-1})= \exp(t (\ln z (a-b+ t^{-1/2}\tilde x \sqrt{a+b} +t^{-1}(\hat x -1)) + bz
+a/z))
$$
and
$$
 \Re ((a-b) \ln z  + bz +a/z)
$$
has a saddle point at $z=1$. { We claim that for large $t$, only a small neighborhood of
$z=1$ gives a non-negligible contribution the integral. Indeed, let $z=u+iv$. Then
\begin{equation}
 \Re ((a-b) \ln z  + bz +a/z) = (a-b) \ln\sqrt{u^2+v^2} +bu +\frac{au}{\sqrt{u^2+v^2}}.
\end{equation}
Setting $u=1$ and differentiating with respect to $v$ we see that for a small enough
$\varepsilon>0$ the function
$$
 (a-b) \ln\sqrt{1+v^2} +b +\frac{a}{\sqrt{1+v^2}}
$$
increases on $[-\varepsilon,0]$, decreases on $[0,\varepsilon]$ and its maximum
is $a+b$.

Now set the radius of the bigger arc $R$ to be equal to $\sqrt{1+\varepsilon^2}$ and note that
along this arc
\begin{multline*}
 (a-b) \ln\sqrt{u^2+v^2} +bu +\frac{au}{\sqrt{u^2+v^2}}=(a-b)\ln R+bu+\frac{au}{R}\\<(a-b)\ln
 R+b+\frac{a}R= (a-b) \ln\sqrt{1+\varepsilon^2} +b +\frac{a}{\sqrt{1+\varepsilon^2}}.
\end{multline*}
Summing up, if we choose any $\delta>0$, then everywhere outside the $\delta$-neighborhood of
$z=1$ the function  $f(z)=\Re ((a-b) \ln z  + bz +a/z)$ is smaller than $f(1+\delta)$ which, in
turn, is smaller than $f(1)=a+b$. Therefore, the integral along our contour outside the
$\delta$-neighborhood of $z=1$ can be bounded by $const\cdot \exp((a+b)t) \exp
(t(f(1)-f(1+\delta)))$ and, thus, as $t\to\infty$ it becomes exponentially smaller than
$\exp((a+b)t)$, which is the smallest term in our asymptotic expansion \eqref{eq_x4}.
{Consequently, this integral contributes only to $o(1)$ term in \eqref{eq_x4} and can be
neglected.}

To calculate the integral in the small $\delta$-neighborhood of $1$ we change the integration
variable $z=1+t^{-1/2}u$ and arrive at the integral
\begin{multline*}
 t^{-1/2}\frac1{2\pi i}\int_{\Co_2'}  \dfrac{\prod_{i=1}^{k-1}(1-\zeta_i
(1+t^{-1/2}u))}{\prod_{j=1}^{l-1}(1-\zeta_j (1+t^{-1/2}u))}\\
\times \exp\biggl(t \bigr(\ln(1+t^{-1/2}u) (a-b+ t^{-1/2}\tilde x \sqrt{a+b} +t^{-1}(\hat x -1))\\ +
b(1+t^{-1/2}u) +a/(1+t^{-1/2}u)\bigr)\biggr) du
\end{multline*}
with contour $\Co_2'$ that is a part of contour $\Co_2$ shown in Figure \ref{Figure_Contours_3}
 between points $u=\pm i \delta t^{1/2}$

}

Simplifying the integrand we arrive at
\begin{multline*}
 t^{-1/2}\exp \bigl(t(a+b)\bigr) \frac1{2\pi i}\int_{\Co_2'}
\dfrac{\prod_{i=1}^{k-1}(1-\zeta_i (1+t^{-1/2}u))}{\prod_{j=1}^{l-1}(1-\zeta_j (1+t^{-1/2}u))}
\\ \times \exp \bigl( u^2 (a+b)/2 + u\tilde x \sqrt{a+b}+ o(1)\bigr).
\end{multline*}
{Making the linear change of variables $v=u\sqrt{a+b}$ and sending $t$ to $\infty$
we obtain the required integral.}

{ As for the case $b=0$, the argument is similar and we omit it. The only
difference is that now contour of integration should have no vertical line part, i.e.\ it consists
of the arcs of unit circle and $t^{-1/2}$-circle.}
\end{proof}

Now we continue the proof of Proposition \ref{Proposition_Convergence_in_PushASEP}.

The next step is to do certain elementary transformations to the  matrix $[F_{k,l}(x_l-y_k)]$ in
order to simplify its determinant. First, suppose that all $\zeta_m$ (except for those equal to 1) are distinct.

We take the $(N-1)$st row of matrix $[F_{k,l}(x_l-y_k)]_{k,l=1}^N$ (i.e.~$k=N-1$), multiply it by
\begin{equation}
 \zeta_m^{y_n-y_{N-1}} \dfrac{1}{\prod_{j=n}^{N-2}(1-\zeta_j/\zeta_{N-1})}\,,
\end{equation} and add
to the $n$th row for $n=1,\dots, N-2$. As a result the term in \eqref{eq_x4} coming from the
residue at $\zeta_{N-1}$ remains only in the $(N-1,N)$ matrix element.

Next, we take the $(N-2)$nd row of the matrix and add this row (again with coefficients) to rows
$1,\dots,N-3$ so that the term in \eqref{eq_x4} coming from the residue at $\zeta_{N-2}$ remains
only in the $(N-2,N-1)$ and $(N-2,N)$ matrix elements.

Repeating this procedure for every row $h$ such that $\zeta_h\ne 1$ we get a transformed matrix
$[\widehat F_{k,l}(x_l-y_k)]_{k,l=1}^N$. In this new matrix, the term in \eqref{eq_x4} coming from the
residue at $\zeta_h$ remains only in the matrix elements $\widehat F_{h-1,h}(x_h-y_h-1)$, \dots,
$\widehat F_{h-1,N}(x_N-y_{h-1})$.

Let us now do some columns transformations. We take the second column ($l=2$) and add it
with coefficients to the columns $3,\dots,N$ so that the terms coming from the residue at
$\zeta_1$ vanish in columns $3,\dots,N$. Then we repeat this procedure for the third column and
so on.

Finally, we get a matrix $[\widetilde F_{k,l}(x_l-y_k)]_{k,l=1}^N$ whose determinant coincides with that of
$[F_{k,l}(x_l-y_k)]_{k,l=1}^N$, but this matrix has only $N-h$ elements with order of growth greater than
$\exp((a+b)t)$ (i.e.\ elements with terms coming from the residues). These elements are $\widetilde
F_{n-1,n}(x_{n}-y_{n-1})$ for $n\in\{2,\dots,N\}$, $n\notin D+1$,
$$
 \widetilde F_{n-1,n}(x_{n}-y_{n-1})=-\zeta_{n-1}^{-1} \zeta_{n-1}^{1-x_{n}+y_{n-1}}
\exp(bt \zeta_{n-1}^{-1}+at\zeta_{n-1}) (1+o(1)).
$$

It follows that asymptotically as $t\to\infty$ the determinant $\det[\widetilde F_{k,l}(x_l-y_k)]$
factorizes:
\begin{multline}
 \label{eq_x6} \det[\widetilde F_{k,l}(x_l-y_k)]_{k,l=1,\dots,N}\\=\prod_{n\in\{2,\dots,
N\}\setminus(D+1)} \zeta_{n-1}^{-x_{n}+y_{n-1}} \exp(bt \zeta_m^{-1}+at\zeta_m)
\det[G_{k,l}]_{k,l=1,\dots,h} (1+o(1)),
\end{multline}
where $[G_{k,l}]_{k,l=1,\dots,h}$ is the submatrix of $[\widetilde F_{k,l}(x_l-y_k)]$ with rows
$n_1,n_2,\dots,n_h$ and columns $1, n_1+2, n_2+1,\dots,n_{h-1}+1$. Looking back at row and column
operations that we made with $[F_{k,l}(x_l-y_k)]$ to get $[\widetilde F_{k,l}(x_l-y_k)]$, we
see that the determinant of $G_{k,l}$ coincides with the determinant of the Gaussian terms of the
decomposition \eqref{eq_x4} of matrix elements $F_{k,l}(x_l-y_k)$ in rows $n_1,n_2,\dots,n_h$ and
columns $1, n_1+1, n_2+1,\dots,n_{h-1}+1$. Therefore,
\begin{multline}
\label{eq_x7} \det[G_{k,l}]_{k,l=1,\dots,h}=\exp(ht(a+b)) (-1)^h((a+b)t)^{-h/2}\biggl(o(1)\\
+ \prod_{m=1}^{h} \prod_{i\in\{1,\dots, n_m-1\}\setminus D} (1-\zeta_i) \prod_{m=1}^{h-1}
\prod_{j\in\{1,\dots, n_m\}\setminus D}\dfrac{1}{(1-\zeta_i)} \,\det[{\mathfrak
G}_{k-l}(\widetilde x_l)]_{k,l=1,\dots,h}\biggr)
\\= \exp(ht(a+b)) (-1)^h((a+b)t)^{-h/2}  \prod_{i\in\{1,\dots, N-1\}\setminus D}
(1-\zeta_i)\\ \times \left(\det[{\mathfrak G}_{k-l}(\widetilde x_l)]_{k,l=1,\dots,h} +o(1)\right).
\end{multline}
Combining
\eqref{eq_PushASEP_transition} with asymptotic formulas \eqref{eq_x6} and \eqref{eq_x7} we
conclude the proof.

\smallskip
{
Now suppose that some of $\zeta_i$ coincide. In this case the computation of the residues in the
decomposition \eqref{eq_x4} becomes more complicated, since some of the poles are not simple.
However, the scheme of the proof remains the same: First we transform the matrix
$[F_{k,l}(x_l-y_k)]$ by means of elementary row transforms so that the terms in \eqref{eq_x4}
coming from the residue at $\zeta_h$ remain only in the matrix elements $\widehat
F_{h-1,h}(x_h-y_h-1)$, \dots, $\widehat F_{h-1,N}(x_N-y_{h-1})$. Then we do elementary column
transforms so that  only $N-h$ elements with order of growth greater than $\exp((a+b)t)$ remain in
the resulting matrix. After that the computation of the determinant repeats the case of distinct
$\zeta_m$.}

\end{proof}

\subsection*{Acknowledgements}
The authors would like to thank the anonymous referee for his valuable suggestions which improved
the text. A.B. was partially supported by NSF grant DMS-1056390.
 V.G.\ was partially supported by ``Dynasty'' foundation,  by RFBR --- CNRS grant 10-01-93114, by
the program ``Development of the scientific potential of the higher school'' and by Simons
Foundation - IUM scholarship.


\begin{thebibliography}{99}

\bibitem{AGZ} G. W. Anderson, A. Guionnet, and O. Zeitouni, An introduction to random matrices,
Cambridge University Press, 2010.

\bibitem{Bar} Yu. Baryshnikov, GUEs and queues, Probab. Theory Relat. Fields 119
(2001), 256--274

\bibitem{B-Schur} A.~Borodin, Schur dynamics of the schur processes, Adv. Math.
228 (2011), no,\ 4, 2268--2291. arXiv:1001.3442.

\bibitem{BF-Push} A.~Borodin, P.~Ferrari, Large time asymptotics of growth models on space-like paths I:
PushASEP, Electron. J. Probab. 13 (2008), 1380--1418. arXiv:0707.2813.

\bibitem{BF} A.~Borodin, P.~Ferrari, Anisotropic growth of random
surfaces in 2 + 1 dimensions. arXiv:0804.3035.

\bibitem{BGR} A.~Borodin, V.~Gorin, E.~M.~Rains, $q$-Distributions on boxed plane partitions, Selecta
Mathematica, New Series, 16:4 (2010), 731--789. arXiv:0905.0679.

\bibitem{BK} A.~Borodin, J.~Kuan, Asymptotics of Plancherel measures for the infinite-dimensional unitary group, Adv. Math.
219:3 (2008), 894--931. arXiv:0712.1848.

\bibitem{BO-RSK} A.~Borodin, G.~Olshanski, Z-Measures on partitions,
 Robinson-Schensted-Knuth correspondence, and $\beta=2$ random matrix ensembles, Random matrix models and their applications
  (P.M.Bleher and R.A.Its, eds), Math. Sci. Res. Inst. Publ., vol. 40, Cambridge Univ. Press, Cambridge, 2001,
  71--94. arXiv:math/9905189.

\bibitem{BO} A.~Borodin, G.~Olshanski, Markov processes on the path space of the Gelfand-Tsetlin graph and on its
boundary. arXiv:1009.2029.

\bibitem{Bo} R.~P.~Boyer, Infinite Traces of AF-algebras and characters of $U(\infty)$,
J. Operator Theory 9 (1983), 205--236.

\bibitem{CK} R. Cerf and R. Kenyon, The low
temperature expansion of the Wulff crystal in the 3D Ising model, Comm. Math. Phys. 222:1 (2001),
147--179.

\bibitem{DF} P.~Diaconis, D.~Freedman, Partial Exchangeability and Sufficiency.
 Proc. Indian Stat. Inst. Golden Jubilee Int'l Conf. Stat.: Applications and New Directions,
  J. K. Ghosh and J. Roy (eds.), Indian Statistical Institute, Calcutta, pp.\ 205-236.

\bibitem{Ed} A.~Edrei, On the generating function of a doubly--infinite, totally positive
sequence, Trans. Amer. Math. Soc. 74 (3) (1953), 367--383.

\bibitem{EK} S.~N.~Ethier, T.~G.~Kurtz, Markov processes --- Characteriztion and convergence.
Wiley-Interscience, New-York 1986.

\bibitem{F} P.~J.~Forrester, Log-gases and random matrices, Princeton University Press 2010.

\bibitem{Gor_Hahn} V.~Gorin, Non-intersecting paths and Hahn orthogonal polynomial ensemble,
Funct. Anal. Appl.,  42:3 (2008), 180--197. arXiv: 0708.2349.

\bibitem{Gor} V.~Gorin, The $q$--Gelfand--Tsetlin graph, Gibbs measures and $q$--Toeplitz
matrices.  Adv. Math., 229 (2012), no.\ 1, 201--266, arXiv:1011.1769


\bibitem{J_annals}  K. Johansson, Discrete orthogonal polynomial ensembles and the Plancherel
measure, Ann. Math.,(2), 153:1 (2001), 259--296. arXiv:math/9906120.

\bibitem{J_nonintersecting} K.~Johansson,  Non-intersecting Paths, Random Tilings and
Random Matrices.  Probab. Theory and Related Fields, 123:2 (2002), 225--280. arXiv:math/0011250.

\bibitem{J_Hahn} K.~Johansson, Non-intersecting, simple, symmetric random walks
and the extended Hahn kernel. Ann. Inst. Fourier (Grenoble) 55:6 (2005),  2129--2145.
  arXiv:math.PR/0409013.

\bibitem{J_artic}  K. Johansson, The arctic circle boundary and the Airy process, Ann. Probab. 33:1 (2005),
1--30. arXiv:math/030621.

\bibitem{JN} K.~Johansson, E.~Nordenstam, Eigenvalues of GUE Minors.  Electronic
Journal of Probability,  11 (2006), paper 50, 1342--1371. arXiv:math/0606760.

\bibitem{JNe} K.~Johansson, E.~Nordenstam, Erratum to Eigenvalues of GUE minors.  Electronic
Journal of Probability,  12 (2007), paper 37, 1048--1051.

\bibitem{Liza_Jones} Liza Anne Jones, Non-Colliding Diffusions and Infinite Particle Systems,
Thesis, University of Oxford, 2008.

\bibitem{KM} S.~P.~Karlin and G.~MacGregor, Coincidence probabilities, Pacif. J. Math. 9 (1959),
1141--1164.

\bibitem{KT09} M.~Katori and H.~Tanemura, Zeros of Airy function and relaxation process. J. Stat.
Phys. 136 (2009) 1177--1204. arXiv:0906.3666.

\bibitem{KT10} M.~Katori and H.~Tanemura, Non-equilibrium dynamics of Dyson's model with an infnite
number of particles . Commun. Math. Phys. 293 (2010), 469--497. arXiv:0812.4108.

\bibitem{Kerov_book} S.~Kerov, Asymptotic Representation Theory of the Symmetric Group and its Applications in Analysis
    AMS, Translations of Mathematical Monographs, v. 219 (2003).

\bibitem{KOR} W.~Konig, N.~O'Connell, S.~Roch, Non-colliding random walks,
tandem queues and discrete orthogonal polynomial ensembles. Electronic Journal of Probability,
vol.7 (2002), paper no. 1, 1--24.

\bibitem{L1} T.~Liggett, Interacting Particle Systems, Springer-Verlag, New York, 1985.

\bibitem{L2} T.~Liggett,  Stochastic Interacting Systems: Contact, Voter and Exclusion
Processes. Grundlehren der mathematischen Wissenschaften, volume 324, Springer, 1999.

 \bibitem{Mac} I.~Macdonald, Symmetric Functions and Hall Polinomials, Clarendon Press Oxford, 1979.

\bibitem{Meh} M.~L.~Mehta, Random matrices, 2nd ed. Academic Press, Boston, 1991.

\bibitem{Ok_wedge} A.~Okounkov, Inifinite wedge and random partitions.
 Selecta Mathematica, New Series Volume 7:1 (2001), 57--81. arXiv:math/9907127.

\bibitem{OR}  A.~Okounkov, N.~Reshetikhin, Correlation functions of Schur process
with application to local geometry of a random 3-dimensional Young diagram. J. Amer. Math. Soc. 16
(2003), 581--603. arXiv: math.CO/0107056

\bibitem{OkOlsh} A.~Okounkov, G.~Olshansky, Asymptotics of Jack Polynomials as the
Number of Variables Goes to Infinity, International Mathematics Research Notices 13 (1998),
641--682. arXiv:q-alg/9709011.

\bibitem{OR_birth} A. Yu. Okounkov, N. Yu. Reshetikhin, The birth
of a random matrix, Mosc. Math. J., 6:3 (2006), 553--566

\bibitem{Olsh} G.~Olshanski, The problem of harmonic analysis
 on the infinite-dimensional unitary group, Journal of Functional Analysis, 205 (2003), 464--524.
 arXiv:math/0109193.

\bibitem{Olsh2} G. Olshanski, Laguerre and Meixner symmetric functions, and infnite-dimensional diffusion processes.
Journal of Mathematical Sciences, 174:1, 41-57. Translated from Zapiski Nauchnykh Seminarov POMI,
378 (2010), 81--110. arXiv:1009.2037.

\bibitem{Os} H.~Osada, Interacting Brownian motions in infnite dimensions with logarithmic inter-
action potentials. arXiv:0902.3561.

\bibitem{RS} A.R\'{a}kos, G.Sch\"{u}tz, Bethe Ansatz and current distribution for the TASEP with particle-dependent
hopping rates, Markov Process. Related Fields 12 (2006), 323--334. arXiv:cond-mat/0506525.

\bibitem{Spitzer} F. Spitzer, Interaction of Markov processes. Adv. Math., 5 (1970), 246--290.

\bibitem{S} H.~Spohn, Interacting Brownian particles: a study of Dyson's model. In: Hydrodynamic
Behavior and Interacting Particle Systems, Papanicolaou, G. (ed), IMA Volumes in Mathematics and
its Applications, 9, Berlin: Springer-Verlag, 1987, 151--179.

\bibitem{Vo}
 D.~Voiculescu, Repr\'esentations factorielles de type \rm{II}${}_1$ de $U(\infty)$, J. Math.
 Pures et Appl. 55 (1976), 1--20.

\bibitem{V_Vienna} A. Vershik. The generating function $\prod_{k=1}^{\infty} (1-x^k)^{-k}$ --- MacMahon and Erd\"{o}s.
Talk at the 9-th International Conference on Formal Power Series and Algebraic Combinatorics,
Vienna, 1997.

\bibitem{V_ergodic} A.~M.~Vershik, Description of invariant  measures for the actions of some
infinite--dimensional groups, Sov. Math. Dokl. 15 (1974), 1396--1400.


\bibitem{Vk_unitary} A.~M.~Vershik, S.~,V.~Kerov, Characters and factor representations of the
inifinite unitary group, Sov. Math. Dokl. 26 (1982), 570--574.

\bibitem{Wey} H. Weyl, The classical groups. Their invariants and representations. Princeton Univ.
Press, 1939; 1997 (fifth edition).

\end{thebibliography}
\end{document}